\documentclass[10pt, a4paper]{amsart}

\usepackage[T1]{fontenc}

\usepackage{microtype}

\usepackage[abbrev]{amsrefs}

\newtheorem{theorem}{Theorem}
\newtheorem{proposition}{Proposition}[section]
\newtheorem{lemma}[proposition]{Lemma}
\newtheorem{corollary}[proposition]{Corollary}
\theoremstyle{definition}

\theoremstyle{remark}
\newtheorem{remark}{Remark}

\newcommand{\eps}{\varepsilon}
\newcommand{\Eps}{{\boldsymbol{\varepsilon}}}

\newcommand{\R}{\mathbf{R}}
\newcommand{\C}{\mathbf{C}}
\newcommand{\N}{\mathbf{N}}

\newcommand{\Peps}{(\mathcal{P}^\eps)}
\newcommand{\Neps}{\mathcal{N}^\eps}
\newcommand{\Qeps}{(\mathcal{Q}^\eps)}
\newcommand{\Meps}{\mathcal{M}^\eps}

\newcommand{\abs}[1]{\lvert#1\rvert}
\newcommand{\dualprod}[2]{\langle #1 , #2 \rangle}
\newcommand{\st}{\: :\:}
\newcommand{\charfun}[1]{\chi_{#1}}
\newcommand{\restrictedto}[1]{\vert_{#1}}
\newcommand{\weakto}{\rightharpoonup}
\newcommand{\muleb}[1]{\mathcal{L}^{#1}}

\newcommand{\Norm}[1]{\Vert#1\Vert}
\DeclareMathOperator{\supp}{supp}
\DeclareMathOperator{\diam}{diam}
\DeclareMathOperator{\capa}{cap}
\DeclareMathOperator{\dist}{dist}

\newcommand\logeps{{\vert\!\log\eps\vert}}
\newcommand\logEps{{\vert\!\log\vert {\boldsymbol\eps }\vert\vert}}

\title{Desingularization of vortices for the Euler equation}

\author{Didier Smets}
\address{Laboratoire Jacques-Louis Lions\\ Universit\'e Pierre et Marie Curie\\ 4 place Jussieu BC 187\\ 75252 Paris\\ France}
\email{smets@ann.jussieu.fr}
\author{Jean Van Schaftingen}
\thanks{The second author was partially supported by the Fonds de la Recherche Scientifique-FNRS (Belgium) and by the Fonds Sp\'eciaux de Recherche (Universit\'e catholique de Louvain)}
\address{Universit\'e catholique de Louvain\\
D\'epartment de Math\'ematique\\
Chemin du Cyclotron 2\\
1348 Louvain-la-Neuve\\
Belgium}
\email{Jean.VanSchaftingen@uclouvain.be}
\date{\today}



\begin{document}
\begin{abstract}
We study the existence of stationary classical solutions of the incompressible Euler equation in the plane that approximate singular stationary solutions of this equation.
The construction is performed by studying the asymptotics of equation $-\eps^2 \Delta u^\eps=(u^\eps-q-\frac{\kappa}{2\pi} \log \frac{1}{\eps})_+^p$ with Dirichlet boundary conditions and $q$ a given function.
We also study the desingularization of pairs of vortices by minimal energy nodal solutions and the desingularization of rotating vortices.
\end{abstract}

\maketitle


\section{Introduction}

\subsection{Singular solutions to the Euler equation}

The incompressible Euler equations
\[
\left\{
  \begin{aligned}
    \nabla \cdot \mathbf{v} &= 0, \\
    \mathbf{v}_t + \mathbf{v}\cdot \nabla \mathbf{v}&=-\nabla p,
  \end{aligned}
\right. 
\]
describe the evolution of the velocity $\mathbf{v}$ and the pressure $p$ in an incompressible flow. In $\R^2$, the vorticity $\omega = \nabla \times \mathbf{v}=\partial_1 \mathbf{v}_2-\partial_2 \mathbf{v}_1$ of a solution of the Euler equations obey the transport equation
\[
\omega_t + \mathbf{v} \cdot \nabla \omega = 0 
\]
and the velocity field $\mathbf{v}$ can be recovered from the vorticity function $\omega$ through the Biot--Savart law
\[
\mathbf{v} = \omega * \frac{1}{2\pi} \frac{-x^\perp}{\abs{x}^2},
\]
where $x^\perp=(x_2, -x_1)$. 
Special singular solutions of the Euler equations are given by
\footnote{One needs to give a meaning to the equation in this case, since the velocity field generated by a vortex point is singular precisely on that vortex point. It consists in considering that each vortex point is transported only by the velocity field created by the other vortex points (see e.g. S.\thinspace Schochet \protect{\cite{Schochet_CPDE_95}} for details and further discussion). }
\[
\omega = \sum_{i=1}^k \kappa_i \delta_{x_i(t)}, 
\]
corresponding to  
\[
  \mathbf{v}(x)=-\sum_{i=1}^k \frac{\kappa_i}{2\pi} \frac{(x-x_i(t))^\perp}{\abs{x-x_i(t)}^2},
\]
and the positions of the vortices $x_i : \R \to \R^2$ satisfy
\[
  \dot{x}_i(t)=-\sum_{\substack{j=1 \\ j \ne i}}^k \frac{\kappa_j}{2\pi} \frac{(x_i(t)-x_j(t))^\perp}{\abs{x_i(t)-x_j(t)}^2}. 
\]
In terms of the Kirchhoff--Routh function
\[
  \mathcal{W}(x_1, \dotsc, x_k)=\frac{1}{2} \sum_{i \ne j} \frac{\kappa_i\kappa_j}{2\pi} \log \frac{1}{\abs{x_i-x_j}},
\]
the positions obey Kirchhoff's law
\begin{equation}
\label{equationKirchhoff}
  \kappa_i \dot{x}_i=(\nabla_{x_i} \mathcal{W})^\perp,
\end{equation}
which is a Hamiltonian formulation of the dynamics of the vortices. 

\medskip

In simply-connected bounded domains $\Omega \subset \R^2$, similar singular solutions exist. If one requires for example that the normal component of $\mathbf{v}$ vanishes on the boundary, the associated Kirchoff--Routh function is then given by 
\begin{equation}
\label{eqKRDomainsHomog}
  \mathcal{W}(x_1, \dotsc, x_k)=\frac{1}{2} \sum_{i \ne j} \kappa_i\kappa_j G(x_i, x_j)+\sum_{i=1}^k \frac{\kappa_i^2}{2}H(x_i, x_i),
\end{equation}
where $G$ is the Green function of $-\Delta$ on $\Omega$ with Dirichlet boundary conditions and $H$ is its regular part.\footnote{The function $x \mapsto H(x, x)$ is called the \emph{Robin function} of $\Omega$.} 
One can also prescribe a condition $v_n$ on the outward component of the velocity on the boundary. Since we are dealing with an incompressible flow, the boundary data should satisfy $\int_{\partial \Omega} v_n=0$. Let $\mathbf{v}_0$ be the unique harmonic field whose normal component on the boundary is $v_n$; i.e., $\mathbf{v}_0$ satisfies
\[
\left\{
\begin{aligned}
  \nabla \cdot \mathbf{v}_0&=0, & & \text{in $\Omega$}, \\
  \nabla \times \mathbf{v}_0&=0, & & \text{in $\Omega$}, \\
  n \cdot \mathbf{v}_0&=v_n& & \text{on $\partial \Omega$},
\end{aligned}
\right. 
\]
where $\nabla \times (u, v)=\partial_1 v-\partial_2 u$ and $n$ is the outward normal, then the positions of the vortices are obtained by the modified law 
\[
  \dot{x}_i=(\nabla_{x_i} \mathcal{W})^\perp +\mathbf{v}_0. 
\]
Since $\Omega$ is simply-connected $\mathbf{v}_0$ can be written $\mathbf{v}_0=(\nabla \psi_0)^\perp$ where the stream function $\psi_0$ is characterized up to a constant by
\begin{equation}
 \label{eqpsi0}
\left\{
\begin{aligned}
-\Delta \psi_0&=0& &\text{in $\Omega$}, \\
-\frac{\partial \psi_0}{\partial \tau}&=v_n & &  \text{on $\partial \Omega$},
\end{aligned}
\right. 
\end{equation} 
where $\frac{\partial \psi_0}{\partial \tau}$ denotes the tangential derivative on $\partial \Omega$. 
The Kirchhoff--Routh function associated to the vortex dynamics becomes then
\begin{equation}
\label{KRDomains}
  \mathcal{W}(x_1, \dotsc, x_k)=\frac{1}{2} \sum_{i \ne j} \kappa_i\kappa_j G(x_i, x_j)+\sum_{i=1}^k \frac{\kappa_i^2}{2}H(x_i, x_i)+\sum_{i=1}^k \kappa_i \psi_0(x_i),
\end{equation}
see C.\thinspace C.\thinspace Lin \cite{Lin1941} (who uses opposite sign conventions).

\subsection{Desingularization of vortices} One way to justify the weak formulation for point vortex solutions of the Euler equations is to approximate these solutions by classical solutions. This can actually be done, on finite time intervals, by considering regularized initial data for the vorticity (see e.g.\ C.\thinspace Marchioro and M. Pulvirenti \cite{MarchioroPulvirenti1983}). 

 Critical points of the Kirchhoff--Routh function $\mathcal{W}$ give rise to stationary vortex points solutions of the Euler equations. As noted above, these weak stationary solutions can be approximated by classical solutions of the Euler equations. These do not need be stationary solutions though, and one can wish to approximate the stationary vortex-point solutions by stationary classical solutions. In the simplest case, corresponding to a single point vortex in a simply-connected domain, we obtain the following

 \begin{theorem}\label{thm:resu}
	 Let $\Omega \subset \R^2$ be a bounded simply-connected smooth\footnote{Here and in the sequel, smooth means Lipschitz and is sufficient for our goals.} domain and $v_n:\partial \Omega \to \R\in L^s(\partial \Omega)$ for some $s>1$ be such that $\int_{\partial \Omega} v_n = 0$. Let $\kappa >0$ be given. For $\eps>0$ there exist smooth stationary solutions $\mathbf{v}_\eps$ of the Euler equation in $\Omega$ with outward boundary flux given by $v_n$, corresponding to vorticities $\omega_\eps$, such that ${\rm supp}(\omega_\eps) \subset B(x_\eps, C\eps)$ for some $x_\eps \in \Omega$ and $C>0$ not depending on $\eps$. Moreover, as $\eps \to 0$, 
\[
\int_\Omega \omega_\eps \to \kappa,
\]
and
\[
   \mathcal{W}(x^\eps) \to \sup_{x \in \Omega} \mathcal{W}(x). 
\]
\end{theorem}

Other situations, corresponding to pairs of vortices of opposite signs, multiply-connected bounded 
domains or unbounded domains are discussed in Section~\ref{sect:resu}. 

\medskip

We are aware essentially of two methods to construct stationary solutions of the Euler equations that we call the vorticity method and the stream-function method. 

The vorticity method was introduced by V.\thinspace Arnold (see \cite{ArnoldKhesin}*{Chapter II \S 2}), and was implemented successfully by G.\thinspace R.\thinspace Burton \cite{Burton1988} and B.\thinspace Turkington \cite{Turkington1983}. It roughly consists in maximizing the kinetic energy
\[
  \frac{1}{2}\int_{\Omega} \int_{\Omega} \omega(x)G(x, y)\omega(y)\, dx\, dy+\int_{\Omega} \psi_0(x)\omega(x)\, dx+\frac{1}{2} \int_{\Omega} \abs{\nabla \psi_0}^2,
\]
under some constraints on the sublevel sets of $\omega$. The function $\omega$ is the vorticity of the flow and a stream function $\psi$ is the solution to
\[
\left\{
 \begin{aligned}
   -\Delta \psi &=\omega & & \text{in $\Omega$},\\
   \psi&=\psi_0 & & \text{on $\partial \Omega$}.
 \end{aligned}
\right.
\]
Considering suitable families of constraints on the sublevel sets of $\omega$, one can obtain families of solutions converging to stationary vortex-point solutions. The differentiability of those solutions is not guaranteed (the solutions correspond to vortex patches of constant density). 

The stream-function method starts from the observation that if $\psi$ satisfies
\[
  -\Delta \psi=f(\psi),
\]
for some arbitrary function $f \in C^1(\R)$, then $\mathbf{v}=(\nabla \psi)^\perp$ and $p=F(\psi)-\frac{1}{2}\abs{\nabla \psi}^2$, with $F(s)=\int_0^s f$ form a stationary solution to the Euler equations. Moreover, the velocity $\mathbf{v}$ is irotational on the set where $f(\psi)=0$.

We now set $q=-\psi_0$ and $u=\psi-\psi_0$, so that $u=0$ on $\partial \Omega$ and $-\Delta u = f(u-q)$ in $\Omega$. If we assume that $\inf_{\Omega} q > 0$ and $f(t)=0$ when $t \le 0$, the vorticity set $\{ x \st f(\psi(x))> 0 \}$ is bounded away from the boundary. 
When $f$ satisfies also some monotonicity and growth conditions, $\Omega=\R^2_+$ and $q(x)=W x_1+d$ with $W > 0$ and $d>0$, J.\thinspace Norbury \cite{Norbury1975} has shown the existence of solutions to $-\Delta u = \nu f(u-q)$, where $\nu > 0$ is a Lagrange multiplier a priori unknown by minimizing $\int_{\Omega} \abs{\nabla u}^2$ under the constraint
\[
  \int_{\Omega} F(u-q)=\mu
\]
in $H^1_0(\Omega)$ when $\Omega$ is the half-plane $\R^2_+$. M.\thinspace S.\thinspace Berger and L.\thinspace E.\thinspace Fraenkel \cite{BergerFraenkel1980} have obtained corresponding results for a bounded domain $\Omega \subset \R^2$, and they began studying the asymptotics for variable $\mu$ and $q$, but the lack of information on $\nu$ remained an obstacle. 

The unknown $\nu$ can be avoided by minimizing $\int_{\Omega} \frac{1}{2}\abs{\nabla u}^2-\frac{1}{\eps^2}F(u-q)$ under the natural constraint $\int_{\Omega} \frac{1}{2}\abs{\nabla u}^2-\frac{1}{\eps^2}uf(u-q)=0$. Yang Jianfu \cite{Yang1991} has used this approach in $\R^2_+$ with $q(x)=Wx_1+d$ and has studied the asymptotic behavior of the solution $u^\eps$ when $\eps \to 0$: If
\begin{align*}
  A_\eps&=\{ x \in \R^2_+ \st f(u^\eps-q) > 0\}, & \kappa_\eps&=\frac{1}{\eps^2} \int_{\Omega} f(u^\eps-q),
\end{align*}
and $x^\eps \in A^\eps$, then $\diam A^\eps \to 0$, $\dist(x^\eps, \partial \R^2_+)\to 0$,  
and
\[
  \frac{u^\eps}{\kappa^\eps}-G(\cdot, x^\eps) \to 0
\]
in $W^{1, r}_{\mathrm{loc}}(\R^2_+)$, for $r \in [1, 2)$. Li Gongbao, Yan Shusen and Yang Jianfu \cite{LiYanYang2005} obtained a similar result on bounded domains, with the additional information that $q(a^\eps) \to \min_{\Omega} q$. 
These results are in striking contrast with the observation made at the beginning that the dynamics of the vortices is governed by the Kirchhoff--Routh function $\mathcal{W}$ defined by \eqref{KRDomains}, which implies that stationary vortices should be localized around a critical point of $x \mapsto \frac{\kappa^2}{2} H(x, x)-\kappa q(x)$. 

In fact, the results in \cites{Yang1991, LiYanYang2005} do not answer the question about the desingularization of stationary vortex point solutions to the Euler equation. Indeed, in the  case of bounded domains for example, their solutions satisfy  
$\Norm{\nabla u}_{\mathrm{L}^2}^2=O\bigl(\logeps^{-1}\bigr)$, so that testing the equation against the function $\min(u^\eps, q)$ and using the fact that $q$ is harmonic and nonnegative, we have
\[
  \kappa^\eps \min_{\partial \Omega} q \le
  \frac{1}{\eps^2}\int_{\Omega} f(u^\eps-q) =\int_{\Omega \setminus A^\eps} \abs{\nabla u^\eps}^2=O\bigl(\logeps^{-1}\bigr),
\]
i.e.\ $\kappa^\eps \to 0$. In some sense, the family of solutions $u^\eps$ provides a desingularization of point-vortex solutions with vanishing vorticity. The asymptotic position is consistent with the fact that when the vorticities tend to zero, the term $\sum_{i=1}^k \kappa_i \psi_0(x_i)$ becomes dominant in the Kirchhoff--Routh function \eqref{KRDomains}. 

\medskip

In order to desingularize point-vortex solutions with non-vanishing vorticity, M.\thinspace S.  Berger and L.\thinspace E.\thinspace Fraenkel \cite{BergerFraenkel1980}*{Remark 2} suggest that $q$ should grow like $\log \frac{1}{\eps}$. This brings us to the study of the problem
\begin{equation}
\label{problemPeps}
\left\{
\begin{aligned}
-\Delta u^\eps &=\frac{1}{\eps^2} f(u^\eps - q^\eps ) & &\text{in $\Omega$, }\\
u^\eps &=  0 & &\text{on $\partial \Omega$},
\end{aligned}
\right. \tag{\protect{$\mathcal{P}^\eps$}}
\end{equation}
where $q^\eps=q+\frac{\kappa}{2\pi} \log \frac{1}{\eps}$. 

In Section \ref{sectionSingleVortex}, we study $(\mathcal{P}_\eps)$ in a bounded domain: we first construct solutions and then analyze their asymptotic behavior. Theorem~\ref{thm:resu} is an easy consequence of the results in Section~\ref{sectionSingleVortex}. In Section~\ref{sectionmultiply} we present and extension to multiply-connected domains, while in Section~\ref{sectUnbounded}, we present an extension to unbounded domains which are a perturbation of a half-plane. In Section~\ref{sectionVortexPair} we modify slightly $(\mathcal{P}_\eps)$ in order to construct desingularized solutions for two point vortices of opposite signs. 

\medskip

As a final remark, our results seem connected with the work of M.\thinspace del Pino, M.\thinspace Kowalczyk, and M.\thinspace Musso \cite{delPinoKowalczykMusso2005} on the equation 
\[
-\Delta u=\eps^2 K(x)e^u 
\]
for which the energy concentrates in small balls around points $x_1^\eps, \dotsc, x_k^\eps$. These points tend to a critical point of the function $-\sum_{i=1}^k 2\log K(x_i)-8\pi H(x_i, x_i)-\sum_{i \ne j} 8\pi G(x_i, x_j)$. The connection is clear when one rewrites their equation as $-\Delta u=\frac{1}{\eps^2}\exp(u+\log K-\frac{8\pi}{2\pi}\log \frac{1}{\eps})$. Other related work include the study of the equation $-\Delta u = u^p$ as $p \to \infty$ by P.\thinspace Esposito, M.\thinspace Musso and A. Pistoia \cites{EspositoMussoPistoia2006,EspositoMussoPistoia2007}, and the recent work of T.\thinspace Bartsch, A.\thinspace Pistoia and T.\thinspace Weth \cite{BartschPistoiaWeth} in which systems of three and four vortices are desingularized by studying the equation $-\Delta u= \eps^2 \sinh u$. In all the references, whereas the vorticity concentrates at points, its support does not shrink as $\eps \to 0$.

We also bring to the attention of the reader that there is a similar situation with similar results for three-dimensional axisymmetric incompressible inviscid flows by vorticity methods \cites{Burton1987, FridemannTurkington1981} and stream-function methods \cites{BergerFraenkel1974, AmbrosettiStruwe1989, Yang1995}. However we are not aware of a counterpart of the present work for three-dimensional axisymmetric incompressible inviscid flows. 

\medskip
\noindent{\bf Acknowledgements.} This work was initiated during a visit of the second author at Laboratoire Jacques-Louis Lions of Universit\'e Pierre \& Marie Curie. The authors wish to thank Franck Sueur for fruitful remarks following a first version of the manuscript.

\section{Single vortices in bounded domains}
\label{sectionSingleVortex}

In this section, $\Omega\subset \R^2$ is a bounded simply-connected smooth domain,  $f : \R \to \R$ is the real
function defined 
by $f(s)=s_+^p$ for some $1<p<+\infty$ and where $s_+=\max(s, 0)$,  $\kappa > 0$ is given as well as $q\in \mathrm{W}^{1, r}(\Omega)$ for some $r > 2$.\footnote{Notice that for the proof of Theorem~\ref{thm:resu} we only require a harmonic function $q$ but the proofs of Theorems~\ref{thmLocalMinimum} and \ref{thmRotating} require more general $q$.} We will consider solutions of the boundary value problem \eqref{problemPeps}
where $\eps>0$ is a real parameter. The solutions we consider are the least energy solutions obtained by minimizing the energy functional
\begin{equation}
\label{energyFunctional}
\mathcal{E}^\eps(u)=  \int_{\Omega} \Bigl(\frac{|\nabla u|^2}{2} -
\frac{1}{\eps^2}F(u-q^\eps)\Bigr) 
\end{equation}
over the natural constraint given by the Nehari manifold
\[
\mathcal{N}^\eps = \left\{ u\in H^1_0(\Omega)\setminus \{0\} \ : \ \langle
d\mathcal{E}^\eps(u), u\rangle = 0\right\},
\] 
where $F(s)=\frac{1}{p+1}s_+^{p+1}$ is a primitive of $f$. 
It is standard to
prove the (see e.g.~\cite{Willem1996}*{Theorem 2.18})
\begin{proposition} \label{prop:2.1}
Assume that $q^\eps\geq 0$ on $\Omega$, so that $\Neps \neq \emptyset$, and define
\[
c^\eps = \inf_{u\in \mathcal{N}^\eps} \mathcal{E}^\eps(u). 
\]
Then, there exists $u^\eps \in \mathcal{N}^\eps$ such that $\mathcal{E}^\eps(u^\eps)=c^\eps$,
and $u^\eps$ is a positive solution of $\Peps$. 
\end{proposition}

Note that $q$ is bounded since $r>2$, and therefore $q^\eps\geq 0$ provided $\eps$ is sufficiently small. 

\medskip

Our focus is the asymptotics of $u^\eps$ when $\eps \to 0$. 
In order to describe the asymptotic behavior of $u^\eps$, we introduce the limiting profile
$U_\kappa : \R^2 \to \R$ defined as the unique radially symmetric solution of the problem
\[
\tag{\protect{$\mathcal{U}_\kappa$}}
\label{Ukappa}
\left\{
\begin{aligned}
  &-\Delta U_\kappa = f(U_\kappa), \\
  &\int_{\R^2} f(U_\kappa) =\kappa. 
\end{aligned}
\right. 
\]
For every $\kappa>0$, there exists $\rho_\kappa>0$ such that 
\[
  U_\kappa(y)=
\left\{
\begin{aligned}
  &V_{\rho_\kappa}(y)& &\text{if $y  \in B(0, \rho_\kappa)$}, \\
  &\frac{\kappa}{2\pi} \log \frac{\rho_\kappa}{\abs{y}} & &\text{if $y \in \R^2 \setminus B(0, \rho_\kappa)$},
\end{aligned}\right. 
\]
where $V_\rho : B(0,\rho) \to \R$ satisfies
\[
\left\{
\begin{aligned}
\displaystyle -\Delta V_\rho &= V_\rho^p & & \text{in $B(0, \rho)$}, \\
V_\rho &=  0 && \text{on $\partial B(0, \rho)$}. 
\end{aligned}
\right. 
\]
One can show that $\kappa=\gamma \rho^{-\frac{2}{p-1}}$, for some constant
$\gamma > 0$ depending on the value of $p$. 

\medskip

The Kirchhoff-Routh function $\mathcal{W}$ for one vortex of vorticity $\kappa$ is defined by
\[
 \mathcal{W}(x)=\frac{\kappa^2}{2} H(x, x)-\kappa q(x). 
\] 
Let us also define the quantity
\[
\mathcal{C} = \frac{\kappa^2}{4\pi} \log \rho_\kappa +
\int_{B(0, \rho_\kappa)}\Bigl(\frac{|\nabla U_{\rho_\kappa}|^2}{2} - \frac{U_{\rho_\kappa}^{p+1}}{p+1}\Bigr). 
\]
While the function $\mathcal{W}$ depends on $x \in \Omega$ and on $\kappa$, the quantity $\mathcal{C}$ only depends on $\kappa$ and on $p$. 

\medskip

We set
\begin{equation}
\begin{aligned}\label{defiq}
  A^\eps&=\Big\{ x \in \Omega \st u^\eps(x)> q^\eps(x)\Big\}, \\
  \omega^\eps&=\frac{1}{\eps^2} f(u^\eps-q^\eps), \\ 
  \kappa^\eps&=\int_{\Omega} \omega^\eps, \\
  x^\eps&=\frac{1}{\kappa^\eps}\int_{\Omega} x \, \omega^\eps(x)\, dx, \\
  \rho^\eps&=\rho_{\kappa^\eps},
\end{aligned}
\end{equation}
and respectively refer to these as the vorticity set, the vorticity, the total vorticity, the center of vorticity, and the vorticity radius. 

We will prove

\begin{theorem}\label{thm:K1}
As $\eps \to 0$, we have
\[
  u^\eps=U_{\kappa^\eps} \Big(\frac{\cdot-x^\eps}{\eps}\Big)+\kappa^\eps\Bigl(\frac{1}{2\pi} \log \frac{1}{\eps \rho^\eps}+ H(x^\eps, \cdot)\Bigr)+o(1),
\]
\text{in $\mathrm{W}^{2, 1}_\mathrm{loc}(\Omega)$, in $\mathrm{W}^{1, 2}_0(\Omega)$, and in $\mathrm{L}^\infty(\Omega)$}, where
\[
  \kappa^\eps=\kappa+\frac{2\pi}{\log \frac{1}{\eps}}\Bigl(q(x^\eps)-\kappa H(x^\eps, x^\eps) -\frac{\kappa}{2\pi} \log \frac{1}{\rho_\kappa} \Bigr)+o(\logeps^{-1}),
\]
and
\[
   \mathcal{W}(x^\eps) \to \sup_{x \in \Omega} \mathcal{W}(x). 
\]
One also has
\[
  B(x^\eps, \Bar{r}^\eps) \subset A^\eps \subset B(x^\eps, \mathring{r}^\eps),
\] 
with $\Bar{r}^\eps=\eps \rho_\kappa+o(\eps)$ and $\mathring{r}^\eps=\eps
\rho_\kappa +o(\eps)$. Finally, 
\[
  \mathcal{E}^\eps (u^\eps)= \frac{\kappa^2}{4\pi}\log \frac{1}{\eps}-\mathcal{W}(x^\eps)+\mathcal{C}+o(1). 
\]
\end{theorem}

Since $\mathcal{W}(x) \to -\infty$ as $x \to \partial \Omega$, by
Theorem~\ref{thm:K1}, up to a subsequence, $x^\eps \to x^*\in \Omega$. Combined
with standard elliptic estimates this yields the convergence  
$u^\eps \to \kappa G(x_*, \cdot\, )$ in $\mathrm{W}^{1, p}_0(\Omega)$ for any $p<2$ and in
$\mathcal{C}^k_{\mathrm{loc}}(\Omega\setminus \{x_*\})$ for any $k\in \N$. If $\partial \Omega$ is smooth enough, then one also has convergence in $\mathcal{C}^k_{\mathrm{loc}}(\Bar{\Omega}\setminus \{x_*\}\})$.

\medskip

The proof of Theorem~\ref{thm:K1} is twofold. First, in Corollary~\ref{cor:upper}, we prove a sharp upper
bounds for the critical level $c^\eps$. Then, in Proposition~\ref{prop:1mai} we show that any solution
satisfying this upper bound needs to satisfy the asymptotic expansion.

\subsection{Upper bounds on the energy}
\label{upperBounds}

We will derive upper bounds for $c^\eps$ by constructing elements of $\mathcal{N}^\eps$ similar to the asymptotic expression of Theorem~\ref{thm:K1}. 

\begin{lemma}
\label{lemmaHatuNehari}
For every $\Hat{x} \in \Omega$, if $\eps>0$ is small enough, there exists 
\[
  \Hat{\kappa}^\eps=\kappa+\frac{2\pi}{\log \tfrac{1}{\eps}}\Bigl( q(\Hat{x})-\kappa H(\Hat{x}, \Hat{x})+\dfrac{\kappa}{2\pi} \log \rho_\kappa \Bigr)+O\bigl(\logeps^{-2}\bigr),
\]
such that, if 
\[
  \Hat{u}^\eps(x)=U_{\Hat{\kappa}^\eps}\Bigl(\frac{x-\Hat{x}}{\eps}\Bigr)+\Hat{\kappa}^\eps 
    \Bigl( \frac{1}{2\pi} \log \frac{1}{\eps\rho_{\Hat{\kappa}^\eps}}+H(\Hat{x}, x) \Bigr),
\]
then 
\[
  \Hat{u}^\eps\in \mathcal{N}^\eps. 
\]
Moreover, we have
\[ 
  \Hat{A}^\eps:=\Bigl\{ x \st \Hat{u}^\eps(x) > q(x)+\frac{\kappa}{2\pi} \log \frac{1}{\epsilon} \Bigr\} \subset B(\Hat{x}, \Hat{r}^\eps),
\]
with $\Hat{r}^\eps=O(\eps)$. 
\end{lemma}
\begin{proof}
For $\sigma \in \R$, define
\begin{align*}
  \Hat{\kappa}^{\eps, \sigma}&=\frac{q^\eps(\Hat{x})+\sigma}{\tfrac{1}{2\pi} \log \tfrac{1}{\eps \rho_\kappa}+H(\hat{x}, \Hat{x})}, \\
  \Hat{\rho}^{\eps, \sigma}&=\rho_{\Hat{\kappa}^{\eps, \sigma}}, \\
  \Hat{u}^{\eps, \sigma}(x)&=U_{\Hat{\kappa}^{\eps, \sigma}}\Bigl(\frac{x-\Hat{x}}{\eps}\Bigr)+\Hat{\kappa}^{\eps, \sigma} \Bigl( \frac{1}{2\pi} \log \frac{1}{\eps\rho_{\Hat{\kappa}^{\eps, \sigma}}}+H(\Hat{x}, x) \Bigr). 
\end{align*}
First note that when $\eps>0$ is sufficiently small,
$\Hat{u}^{\eps, \sigma}(x)=\hat{\kappa}_{\sigma, \eps} G(\Hat{x}, x)$ in a
neighborhood of $\partial \Omega$, so that $\Hat{u}^{\eps, \sigma} \in
W^{1, 2}_0(\Omega)$ and we can define
\[  
  g^\eps(\sigma)=\langle d \mathcal{E}^\eps (\Hat{u}^{\eps, \sigma}), \Hat{u}^{\eps, \sigma} \rangle. 
\]
Among the terms involved in $g^\eps(\sigma)$, we may already compute 
\[
\begin{split}
  \int_{\Omega} \abs{\nabla \Hat{u}^{\eps, \sigma}}^2
  &=\int_{B(\Hat{x}, \eps \rho_{\Hat{\kappa}^{\eps, \sigma}})}\!\!\!\!\!\!\!\!\!\!\!\! \abs{\nabla (U_{\Hat{\kappa}^{\eps, \sigma}}(\tfrac{\cdot-\Hat{x}}{\eps}) +\Hat{\kappa}^{\eps, \sigma} H(\Hat{x}, \cdot))}^2 \\
&\qquad\qquad+(\Hat{\kappa}^{\eps, \sigma})^2\int_{\Omega \setminus B(\Hat{x}, \rho_{\Hat{\kappa}^{\eps, \sigma}} \eps)}\!\!\!\!\!\!\!\!\!\!\!\! \abs{\nabla G(\Hat{x}, \cdot)}^2 \\
  &=\int_{B(0, \rho_{\Hat{\kappa}^{\eps, \sigma}})}\!\!\!\!\!\!\!\!\!\!\!\! \abs{\nabla U_{\Hat{\kappa}^{\eps, \sigma}}}^2+O(\eps) \\
&\qquad\qquad+(\Hat{\kappa}^{\eps, \sigma})^2\Bigl(\frac{1}{2\pi} \log \frac{1}{\eps \rho_{\Hat{\kappa}^{\eps, \sigma}}}+H(\Hat{x}, \Hat{x})+O(\eps) \Bigr) \\
 &=\int_{B(0, \rho_{\Hat{\kappa}^{\eps, \sigma}})}\!\!\!\!\!\!\!\!\!\!\!\! \abs{\nabla U_{\Hat{\kappa}^{\eps, \sigma}}}^2
     +\Hat{\kappa}^{\eps, \sigma} \bigl( q^\eps(\Hat{x})+\sigma \bigr)+O(\eps). 
\end{split}
\]

In order to estimate the second term involved in $g^\eps(\sigma)$, namely  $\frac{1}{\eps^2}\int_{\Omega}
f(\Hat{u}^{\eps, \sigma}-q^\eps)\Hat{u}^{\eps, \sigma}$,  we first claim that
\begin{equation}
\label{HatAepssigma}
  \Hat{A}^{\eps, \sigma}:=\bigl\{x \in \Omega \st \Hat{u}^{\eps, \sigma}(x) >
  q^\eps(x)\bigr\} \subset B(\Hat{x}, r^\eps),
\end{equation}
with $r^\eps=O(\eps)$. 
Indeed, let $x \in \Hat{A}^{\eps, \sigma} \setminus
B(\Hat{x}, \Hat{\rho}^{\eps, \sigma}\eps)$. 
One has, by definition of $\Hat{u}^{\eps, \sigma}(x)$ and of $\Hat{\kappa}^{\eps, \sigma}$,
\[
  \Hat{\kappa}^{\eps, \sigma}\Bigl(\frac{1}{2\pi}\log
  \frac{1}{\eps}+\frac{1}{2\pi} \log \frac{\eps}{\abs{x-\Hat{x}}}+H(\Hat{x}, x)\Bigr)
> q(x)+\frac{\kappa}{2\pi} \log \frac{1}{\eps},
\]
so that
\begin{equation}
\label{ineqVorticitySetUpperFrac}
  \frac{\dfrac{1}{2\pi} \log \dfrac{1}{\eps}+\dfrac{1}{2\pi} \log
  \dfrac{\eps}{\abs{x-\Hat{x}}}+ H(\Hat{x}, x)}{\dfrac{\kappa}{2\pi} \log \dfrac{1}{\eps}+q(x)} \ge  
\dfrac{\log \dfrac{1}{\eps}+H(\Hat{x}, \Hat{x})}{\dfrac{\kappa}{2\pi} \log \dfrac{1}{\eps}+q(\Hat{x})+\sigma}. 
\end{equation}
Since $q$ and $H(\Hat{x}, \cdot)$ are bounded functions, one obtains that
\[
\frac{1}{\kappa}+\frac{\log \frac{\eps}{\abs{x-\Hat{x}}}}{\kappa \log \dfrac{1}{\eps}}\ge \frac{1}{\kappa}+O\bigl(\logeps^{-1}\bigr),
\]
and the claim is proved. 
We deduce from~\eqref{HatAepssigma}, that for every $x \in \Hat{A}^{\eps, \sigma}$
\[
  \Hat{u}^{\eps, \sigma}(x)-q^\eps(x)=U_{\Hat{\kappa}^{\eps, \sigma}}\Bigl(\frac{x-\Hat{x}}{\eps}\Bigr)+\sigma+O(\eps). 
\]

We may now estimate
\[
\begin{split}
  \frac{1}{\eps^2}\int_{\Omega}& f(\Hat{u}^{\eps, \sigma}-q^\eps)\Hat{u}^{\eps, \sigma}
  =\frac{1}{\eps^2}\int_{\Hat{A}^{\eps, \sigma}} f(\Hat{u}^{\eps, \sigma}-q^\eps)\Hat{u}^{\eps, \sigma}\\
  &=\frac{1}{\eps^2}\int_{\Hat{A}^{\eps, \sigma}}   f(\Hat{u}^{\eps, \sigma}-q^\eps)U_{\Hat{\kappa}^{\eps, \sigma}}(\tfrac{\cdot-\Hat{x}}{\eps}) \\
  &\qquad\qquad+\frac{\Hat{\kappa}^{\eps, \sigma}}{\eps^2}\int_{\Hat{A}^{\eps, \sigma}} f(\Hat{u}^{\eps, \sigma}-q^\eps)\bigl(\tfrac{1}{2\pi}\log \tfrac{1}{\eps \Hat{\rho}^{\eps, \sigma}}+ H(\Hat{x}, \cdot)\bigr) \\
  &=\int_{\R^2} f(U_{\Hat{\kappa}^{\eps, \sigma}}+\sigma)U_{\Hat{\kappa}^{\eps, \sigma}}+O(\eps) \\
&\qquad\qquad+ \Hat{\kappa}^{\eps, \sigma}\bigl(\tfrac{1}{2\pi} \log \tfrac{1}{\eps \rho_\kappa}+ H(\Hat{x}, \Hat{x})+O(\eps)\bigr)\Bigl(\int_{\R^2} f(U_\kappa+\sigma)+O(\eps)\Bigr)\\
 &=\int_{\R^2} f(U_{\Hat{\kappa}^{\eps, \sigma}}+\sigma)U_{\Hat{\kappa}^{\eps, \sigma}}  \\
&\qquad\qquad+ \big(\tfrac{\kappa}{2\pi} \log
\tfrac{1}{\eps}+q(\Hat{x})+\sigma\big)\int_{\R^2}f(U_{\Hat{\kappa}^{\eps, \sigma}}+\sigma)+O(\eps \logeps). 
\end{split}
\]
Summarizing, we have
\[
\begin{split}
g^\eps(\sigma)&=\frac{\kappa}{2\pi}\log \frac{1}{\eps} \Bigl(\Hat{\kappa}^{\eps, \sigma} - \int_{\R^2} f(U_{\Hat{\kappa}^{\eps, \sigma}}+\sigma)\Bigr)+O(1)\\
&=\frac{\kappa}{2\pi} \log \frac{1}{\eps} \Bigl(\int_{\R^2} f(U_{\kappa})-f(U_{\kappa}+\sigma)\Bigr)+O(1). 
\end{split}
\]
Since $g^\eps$ is continuous and $\sigma \cdot \Bigl(\int_{\R^2} f(U_\kappa)-f(U_\kappa+\sigma)\Bigr)<0$ when
$\sigma \ne 0$, there exists $\sigma^\eps$ such that $g(\sigma^\eps)=0$ and
$\sigma^\eps \to 0$ as $\eps\to 0$. One then sets $\Hat{\kappa}^\eps=\Hat{\kappa}^{\eps, \sigma^\eps}$. 
\end{proof}

\begin{lemma}
\label{lemmaEnergyHatu}
For every $\Hat{x} \in \Omega$, we have
\[
  c^\eps \le \frac{\kappa^2}{4\pi}\log \frac{1}{\eps}
  -\mathcal{W}(\Hat{x})+\mathcal{C}+o(1)\qquad\text{as }\eps\to 0. 
\]
\end{lemma}
\begin{proof}
By Lemma~\ref{lemmaHatuNehari}, $\Hat{u}^\eps \in \mathcal{N}^\eps$, so that
$c^\eps \leq \mathcal{E}^\eps(\Hat{u}^\eps)$. We compute the energy of
$\Hat{u}^\eps$ as follows. First,
\[
\begin{split}
  \int_{\Omega} \abs{\nabla \Hat{u}^\eps}^2
  &=\int_{\Omega} \Hat{u}^\eps \Delta \Hat{u}^\eps\\
  &= -\int_{\R^2} U_\kappa \Delta U_\kappa+(\Hat{\kappa}^\eps)^2\Bigl(\frac{1}{2\pi} \log \frac{1}{\eps}+H(\Hat{x}, \Hat{x})\Bigr)+o(1)\\
  &=\int_{\R^2} \abs{\nabla (U_\kappa)_+}^2+\frac{\kappa^2}{2\pi} \log \frac{1}{\eps} +2\kappa q(\Hat x)-\kappa^2 H(\Hat x, \Hat x) 
+\frac{\kappa^2}{2\pi} \log{\rho_\kappa} +o(1). 
\end{split}
\]
Next,
\[
\begin{split}
  \frac{1}{\eps^2}\int_{\Omega} F(\Hat{u}^\eps-q^\eps)
  &=\frac{1}{\eps^2}\int_{\Hat{A}^\eps} F(\Hat{u}^\eps-q^\eps)\\
  &=\frac{1}{\eps^2}\int_{\Hat{A}^\eps} F(\Hat{u}^\eps-q^\eps(x^\eps))+o(1)\\
  &=\int_{\R^2} F(U_\rho)+o(1),
\end{split}
\]
and the conclusion follows from the definitions of $\mathcal{W}$ and $\mathcal{C}$. 
\end{proof}

\begin{corollary}\label{cor:upper}
We have
\[
c^\eps \leq \frac{\kappa^2}{4\pi} \log \frac{1}{\eps} -\sup_{x\in \Omega} \mathcal{W}(x) +\mathcal{C} + o(1). 
\]
\end{corollary}

\subsection{Asymptotic behavior of solutions}

The main goal of this section is to prove
\begin{proposition}\label{prop:1mai}
Let $(v^\eps)$ be a family of solutions to \eqref{problemPeps} such that $v^\eps \ne 0$
\begin{equation}
\label{assumptEnergyUpperbound}
  \mathcal{E}^\eps(v^\eps) \le \frac{\kappa^2}{4\pi} \log \frac{1}{\eps}+O(1),
\end{equation}
as $\eps \to 0$. Define the
quantities $A^\eps$, $\omega^\eps$, $\kappa^\eps$, $x^\eps$ and $\rho^\eps$ for $v^\eps$ as in \eqref{defiq} for $u^\eps$. 
 Then 
\[
  v^\eps=U_{\kappa^\eps} (\tfrac{\cdot-x^\eps}{\eps})+\kappa^\eps\Bigl(\frac{1}{2\pi} \log \frac{1}{\eps \rho^\eps}+ H(x^\eps, \cdot)\Bigr)+o(1),
\]
in $\mathrm{W}^{2, 1}_\mathrm{loc}(\Omega)$, in $\mathrm{W}^{1, 2}_0(\Omega)$, and in $\mathrm{L}^\infty(\Omega)$, where
\[
  \kappa^\eps=\kappa+\frac{2\pi}{\log \frac{1}{\eps}}\Bigl(q(x^\eps)-\kappa H(x^\eps, x^\eps) -\frac{\kappa}{2\pi} \log \frac{1}{\rho_\kappa} \Bigr)+o(\logeps^{-1}),
\]
In particular, we have
\[
  \mathcal{E}^\eps (v^\eps)= \frac{\kappa^2}{4\pi}\log \frac{1}{\eps}-\mathcal{W}(x^\eps)+\mathcal{C}+o(1)
\]
and
\[
  B(x^\eps, \Bar{r}^\eps) \subset A^\eps \subset B(x^\eps, \mathring{r}^\eps),
\] 
with $\Bar{r}^\eps=\eps \rho_\kappa+o(\eps)$ and $\mathring{r}^\eps=\eps \rho_\kappa +o(\eps)$. 
\end{proposition}
In other words, 
$v^\eps$ satisfies the same asymptotics as the one stated in
Theorem~\ref{thm:K1} for $u^\eps$ except for the convergence of $x^\eps$.

\medskip

In the sequel, $v^\eps$ denotes a family of nontrivial solutions to \eqref{problemPeps}
verifying \eqref{assumptEnergyUpperbound}. We divide the proof of Proposition
\ref{prop:1mai} into several steps. 

\subsubsection{Step 1: First quantitative properties of the solutions}

In this section, we derive various types of estimates for $v^\eps$. 
\begin{proposition}\label{propositionEstimatesueps} We have, as $\eps \to 0$,
\begin{gather}
\label{ineqMuAeps}\muleb{2}(A^\eps) = O\bigl(\logeps^{-1}\bigr), \\
\label{ineqVortexEnergy} \int_{A^\eps} \abs{\nabla (v^\eps-q^\eps)}^2 =O(1), \\
\label{ineqVortexPotential}\frac{1}{\eps^2}\int_{A^\eps} F(v^\eps-q^\eps) =O(1), \\
\label{eq:2etoiles}\int_{\Omega\setminus A^\eps} |\nabla v^\eps|^2 \leq \frac{\kappa^2}{2\pi} \log\frac{1}{\eps} + O(1), \\
\label{ineqTotalVorticity}\int_{\Omega} \omega^\eps \leq \kappa + O\bigl(\logeps^{-1}\bigr).
\end{gather}
\end{proposition}
\begin{proof}
First note that for $\eps>0$ sufficiently small, 
\begin{equation}
\label{ineqEnergy}
\Bigl(\frac{1}{2}-\frac{1}{p+1}\Bigr)\int_\Omega |\nabla v^\eps|^2 \leq \mathcal{E}^\eps(v^\eps). 
\end{equation}
Indeed,
\[
 \mathcal{E}^\eps(v^\eps) = \frac{1}{2}\int_\Omega |\nabla v^\eps|^2 -
\frac{1}{p+1}\int_\Omega
\frac{1}{\eps^2}f(v^\eps-q^\eps)(v^\eps-q^\eps)_+, 
\]
and, by testing $(\mathcal{P}^\eps)$ against $v^\eps$,
\[
0 = \frac{1}{p+1}\int_\Omega |\nabla v^\eps|^2 -
\frac{1}{p+1}\int_\Omega
\frac{1}{\eps^2}f(v^\eps-q^\eps)v^\eps. 
\]
Since $(v^\eps-q^\eps)_+\leq v^\eps$ when $q^\eps\geq 0$, and hence when
$\eps$ is sufficiently small, \eqref{ineqEnergy} follows by subtraction. 

In order to obtain \eqref{ineqMuAeps}, first note that since $q$ is bounded from
below, for $\eps$ sufficiently small, $\inf_{\Omega} q_\eps >
\frac{\kappa}{4\pi} \log \frac{1}{\eps}$. By the Chebyshev and Poincar\'e
inequalities, it follows that
\[
  \muleb{2}(A^\eps) \le \Bigl(\frac{1}{\inf_{\Omega} q^\eps}\Bigr)^2 \int_{\Omega} \abs{v^\eps}^2 \le \frac{C}{\logeps^2} \int_{\Omega} \abs{\nabla v^\eps}^2\le \frac{C'}{\logeps},
\]
where the last inequality is a consequence \eqref{ineqEnergy} and \eqref{assumptEnergyUpperbound}. 

We claim that 
\begin{equation}
\label{ineqomegaepsL1}
  \int_{\Omega} \omega^\eps \leq C. 
\end{equation}
By testing $\Peps$ against $\min(v^\eps, q^\eps)$ we obtain
\begin{equation}
\label{ineqVorticityEnergy}
\begin{split}
  \int_{\Omega} \omega^\eps= \int_{A^\eps}
\frac{1}{\eps^2}f(v^\eps-q^\eps)& \leq \frac{1}{\inf_\Omega q^\eps}\int_{A^\eps}
\frac{q^\eps}{\eps^2}f(v^\eps-q^\eps)\\
&=\frac{1}{\inf_\Omega q^\eps}\int_{\Omega\setminus A^\eps}
\abs{\nabla v^\eps}^2 + \frac{1}{\inf_\Omega q^\eps}\int_{A^\eps} \nabla v^\eps\nabla q. 
\end{split}
\end{equation}
In view of \eqref{ineqEnergy}, this yields
\[
\kappa^\eps  \leq C \frac{\mathcal{E}^\eps(v^\eps) +o(1)}{\log \frac{1}{\eps}},
\] 
and the estimate \eqref{ineqomegaepsL1} follows from assumption \eqref{assumptEnergyUpperbound}. 

Testing now $\Peps$ against $(v^\eps-q^\eps)_+$, we obtain
\begin{equation}
\label{eqNehariVortex}
\int_{A^\eps} |\nabla (v^\eps-q^\eps)|^2 = \int_{A^\eps} \frac{1}{\eps^2} (v^\eps-q^\eps)_+^{p+1} -\int_{A^\eps}\nabla (v^\eps-q^\eps) \nabla q. 
\end{equation}
The Gagliardo--Nirenberg inequality \cite{Nirenberg1959}*{p.\thinspace 125} yields
\begin{equation}
\label{ineqGN}
\int_{A^\eps} \frac{1}{\eps^2} (v^\eps-q^\eps)_+^{p+1} \leq C \int_{A^\eps}
\frac{1}{\eps^2} (v^\eps-q^\eps)_+^{p} \left(\int_{A^\eps} |\nabla
(v^\eps-q^\eps)|^2\right)^{\frac{1}{2}},
\end{equation}
so that
\[
\begin{split}
\int_{A^\eps} |\nabla (v^\eps-q^\eps)|^2 &\leq C \bigl(\Norm{\omega^\eps}_{\mathrm{L}^1}+\Norm{\nabla q^\eps}_{L^2(A^\eps)} \bigr) \Bigl(\int_{A^\eps}
\abs{\nabla (v^\eps-q^\eps)}^2\Bigr)^{\frac{1}{2}} \\
&\leq C'\Bigl(\int_{A^\eps}
\abs{\nabla (v^\eps-q^\eps)}^2\Bigr)^{\frac{1}{2}}. 
\end{split}
\]
Inequality \eqref{ineqVortexEnergy} can therefore be deduced from \eqref{ineqomegaepsL1}, and \eqref{ineqVortexPotential} follows from \eqref{eqNehariVortex}. 
Finally, 
\[
\begin{split}
  \frac{1}{2}\int_{\Omega \setminus A^\eps} \abs{\nabla v^\eps}^2&=\mathcal{E}^\eps(v^\eps)+\frac{1}{\eps^2}\int_{A^\eps} F(v^\eps-q^\eps)-\frac{1}{2} \int_{A^\eps} \abs{\nabla v^\eps}^2 \\
  &\le \frac{\kappa^2}{4\pi} \log \frac{1}{\eps}+O(1),
\end{split}
\]
so that \eqref{eq:2etoiles} holds, and inequality \eqref{ineqTotalVorticity}
then follows from \eqref{ineqVorticityEnergy}. 
\end{proof}

\begin{remark}
The use of the Gagliardo--Nirenberg inequality to obtain \eqref{ineqGN} is the only step in our proof that requires $f$ to be a power-like nonlinearity.
\end{remark}

\subsubsection{Step 2: Structure of the vorticity set}

We now examine the vorticity set $A^\eps$ further. Since $A^\eps$ is open, it contains at most countably many connected components that we label $A^\eps_i$, $i \in I^\eps$. 
If $q$ were a harmonic function (e.g. if the only goal was to prove Theorem \ref{thm:resu}), one would deduce from the fact that $u^\eps$ is a minimal energy solution that $A^\eps$ is connected whenever $q^\eps \ge 0$ \cite{BergerFraenkel1974}*{Theorem 3F}, \cite{Norbury1975}*{Theorem 3.4}, \cite{AmbrosettiMancini1981}*{Theorem 4}, \cite{Yang1991}*{Theorem 1}, \cite{LiYanYang2005}*{Proposition 3.1}; this would simplify considerably the analysis that we perform below.

First we have a control on the total area and on the diameter of each connected component. 

\begin{lemma}
\label{lemmaAreaDiameter}
If $\eps > 0$ is sufficiently small, we have
\begin{equation}
\label{ineqVorticityAreaStrong}
  \muleb{2}(A^\eps) \le C \eps^2
\end{equation}
and, for every $i \in I^\eps$,
\begin{equation}
\label{ineqVorticityDiameter}
  \diam(A^\eps_i) \le C \eps. 
\end{equation}
\end{lemma}
\begin{proof}
Set
\[
  w^\eps=\frac{v^\eps}{\min_{\partial A^\eps}q^\eps}. 
\]
Since $v^\eps =q^\eps $ on $\partial A^\eps$, we have, by \eqref{eq:2etoiles},
\begin{equation}
\label{ineqCapacity}
  \frac{2\pi}{\capa(A^\eps, \Omega)} 
  \ge \frac{2\pi}{\displaystyle \int_{\Omega\setminus A^\eps} \abs{\nabla w^\eps}^2} 
  \ge 2 \pi \frac{\frac{\kappa^2}{4\pi} \bigl(\log \frac{1}{\eps}\bigr)^2+O(\logeps)}{\displaystyle \int_{\Omega\setminus A^\eps}
\abs{\nabla v^\eps}^2}=\log \frac{1}{\eps}+O(1). 
\end{equation}
By Proposition~\ref{propositionCapacityArea}, it follows that 
\[
  \log \frac{\muleb{2}(\Omega)}{\muleb{2}(A^\eps)} \ge 2\log \frac{1}{\eps}+O(1),
\]
from which \eqref{ineqVorticityAreaStrong} follows. 

Similarly, we have
\[
  \frac{2\pi}{\capa(A^\eps_i, \Omega)} \geq \frac{2\pi}{\capa(A^\eps, \Omega)} \geq \log \frac{1}{\eps}+O(1). 
\]
It hence follows from Proposition~\ref{propositionBoundDiameter} and the boundedness of $\Omega$ that
\[
  \log C\Bigl(1+\frac{1}{\diam (A_i^\eps)}\Bigr) \ge \log \frac{1}{\eps}+O(1), 
\] 
which implies \eqref{ineqVorticityDiameter}.
\end{proof}

\begin{lemma}
\label{lemmaVortexSplit}
There exist positive constants $\gamma$ and $c$ such that when $\eps$ is
small enough, for every $i \in I_\eps$, 
if 
\begin{equation}
\label{eqSplitVortices}
  \int_{A^\eps_i} \abs{\nabla (v^\eps-q^\eps)}^2 > \gamma^2,
\end{equation}
then
\begin{gather}
\label{ineqLowerBoundArea} \muleb{2}(A_i^\eps)\ge c\eps^2, \\ 
\label{ineqLowerBoundDiam}  \diam(A^\eps_i)\ge c\eps, \\
\label{ineqLowerBoundDistance}  \dist(A^\eps_i, \partial \Omega)\ge c, \\
\label{ineqLowerBoundVortex}  \int_{A^\eps_i} \omega_\eps \ge c,
\end{gather}
while if \eqref{eqSplitVortices} does not hold, then for every $s \ge 1$, 
\begin{equation}
\label{ineqfsVanishing}  \int_{A^\eps_i} f(v^\eps-q^\eps)^s \le C \Norm{\nabla q}_{\mathrm{L}^r(A^\eps_i)}^{sp} \muleb{2}(A^\eps_i)^{1+\frac{sp}{2}(1-\frac{2}{r})},
\end{equation}
where $C>0$ only depends on $s \ge 1$. 
\end{lemma}

\begin{proof}
Starting from \eqref{eqNehariVortex}, and applying the Sobolev and
Cauchy--Schwarz inequalities we obtain
\begin{multline}
\label{ineqGradientVortices}
  \int_{A^\eps_i} \abs{\nabla (v^\eps-q^\eps)_+}^2 
  = \int_{A^\eps_i} \frac{f(v^\eps-q^\eps)}{\eps^2}(v^\eps-q^\eps)_+-\int_{A^\eps_i} \nabla q \cdot \nabla (v^\eps-q^\eps)\\
  \le C\frac{\muleb{2}(A^\eps_i)}{\eps^2} 
\Bigl(\int_{A^\eps_i} \abs{\nabla (v^\eps-q^\eps)_+}^2\Bigr)^{\frac{p+1}{2}}\\
+ \Norm{\nabla q}_{\mathrm{L}^2(A^\eps_i)}\Norm{\nabla (v^\eps-q^\eps)_+}_{\mathrm{L}^2(A^\eps_i)}. 
\end{multline}
By Lemma~\ref{lemmaAreaDiameter}, we may choose $\gamma$ sufficiently small so that
\[
   \gamma^{p-1}\le \frac{\eps^2}{2C\muleb{2}(A^\eps_i)},
\]
independently of $\eps$, and therefore if  \eqref{eqSplitVortices} does not hold
we obtain
\begin{equation}
\label{ineqVanVorticesuq}
   \frac{1}{2}\int_{A^\eps_i} \abs{\nabla (v^\eps-q^\eps)_+}^2\le \int_{A^\eps_i} \abs{\nabla q}^2. 
\end{equation}
Applying successively Sobolev inequality, \eqref{ineqVanVorticesuq} and
Lemma~\ref{lemmaAreaDiameter}, we conclude
\[
  \begin{split}
  \int_{A^\eps_i} f(v^\eps-q^\eps)^s 
     &\le C \Bigl( \int_{A^\eps_i} \abs{\nabla (v^\eps-q^\eps)_+}^2 \Bigr)^\frac{sp}{2} \muleb{2}(A^\eps_i)\\
     &\le C'\Bigl( \int_{A^\eps_i} \abs{\nabla q}^2 \Bigr)^\frac{sp}{2} \muleb{2}(A^\eps_i)\\
     &\le C'' \Norm{\nabla q}_{\mathrm{L}^r(A^\eps_i)}^{sp}
     \muleb{2}(A^\eps_i)^{1+\frac{sp}{2}(1-\frac{2}{r})}. 
  \end{split}
\]

Assume now that \eqref{eqSplitVortices} holds. Combined with
\eqref{ineqGradientVortices} and \eqref{ineqVortexEnergy}, this yields 
\[
   \gamma^2 \le C \frac{\muleb{2}(A^\eps_i)}{\eps^2}+C \Norm{\nabla q}_{\mathrm{L}^2(A^\eps_i)}. 
\]
Since $\Norm{\nabla q}_{\mathrm{L}^2(A^\eps_i)} \to 0$ as $\eps \to 0$, one must
have  $\muleb{2}(A^\eps) \ge c \eps^2$. The isodiametric inequality then yields
\eqref{ineqLowerBoundDiam}. 

Turning back to \eqref{ineqCapacity}, and using
Proposition~\ref{propositionBoundDiameter}, we obtain
\[
    \log C\Bigl(1+\frac{\dist(A_i^\eps, \partial \Omega)}{\eps}\Bigr) \ge \log \frac{1}{\eps}+O(1),  
\]
from which \eqref{ineqLowerBoundDistance} follows.

Testing $\Peps$ against $(v^\eps-q^\eps)_+ \chi_{A_\eps^i}$, applying the Gagliardo--Nirenberg inequality and using then \eqref{ineqVortexEnergy}, we have
\[
\int_{A^\eps_i} \abs{\nabla (v^\eps-q^\eps)}^2 \leq C(\int_{A^\eps_i} \omega_\eps+\Norm{\nabla q^\eps}_{L^2(A^\eps_i)})\Bigl(\int_{A^\eps_i}
\abs{\nabla (v^\eps-q^\eps)}^2\Bigr)^{\frac{1}{2}}
\le C'\int_{A^\eps_i} \omega_\eps,
\]
(cf.\ the proof of Proposition~\ref{propositionEstimatesueps}) and the inequality \eqref{ineqLowerBoundVortex} follows.
\end{proof}

In view of Lemma~\ref{lemmaVortexSplit}, we can split the vortices in two classes: the vanishing vortices
\begin{align}
\label{eqDefVeps}  V^\eps&=\bigcup \Bigl\{A_i^\eps \st \int_{A_i^\eps} \abs{\nabla (v^\eps-q^\eps)}^2 \le \gamma^2\Bigr\}, \\
\intertext{and the essential vortices}
\label{eqDefEeps}  E^\eps&=\bigcup \Bigl\{A_i^\eps \st \int_{A_i^\eps} \abs{\nabla (v^\eps-q^\eps)}^2 > \gamma^2\Bigr\}. 
\end{align}
In view of \eqref{ineqVortexEnergy}, $E^\eps$ contains finitely many connected components. 
We can thus split $E^\eps=\bigcup_{j=1}^{k^\eps} E^\eps_j$, where $E^\eps_j$ are nonempty open sets which are not necessarily connected such that, up to a subsequence,
\begin{equation}
\label{eqDistEepsi}
  \frac{\dist(E^\eps_i, E^\eps_j)}{\eps} \to \infty
\end{equation}
as $\eps \to 0$, and
\begin{equation}
\label{ineqDiamEepsi}  \Tilde{\rho}= \limsup_{\eps \to 0} \frac{\diam (E^\eps_i)}{\eps} < \infty. 
\end{equation}
By definition of $E^\eps$ and by \eqref{ineqVortexEnergy}, $k^\eps$ is bounded as $\eps \to 0$. Finally, 
\begin{equation}
\label{eqDistbord}
  \liminf_{\eps \to 0} \dist(E^\eps_i, \partial \Omega)>0. 
\end{equation}

We set 
\begin{align*}
  \omega^\eps_v&=\omega^\eps \charfun{V^\eps}, &
  \omega^\eps_i&=\omega^\eps \charfun{E^\eps_i}, &
  \kappa^\eps_i&=\int_{\Omega} \omega^\eps_i. 
\end{align*}
By \eqref{ineqTotalVorticity}, we have 
\begin{equation}
\label{ineqSumVortices}
  \sum_{i=1}^{k^\eps} \kappa^\eps_i \le \kappa+O\bigl(\logeps^{-1}\bigr). 
\end{equation}

\begin{lemma}
\label{lemmaVanishingVorticity}
For every $s \ge 1$, we have
\[
  \Norm{\omega^\eps_v}_{\mathrm{L}^s} = o\bigl(\eps^{p(1-\frac{2}{r})-2(1-\frac{1}{s})}\bigr). 
\]
In particular, if $\frac{1}{s} \ge 1-p(\frac{1}{2}-\frac{1}{r})$, then $\omega^\eps_v \to 0$ in $\mathrm{L}^s(\Omega)$. 
\end{lemma}

\begin{proof}
Set
\[
 I_v^\eps=\Bigl\{ i \in I^\eps\st \int_{A_i^\eps} \abs{\nabla (v^\eps-q^\eps)}^2 \le \gamma^2\Bigr\}
\]
We have, by Lemma~\ref{lemmaVortexSplit} and by \eqref{ineqVorticityAreaStrong},
\[
\begin{split}
  \int_{\Omega} \abs {\omega^\eps_v}^s 
  &=\sum_{i \in I^\eps_v}\int_{A^\eps_i} \abs{\omega^\eps_v}^s \\
  &\le C\frac{1}{\eps^{2s}} \sum_{i \in I^\eps_v} \Norm{\nabla q}_{\mathrm{L}^r(A^\eps_i)}^{sp} \muleb{2}(A^\eps_i)^{1+\frac{sp}{2}(1-\frac{2}{r})}\\
  &\le C\muleb{2}(V^\eps) \max_{i \in I^\eps_v} \Norm{\nabla q}_{\mathrm{L}^r(A^\eps_i)}^{sp} \frac{\muleb{2}(A^\eps_i)^{1+sp(\frac{1}{2}-\frac{1}{r})}}{\eps^{2s}}\\
  &\le C' \Norm{\nabla q}_{\mathrm{L}^r(V^\eps)}^{sp} \eps^{sp(1-\frac{2}{r})-2(s-1)}. \qedhere
\end{split}
\]
\end{proof}

\begin{lemma}
\label{lemmaNonVanisingVortex}
For $\eps > 0$ sufficiently small, $k_\eps \ge 1$.
\end{lemma}
\begin{proof}
Assume by contradiction that there is a sequence $(\eps_n)$ such that $\eps_n \to 0$ and $k_{\eps_n} =0$.
Take $s > 1$ such that $\frac{1}{s} \ge 1-p(\frac{1}{2}-\frac{1}{r})$. Since $\omega_{\eps_n}=\omega_{\eps_n}^v \to 0$ in $\mathrm{L}^s(\Omega)$ for some $s > 1$ by Lemma~\ref{lemmaVanishingVorticity}; by classical estimates, \cite{GilbargTrudinger2001}*{Theorem 8.15}  $v_{\eps_n} \to 0$ in $L^\infty(\Omega)$. Therefore, when $n$ is large enough, one would have $\omega_{\eps_n}=0$ and thus $v_{\eps_n} = 0$.
\end{proof}

\subsubsection{Step 3: Small scale asymptotics}

\medskip

We define 
\[
  x^\eps_i=\frac{1}{\kappa^\eps_i}\int_{\Omega}\omega^\eps_i(x)x\, dx. 
\]
By \eqref{eqDistEepsi} and \eqref{eqDistbord}, $x^\eps_i \in \Omega$ and $x^\eps_i\ne x^\eps_j$ when $i \ne j$ and $\eps$ is small. 
We also define
\[
  v^\eps_i(y)=v^\eps(x^\eps_i+\eps y)-q^\eps(x^\eps_i),
\]
and
\[
   q^\eps_i(y)=q(x^\eps_i+\eps y)-q(x^\eps_i). 
\]
By \eqref{eqDistbord}, for every $R>0$, $v^\eps_i$ is well-defined in $B(0, R)$ when $\eps$ is sufficiently small, and it satisfies there the equation
\begin{equation}
\label{eqLimit}
 -\Delta v^\eps_i=f(v^\eps_i-q^\eps_i). 
\end{equation}

\begin{lemma}
\label{lemmaSmallScaleLocalEstimates}
For every $R > 0$ and $s\ge 1$, there exist $\eps(R)>0$  and $C>0$ such that for $0<\eps\leq \eps(R)$ we have 
\begin{equation}
\label{ineqRenormEstimate}
  \Norm{f(v^\eps_i-q^\eps_i)}_{\mathrm{L}^s(B(0, R))}\le C. 
\end{equation}
Moreover, for $2\Tilde{\rho} < \abs{y} < R$, we have
\begin{equation}
\label{ineqvepsiDecay}
  \Bigl\lvert v^\eps_i(y)-\frac{\kappa^\eps_i}{2\pi}\log \frac{1}{\eps\abs{y}}+q^\eps(x^\eps_i)-\kappa^\eps_i H(x^\eps_i, x^\eps_i)
  -\sum_{j \ne i} \kappa_j^\eps G(x^\eps_i, x^\eps_j)\Bigr\rvert 
  \le \frac{\kappa}{2\pi} \log \frac{\abs{y}}{\abs{y}-\Tilde{\rho}}+o(1),
\end{equation}
and
\begin{equation}
\label{ineqNablavepsiDecay}
\Bigl\lvert \nabla v^\eps_i(y)-\frac{\kappa}{2\pi} \frac{y}{\abs{y}^2}\Bigr\rvert \le \frac{\Bar{C}}{\abs{y}^3}+o(1).
\end{equation}
as $\eps \to 0$, where $\Bar{C}$ does not depend on $R$. 
\end{lemma}
\begin{proof}
Consider $D^{\eps, R}_i=\bigcup \bigl\{A^\eps_j \st A^\eps_j \cap B(x^\eps_i, \eps R)\ne \emptyset \bigr\}$. By \eqref{ineqVorticityDiameter}, $\muleb{2}(D_i^{\eps, R})=O(\eps^2)$ as $\eps \to 0$, so that one obtains, by Sobolev's inequality,
\begin{multline*}
 \int_{B(0, R)} f(v^\eps_i-q^\eps_i)^s \le
 \frac{1}{\eps^2}\int_{D_i^{\eps, R}} f(v^\eps-q^\eps)^s \\ 
 \le C\frac{1}{\eps^2} \Norm{\nabla(v^\eps-q^\eps)_+}_{\mathrm{L}^2(A^\eps)}^{sp} \muleb{2}(D_i^{\eps, R}) = O(1),
\end{multline*}
which proves \eqref{ineqRenormEstimate}.

We have
\begin{equation}
\label{eqvepsiGomega}
  v^\eps_i(y)=\int_{\Omega} G(x^\eps_i+\eps y, z) \omega^\eps(z)\, dz-q^\eps(x^\eps_i). 
\end{equation}
We first prove \eqref{ineqvepsiDecay}. 
By a classical estimate \cite{GilbargTrudinger2001}*{Theorem 8.15}, 
\begin{equation}
\label{ineqGomegav}
  \Bigl\lvert\int_{\Omega} G(x, z) \omega^\eps_v(z)\, dz\Bigr\rvert
  \le C \Norm{\omega^\eps_v}_{\mathrm{L}^s}. 
\end{equation}
Since by Lemma~\ref{lemmaVanishingVorticity}, $\omega^\eps_v \to 0$ in $\mathrm{L}^s(\Omega)$ for some $s > 1$, we have 
\[
  \int_{\Omega} G(x^\eps_i+\eps y, z) \omega^\eps_v(z)\, dz
  \to 0
\]
uniformly in $y$. 
We also have, since $\diam E^\eps_j=O(\eps)$, $\abs{x^\eps_i-x^\eps_j}/\eps \to \infty$, for $j \ne i$, and $\abs{y} \le R$, 
\[
  \int_{\Omega} G(x^\eps_i+\eps y, z) \omega^\eps_j(z)\, dz
  =\kappa_j^\eps G(x^\eps_i, x^\eps_j)+o(1),
\]
and
\[
 \int_{\Omega} H(x^\eps_i+\eps y, z) \omega^\eps_i(z)\, dz
  =\kappa^\eps_i H(x^\eps_i, x^\eps_i)+O(\eps). 
\]
Finally, we have
\[
\begin{split}
  \int_{\Omega} \frac{1}{2\pi}\log \frac{1}{\abs{x^\eps_i+\eps y-z}} \omega^\eps_i(z)\, dz
  &=\int_{E^\eps_i} \frac{1}{2\pi} \log \frac{1}{\abs{x^\eps_i+\eps y-z}} \omega^\eps_i(z)\, dz\\
  &=\frac{\kappa^\eps_i}{2\pi}\log \frac{1}{\eps \abs{y}}+\frac{1}{2\pi} \int_{E^\eps_i} \log \frac{\eps \abs{y}}{\abs{x^\eps_i+\eps y-z}} \omega^\eps_i(z)\, dz. 
\end{split}
\]
In view of \eqref{ineqDiamEepsi}, $\abs{x^\eps_i-z} \le (1+o(1))\Tilde{\rho}\eps$ when $z\in {\rm supp}(\omega^\eps_i)$ so that for sufficiently small $\eps$
\[
\left\lvert\int_{E^\eps_i} \log \frac{\abs{\eps y}}{\abs{\eps y+x^\eps_i-z}} \omega^\eps_i(z)\, dz\right\rvert\le \kappa^\eps_i \log\frac{\abs{y}}{\abs{y}-\Tilde{\rho}} + o(1). 
\]

We now prove \eqref{ineqNablavepsiDecay}. By Lemma~\ref{lemmaVanishingVorticity}, $\eps \omega^\eps_v \to 0$ in $\mathrm{L}^s(\Omega)$ for $\frac{1}{s} \ge \frac{1}{2}-p(\frac{1}{2}-\frac{1}{r})$. Choosing $s > 2$, by \eqref{eqvepsiGomega} and classical elliptic estimates, one obtains that
\[
  \int_{\Omega} \eps G(x, z) \omega^\eps_v(z)\, dz \to 0
\]
as a function of $x$ in $\mathrm{W}^{2, s}_{\mathrm{loc}}(\Omega)$ and thus in $C^1_{\mathrm{loc}}(\Omega)$. Therefore,
\[
  \int_{\Omega} \eps \nabla G(x_i^\eps+\eps y, z) \omega^\eps_v(z)\, dz \to 0
\]
uniformly in $y$ on compact subsets. 
One also has 
\[
  \int_{\Omega} \eps \nabla G(x^\eps_i+\eps y, z) \omega^\eps_j(z)\, dz
  =\eps \kappa_j^\eps \nabla G(x^\eps_i, x^\eps_j)+o(1)
\]
and
\[
 \int_{\Omega} \eps \nabla H(x^\eps_i+\eps y, z) \omega^\eps_j(z)\, dz
  =\eps \kappa^\eps_i \nabla H(x_i, x_j)+O(\eps^2). 
\]
Finally, recall that $\int_{\Omega} \omega^\eps_i=\kappa^\eps_i$ and $\int_{\Omega} (x^\eps_i-z)\omega^\eps_i(z)\, dz=0$, so that
\begin{multline*}
  \int_{\Omega} \eps \frac{x^\eps_i+\eps y -z}{\abs{x^\eps_i+\eps y-z}^2} \omega^\eps_i(z)\, dz-\kappa^\eps_i\frac{y}{\abs{y}^2}= \\
  \eps \int_{E^\eps_i} \Bigl(\frac{x^\eps_i+\eps y -z}{\abs{x^\eps_i+\eps y-z}^2}-  \frac{\eps y}{\abs{\eps y}^2}-L(\eps y) (x^\eps_i-z) \Bigr)\omega^\eps_i(z)\, dz,
\end{multline*}
where
\[
  L(a)h=\frac{\abs{a}^2h-2(a \cdot h)a}{\abs{a}^4}. 
\]
On the other hand, for $2 \abs{h} \le \abs{a}$, 
\[
  \Bigl \lvert \frac{a+h}{\abs{a+h}^2}-\frac{a}{\abs{a}^2} - L(a)h \Bigr \rvert 
  \le C \frac{\abs{h}^2}{\abs{a}^3},
\] 
so that, by \eqref{ineqDiamEepsi},
\begin{multline*}
   \Bigl\lvert \int_{\Omega} \eps \frac{x^\eps_i+\eps y -z}{\abs{x^\eps_i+\eps y-z}^2} \omega^\eps_i(z) \, dz-\kappa^\eps_i\frac{y}{\abs{y}^2}\Bigr\rvert  \\
\le \int_\Omega \eps \frac{\abs{x^\eps_i-z}^2}{\abs{\eps y}^3}\omega_i^\eps (z)\, dz
\le C \eps \frac{(\diam E^\eps_i)^2}{{\abs{\eps y}^3}}\le \frac{\Bar{C}}{\abs{y}^3},
\end{multline*}
and the lemma is proved. 
\end{proof}

\begin{lemma}
\label{lemmaLocalAsymptotics}
When $\eps$ is small, we have $k^\eps=1$. Moreover,
\[
  \kappa^\eps_1=\kappa+\frac{2\pi}{\log \frac{1}{\eps}}\Bigl(q(x^\eps_1)-\kappa H(x^\eps_1, x^\eps_1)-\frac{\kappa}{2\pi} \log \frac{1}{\rho_\kappa} \Bigr)+o(\logeps^{-1})
\]
and $v^\eps_1 \to U_\kappa$ in $\mathrm{W}^{3, r}_{\mathrm{loc}}(\R^2)$ as $\eps\to 0$. 
\end{lemma}
\begin{proof}
Set 
\[w^\eps_i(y)= v^\eps_i(y)-\frac{\kappa^\eps_i}{2\pi}\log \frac{1}{\eps}+q^\eps(x^\eps_i)-\kappa^\eps_i H(x^\eps_i, x^\eps_i)
  -\sum_{j \ne i} \kappa_j^\eps G(x^\eps_i, x^\eps_j),
\]
so that in particular
\[
  -\Delta w^\eps_i=f(v^\eps_i-q^\eps_i). 
\]
By \eqref{ineqRenormEstimate}, \eqref{ineqvepsiDecay} and classical elliptic estimates \cite{GilbargTrudinger2001}*{Theorem 9.11}, the sequence $(w^\eps_i)$ is bounded in $\mathrm{W}^{2, s}_{\mathrm{loc}}(\R^2)$ for every $s \ge 1$. By Rellich's compactness theorem, it is compact in $\mathrm{W}^{1, t}_{\mathrm{loc}}(\R^2)$ for every $1 \le t < \infty$, and therefore bounded on compact subsets. 
On the other hand, by construction, all the $v^\eps_i+q^\eps_i(x^\eps_i)-q^\eps_i$ take positive and negative value at a uniformly bounded distance from the origin, so that there exists a bounded sequence $\check{x}_i^\eps$ such that $v^\eps_i(\check{x}_i^\eps)=q^\eps_i(\check{x}_i^\eps)-q^\eps_i(x_i^\eps)$. Therefore, $v^\eps_i(\check{x}_i^\eps)$ and $w^\eps_i(\check{x}_i^\eps)$ remain bounded and we obtain that for each $i \in \{1, \dotsc, k^\eps\}$  
\[
q^\eps(x^\eps_i)-\frac{\kappa^\eps_i}{2\pi}\log \frac{1}{\eps}-\kappa^\eps_i H(x^\eps_i, x^\eps_i)
  -\sum_{j \ne i} \kappa_j^\eps G(x^\eps_i, x^\eps_j)=O(1). 
\]
This implies that 
\begin{equation}
\label{eqVorticitiesGreen}
  \frac{\kappa^\eps_i}{2\pi}\log \frac{1}{\eps} + \sum_{\substack{ j \ne i}} \kappa_j^\eps \log \frac{1}{\abs{x^\eps_i-x^\eps_j}} = \frac{\kappa}{2\pi}\log \frac{1}{\eps} +O(1),
\end{equation}
and, in view of \eqref{ineqSumVortices}, that
\[
  k_\eps \frac{\kappa}{2\pi} \log \frac{1}{\eps}\ge \sum_{1 \le i, j \le k_\eps } \frac{\kappa^\eps_i }{2\pi}\log \frac{1}{\eps} +O(1)=k_\eps \frac{\kappa}{2\pi}\log \frac{1}{\eps}+\sum_{\substack{1 \le i, j \le k^\eps \\ j \ne i}} \kappa^\eps_j\log \frac{\abs{x^\eps_i-x^\eps_j}}{\eps}+O(1). 
\]
Therefore,
\[
  \sum_{\substack{1 \le i, j \le k^\eps \\ j \ne i}} \kappa^\eps_j \log \frac{\abs{x^\eps_i-x^\eps_j}}{\eps} \le O(1),
\]
and since $\abs{x^\eps_i-x^\eps_j}/\eps \to \infty $ as $\eps \to 0$, we deduce by \eqref{ineqLowerBoundVortex} that $k^\eps\le 1$ for $\eps$ sufficiently small. 
By Lemma~\ref{lemmaNonVanisingVortex}, $k^\eps=1$.
Going back to \eqref{eqVorticitiesGreen}, we get
\[
  \kappa^\eps_1=\kappa+O\bigl(\logeps^{-1}\bigr). 
\]
Since $v_1^\eps-q^\eps_1$ is compact in $\mathrm{W}^{1, r}_{\mathrm{loc}}(\R^2)$ and $f \in C^1(\R)$, the sequence $f(v_1^\eps-q^\eps_1)$ is compact in $\mathrm{W}^{1, r}_{\mathrm{loc}}(\R^2)$. 
In view of \eqref{eqLimit}, $v^\eps_1$ is compact in $\mathrm{W}^{3, r}_{\mathrm{loc}}$. Let $v$ be one of its accumulation points. It satisfies
\[
 -\Delta v=f(v)
\]
and
\[
  \int_{\R^2} f(v)=\kappa. 
\]
Moreover, letting $\eps$ go to zero, by \eqref{ineqvepsiDecay} we obtain
\[
  v(y)=\frac{\kappa}{2\pi} \log \frac{\Tilde{\rho}}{\abs{y}}+O\Bigl(\log\bigl(1+\frac{1}{\abs{y}} \bigr)\Bigr)
\]
for some $\Tilde{\rho} \in \R$,
and
\[
  \nabla v(y)=\frac{\kappa}{2\pi}\frac{y}{\abs{y}^2}+O\Bigl(\frac{1}{\abs{y}^3}\Bigr). 
\]
By a symmetry result of L.\thinspace A.\thinspace Caffarelli and A.\thinspace Friedman \cite[Theorem 1]{CaffarelliFriedman1980} (see also \cite[Theorem 4.2]{Fraenkel2000}), $v$ is radial, and therefore 
\[
  v(y)=\frac{\kappa}{2\pi}\log \frac{\rho_\kappa}{\abs{y}}
\]
when $\abs{y} \ge \rho_\kappa$. Hence, $v=U_\kappa$. In view of \eqref{ineqvepsiDecay}, this yields
\[
  \Bigl\lvert \frac{\kappa}{2\pi}\log \frac{\rho_\kappa}{\abs{y}}+q^\eps(x^\eps_1)-\frac{\kappa^\eps_1}{2\pi}\log \frac{1}{\eps\abs{y}}-\kappa^\eps_1 H(x^\eps_1, x^\eps_1)
  \Bigr\rvert \le \kappa \log \frac{\abs{y}}{\abs{y}-R}+o(1). 
\]
First fixing $y$, this implies that
\[
 \frac{\kappa-\kappa^\eps_1}{2\pi} \log \frac{1}{\eps}=O(1),
\]
and next we deduce that for every $2\Tilde{\rho}<\abs{y}<R$,  
\[
  \Bigl\lvert \frac{\kappa}{2\pi}\log \frac{\rho_\kappa}{\eps}+q(x^\eps_1)-\frac{\kappa^\eps_1}{2\pi}\log \frac{1}{\eps}-\kappa^\eps_1 H(x^\eps_1, x^\eps_1)
  \Bigr\rvert \le \kappa \log \frac{\abs{y}}{\abs{y}-\Tilde{\rho}}+o(1),
\]
as $\eps \to 0$. We obtain the required asymptotic development of $\kappa^\eps_1$ by letting $R\to +\infty$ and choosing sufficiently large $\abs{y}$. 
\end{proof}

\subsubsection{Step 4: Global asymptotics}

\medskip

We are now going to prove that $v^\eps$ is well approximated by
\[
  \Tilde{v}^\eps=U_{\kappa^\eps_1}\Bigl(\frac{\cdot-x^\eps_1}{\eps}\Bigr)+\kappa^\eps_1\Bigl(\frac{1}{2\pi} \log \frac{1}{\eps \rho_{\kappa_1^\eps}}+H(x^\eps_1, \cdot)\Bigr).
\]

\begin{proposition}
\label{propositionAsymptoticsW21}
We have
\[
  v^\eps=\Tilde{v}^\eps+o(1)
\]
in $\mathrm{W}^{2, 1}_{\mathrm{loc}}(\Omega)$, in $\mathrm{W}^{1, 2}_0(\Omega)$, and in $\mathrm{L}^\infty(\Omega)$. 
\end{proposition}
\begin{proof}
	Choose $r>\Tilde{\rho}$ so that $E^\eps_1 \subset B(x^\eps_1, \eps r)$ when $\eps$ is small. 
	By Lemma~\ref{lemmaLocalAsymptotics}, and the invariance of the $\dot{\mathrm{W}}^{2, 1}$ semi-norm by scaling, we have
\[
  \int_{B(x^\eps_1, 2\eps r)} \abs{D^2 v^\eps-D^2 \Tilde{v}^\eps} \to 0
\]
as $\eps \to 0$. 
Define
\begin{align*}
  \Tilde{\omega}^\eps_1(x)&=\frac{1}{\eps^2}f(\Tilde{v}^\eps-q^\eps), \\
  w_v^\eps(x)&=\int_{\Omega} G(x, y) \omega^\eps_v(y)\, dy, \\
  w_r^\eps(x)&=\int_{\Omega} H(x, y) \bigl(\omega^\eps_1(y)-\Tilde{\omega}^\eps_1(y)\bigr)\, dy, \\
  w_s^\eps(x)&=\int_{\Omega} \Gamma(x-y)  \bigl(\omega^\eps_1(y)-\Tilde{\omega}^\eps_1(y)\bigr)\, dy,
\end{align*}  
where $\Gamma (x)=\frac{1}{2\pi} \log \frac{1}{\abs{x}}$, so that
$v^\eps - \Tilde{v}^ \eps=w_v^\eps+w_r^\eps+w_s^\eps$. 
Since by Lemma~\ref{lemmaVanishingVorticity}, $\omega_v \to 0$ in $\mathrm{L}^s(\Omega)$ for some $s > 1$, we have, by elliptic estimates, $w^\eps_v \to 0$ in $\mathrm{W}^{2, s}_{\mathrm{loc}}(\Omega)$. 
Next, since by \eqref{ineqDiamEepsi} $x^\eps_1$ stays away from $\partial \Omega$ and $\omega^\eps_1-\Tilde{\omega}^\eps_1 \to 0$ in $\mathrm{L}^1(\Omega)$ by Lemma~\ref{lemmaLocalAsymptotics}, we have $w_r^\eps \to 0$ in $C^\infty_{\mathrm{loc}}(\Omega)$. 
Finally, we have
\[
  D^2 w_s^\eps (x)=\int_{\Omega} D^2\Gamma(x-y) \bigl(\omega^\eps_1(y)-\Tilde{\omega}^\eps_1(y)\bigr)\, dy. 
\]
Since $\int_{\Omega} \omega^\eps_1=\int_{\Omega} \Tilde{\omega}^\eps_1=\kappa^\eps_1$, one also has
\[
  D^2 w_s^\eps (x)=\int_{B(x^\eps_1, \eps r)} \bigl(D^2\Gamma(x-y)-D^2\Gamma(x-x^\eps_1)\bigr) \bigl(\omega^\eps_1(y)-\Tilde{\omega}^\eps_1(y)\bigr)\, dy. 
\]
For every $y \in B(x^\eps_1, \eps r)$ and $x \in \Omega \setminus B(x^\eps_1, \eps 2r)$
\[
  \abs{D^2\Gamma(x-y)-D^2\Gamma(x-x^\eps_1)} \le C \frac{\abs{y-x^\eps_1}}{\abs{x-x^\eps_1}^3},
\]
so that 
\[
\abs{D^2 w_s^\eps (x)} \le \frac{C \eps}{\abs{x^\eps_1-x}^3}\Norm{\omega^\eps_1-\Tilde{\omega}^\eps_1}_{\mathrm{L}^1}. 
\]
Integrating the previous inequality we conclude
\begin{multline*}
  \int_{\Omega \setminus B(x^\eps_1, \eps 2r)} \abs{D^2 w_s^\eps (x)} \le C\eps \Norm{\omega^\eps_1-\Tilde{\omega}^\eps_1}_{\mathrm{L}^1}\int_{\R^2 \setminus B(x^\eps_1, \eps 2r)}\frac{1}{\abs{x^\eps_1-x}^3}\, dx \\
  =C\Norm{\omega^\eps_1-\Tilde{\omega}^\eps_1}_{\mathrm{L}^1}\frac{2\pi \eps}{\eps R}=o(1). 
\end{multline*}

The $\mathrm{W}^{2, 1}_\mathrm{loc}(\Omega)$ convergence implies the $\mathrm{W}^{1, 2}_\mathrm{loc}(\Omega)$ and the $\mathrm{L}^\infty_{\mathrm{loc}}(\Omega)$ convergences.
One needs then to prove the convergence in a neighbourhood of the boundary. Consider $U \subset V \subset \Omega$ open bounded sets such that $\partial \Omega \subset \Bar{U}$, $\Bar{U} \subset V$ and $\supp \omega_\eps \cap V = \emptyset$. One has
\[
\left\{
 \begin{aligned}
   -\Delta (v_\eps-\Tilde{v}_\eps)&=\omega^\eps_v && \text{in $U$},\\
  v_\eps-\Tilde{v}_\eps&=0 && \text{on $\partial \Omega$}.
 \end{aligned}
\right.
\]
Since $v_\eps-\Tilde{v}_\eps \to 0$ in $\mathrm{W}^{1,2}(V \setminus U)$ and in $\mathrm{L}^\infty (V \setminus U)$ and $\omega^\eps_v \to 0$ in $\mathrm{L}^s(\Omega)$ for some $s > 1$, one obtains by classical regularity estimates that $v_\eps-\Tilde{v}_\eps \to 0$ in $\mathrm{W}^{1,2}(U)$ and in $\mathrm{L}^\infty(U)$.
\end{proof}
%
%
%
%
%
%


\begin{corollary}
\label{corollaryAsymptotic}
When $\eps$ is small enough, $A^\eps$ is connected, $x^\eps_1=x^\eps$, $\kappa^\eps_1=\kappa^\eps$, $\partial (A^\eps_1-x^\eps_1)/\eps$ tends to $\partial B(0, \rho_\kappa)$ as a $C^2$ manifold. 
In particular, $-\Delta v^\eps=0$ in $\Omega \setminus B(x^\eps_1, 2\eps \rho_\kappa)$
and
\[
     \omega^\eps=\Tilde{\omega}^\eps+o(1)
\]
in $\mathrm{L}^1(\Omega)$. 
\end{corollary}
\begin{proof}
Assume that $y\in A^\eps\setminus B(x_1^\eps, \eps \rho_{\kappa^\eps})$. We have
\begin{equation}
\label{ineqAepsBall}
  q(y)+ \frac{\kappa}{2\pi} \log \frac{1}{\eps} < v^\eps(y) \le \frac{\kappa_1^\eps}{2\pi} \log \frac{1}{\abs{y-x^\eps_1}}+o(1),
\end{equation}
uniformly in $y$, so that $\abs{y-x^\eps_1}=O(\eps)$. One obtains then in view of Proposition~\ref{propositionAsymptoticsW21} that $(A^\eps_1-x^\eps_1)/\eps$ is connected when $\eps$ is small and the required convergence of the boundary. 
\end{proof}

\begin{corollary}
\label{corEnergy}
We have
\[
\begin{split}
\mathcal{E}^\eps(v^\eps)
=\frac{\kappa^2}{4\pi} \log \frac{1}{\eps}-\mathcal{W}(x^\eps)+\mathcal{C}+o(1). 
\end{split}
\]
\end{corollary}
\begin{proof}
First we have in view of Proposition~\ref{propositionAsymptoticsW21} and Corollary~\ref{corollaryAsymptotic},
\[
\begin{split}
  \int_{\Omega} \abs{\nabla v^\eps}^2
  &=\int_{\Omega} v^\eps \omega^\eps\\
  &=\int_{\Omega} \Tilde{v}^\eps \omega^\eps+o(1). 
\end{split}
\]
Since $\Norm{\Tilde{v}^\eps-q^\eps}_{\mathrm{L}^\infty}$ remains bounded as $\eps \to 0$, we obtain, by Proposition~\ref{propositionAsymptoticsW21}
\[
  \int_{\Omega} \abs{\nabla v^\eps}^2=\frac{1}{\eps^2}\int_{\Omega} \Tilde{v}^\eps f\bigl(\Tilde{v}^\eps-q^\eps(x^\eps)\bigr)+o(1). 
\]
Similarly, by Proposition~\ref{propositionAsymptoticsW21},
\[
\frac{1}{\eps^2}\int_{\Omega} F(v^\eps-q^\eps)=\frac{1}{\eps^2}\int_{\Omega} F\bigl(\Tilde{v}^\eps-q^\eps(x^\eps)\bigr)+o(1). 
\]
It suffices then to compute $\mathcal{E}^\eps(\Tilde{v}^\eps)$ as in the proof of Lemma~\ref{lemmaEnergyHatu}. 
\end{proof}

\subsubsection{Conclusion}We are now in position to present the

\begin{proof}[Proof of Proposition~\ref{prop:1mai} completed] It is a direct consequence of Lemma~\ref{lemmaLocalAsymptotics},  Proposition~\ref{propositionAsymptoticsW21}, Corollary~\ref{corollaryAsymptotic} and Corollary~\ref{corEnergy}. \end{proof}

\noindent and the

\begin{proof}[Proof of Theorem~\ref{thm:K1}] It is a direct consequence of the upper estimate of Corollary~\ref{cor:upper} and the asymptotic properties obtained in Proposition~\ref{prop:1mai}. 
\end{proof}

\section{Single vortices in multiply connected domains}
\label{sectionmultiply}

In this section we assume that $\Omega \subset \R^2$ is a bounded smooth multiply-connected domain; it can be written as 
\[
 \Omega = \Omega_0 \setminus \bigcup_{h=1}^m \Omega_h,
\]
where $\Omega_0, \dotsc, \Omega_m$ are bounded simply-connected domains with $\Bar{\Omega}_h \subset \Omega$ for every $h \in \{1, \dotsc, m\}$. In place of problem \eqref{problemPeps}, we consider the problem of finding $u$ and $\lambda^\eps_1, \dotsc, \lambda^\eps_m$ such that
\begin{equation}
\label{problemPepsstar}
\left\{
\begin{aligned}
-\Delta u^\eps &=\frac{1}{\eps^2} f(u^\eps - q^\eps ) & &\text{in $\Omega$, }\\
u^\eps &=  0 & &\text{on $\partial \Omega_0$},\\
u^\eps &= \lambda^\eps_h  & &\text{on $\partial \Omega_h$},\\
\int_{\partial \Omega_h} \frac{\partial u}{\partial n}&=0 & & \text{for $h \in \{1, \dotsc, m\}$}.
\end{aligned}
\right. \tag{\protect{$\mathcal{P}^\eps_*$}}
\end{equation}

The natural space to deal with this problem is the space of functions that are constant on the complement of $\Omega$:
\[
H^1_*(\Omega)=\Bigl\{u \in H^1(\Omega) \st \nabla u=0 \text{ in $\bigcup_{h=1}^m \Omega_h$}\Bigr\}.
\]
It is standard to show that solutions of \eqref{problemPepsstar} are critical points of the functional $\mathcal{E}^\eps$ defined on $H^1_*(\Omega)$ by \eqref{energyFunctional}.
We consider least energy solutions obtained by minimization of the functional on the Nehari manifold.

In order to state our result we also need the corresponding (appropriate) Green functions. Following C.\thinspace C.\thinspace Lin \cites{Lin1941, Lin1943}, we define $G_*$ as the solution of
\[
 \left\{
 \begin{aligned}
   -\Delta G(\cdot, y)&=\delta_y & & \text{in $\Omega$,}\\
   G(\cdot, y)&=0 & & \text{on $\partial \Omega_0$},\\
   G&=\lambda_h & & \text{on $\partial \Omega_h$},\\
   \int_{\partial \Omega_h} \frac{\partial G}{\partial n}&=0 & & \text{for $h \in \{1, \dotsc, m\}$}.\\
 \end{aligned}\right.   
\]
Its regular part $H_*$ is defined by
\[
 H_*(x,y)=G_*(x,y)-\frac{1}{2\pi} \log \frac{1}{\abs{x-y}}.
\]
P.\thinspace Koebe \cite{Koebe1918}*{\S 6} (see also \cite{Lin1943}*{\S 9}), defined $G_*$ in terms of the Green function for the Dirichlet problem $G$ and
the unique solutions $Z_k$ of 
\[
 \left\{
 \begin{aligned}
   -\Delta Z_k&=0 & & \text{in $\Omega$,}\\
   Z_k&=0 & & \text{on $\partial \Omega_0$},\\
   Z_k&=\delta_{kh} & & \text{on $\partial \Omega_h$ with $h \in \{1, \dotsc, m\}$.}\\
 \end{aligned}\right.   
\]
Since the $Z_k$ are linearly independent, the matrix $(\omega_{kh})_{1 \le k, h \le n}$ defined by
\[
 \omega_{kh}=\int_{\Omega} \nabla Z_k \cdot \nabla Z_h. 
\]
is invertible; let $(\omega^{kh})_{1 \le k, h \le n}$ denote its inverse. We have 
\begin{equation}
\label{eqGreenRelationship}
 G_*(x,y)=G(x,y)+\sum_{k, h=1}^m Z_k(x)\omega^{kh}Z_h(y).
\end{equation}

The Kirchhoff--Routh function in this context is defined by 
\[
 \mathcal{W}_*(x)=\frac{\kappa^2}{2}H_*(x, x)-\kappa q(x),
\]
and the various quantities $A^\eps, \omega^\eps, \kappa^\eps, x^\eps, \rho^\eps$ are still defined by \eqref{defiq}.

Theorem~\ref{thm:K1} generalizes then to 
\begin{theorem}\label{thm:K1m}
As $\eps \to 0$, we have
\[
  u^\eps=U_{\kappa^\eps} \Big(\frac{\cdot-x^\eps}{\eps}\Big)+\kappa^\eps\Bigl(\frac{1}{2\pi} \log \frac{1}{\eps \rho^\eps}+ H_*(x^\eps, \cdot)\Bigr)+o(1),
\]
\text{in $\mathrm{W}^{2, 1}_\mathrm{loc}(\Omega)$, in $\mathrm{W}^{1, 2}_0(\Omega)$, and in $\mathrm{L}^\infty(\Omega)$}, where
\[
  \kappa^\eps=\kappa+\frac{2\pi}{\log \frac{1}{\eps}}\Bigl(q(x^\eps)-\kappa H(x^\eps, x^\eps) -\frac{\kappa}{2\pi} \log \frac{1}{\rho_\kappa} \Bigr)+o(\logeps^{-1}),
\]
and
\[
   \mathcal{W}_*(x^\eps) \to \sup_{x \in \Omega} \mathcal{W}_*(x). 
\]
One also has
\[
  B(x^\eps, \Bar{r}^\eps) \subset A^\eps \subset B(x^\eps, \mathring{r}^\eps),
\] 
with $\Bar{r}^\eps=\eps \rho_\kappa+o(\eps)$ and $\mathring{r}^\eps=\eps
\rho_\kappa +o(\eps)$. Finally, 
\[
  \mathcal{E}^\eps (u^\eps)= \frac{\kappa^2}{4\pi}\log \frac{1}{\eps}-\mathcal{W}_*(x^\eps)+\mathcal{C}+o(1). 
\]
\end{theorem}
\begin{proof}
The proof of Theorem~\ref{thm:K1m} follows almost the same lines as one of Theorem~\ref{thm:K1}, so that we only mention the few adaptations. First, the functions $G$ and $H$ should be replaced by $G_*$ and $H_*$. In view of the regularity of $\Theta_h$ and of \eqref{eqGreenRelationship} this does not bring any trouble in the upper estimate nor the small scale and global asymptotics. 

Next, the proof of Theorem~\ref{thm:K1} relies on the Dirichlet boundary condition to estimate $\capa(A^\eps, \Omega)$ in \eqref{ineqCapacity}. Here, we define instead 
\[
 w^\eps=\frac{v^\eps-\max_{\partial \Omega} v^\eps}{\min_{\partial A_\eps} q^\eps - \max_{\partial \Omega}v^\eps}.
\]
For every $h \in \{1, \dotsc, m\}$, let $\Theta_h \in \mathrm{H}^1_*(\Omega)$  be the unique solution of 
\[
 \left\{
 \begin{aligned}
   -\Delta \Theta_h&=0 & & \text{in $\Omega$,}\\
   \Theta_h&=0 & & \text{on $\partial \Omega_0$},\\
   \Theta_h&=\mu_{kh} & & \text{on $\partial \Omega_h$ and $h \in \{1, \dotsc, m\}$},\\
   \int_{\partial \Omega_h} \frac{\partial \Theta_h}{\partial n}&=\delta_{hk} & & \text{for $k \in \{1, \dotsc, m\}$},
 \end{aligned}\right.   
\]
where $\mu_{kh}$ are unknown constants that are part of the problem\footnote{
This solution can be found by minimizing the functional $u \mapsto \frac{1}{2}\int_{\Omega} \abs{\nabla u}^2+u\vert_{\partial \Omega_k}$ over $\mathrm{H}^1_*(\Omega)$. (A similar problem appears in \cite[Chapter I, (3)]{BBH})}. By construction of $\Theta_k$, one has
\[
 v^\eps \vert_{\Omega_h}=\int_{\partial \Omega_h} v^\eps\frac{\partial \Theta_h}{\partial n} 
=\int_{\Omega} \nabla v^\eps \cdot \nabla \Theta_h=\int_{\Omega} \omega_\eps \Theta_h,
\]
and hence, in view of \eqref{ineqTotalVorticity},
\[
  \Norm{v^\eps}_{\mathrm{L}^\infty(\partial \Omega)} \le \max_{h \in \{1, \dotsc, m\}} \Norm{\Theta_h}_{\mathrm{L}^\infty(\Omega)}\bigl(\kappa+O(\logeps^{-1})\bigr),
\]
Therefore, 
\[
 \frac{2\pi}{\capa (A^\eps, \Omega)} \ge \frac{2\pi}{\int_\Omega \abs{\nabla w^\eps}^2}
\ge \log \frac{1}{\eps}+O(1),
\]
and one can continue as in the proof of Lemma~\ref{lemmaAreaDiameter}.
\end{proof}

\section{Single vortices in unbounded domains}
\label{sectUnbounded}

In this section, we assume that $\Omega \subset \R^2$ is an unbounded simply-connected domain whose boundary is bounded in one direction; to fix the ideas, 
\[
  ]a_0, +\infty[\times \R \subset \Omega \subset ]a_1, +\infty[ \times \R. 
\]
Our goal is to carry out an analysis similar to that of the previous section.

We assume that $q \in \mathrm{W}^{1, 1}_\textrm{loc}(\Omega)$,
\[
  \sup_{x \in \Omega} \int_{B(x, 1)} \abs{\nabla q}^r < \infty
\]
for some $r > 2$, and that 
\[
  q(x) \ge W(x_1-a_0)+d,
\]
for some $W > 0$ and $d > 0$, where $x=(x_1, x_2)$. Since $\partial \Omega$ is bounded in the $x_1$ direction, this is equivalent with requiring that 
\[
  q(x) \ge W \dist(x, \partial \Omega)+d'. 
\]
The natural space for solutions is
\[
 \mathrm{D}^{1, 2}_0(\Omega)= \{ u \in \mathrm{W}^{1, 1}_{\mathrm{loc}}(\Omega) \st \int_{\Omega} \abs{\nabla u}^2 < \infty\}.
\]

The Nehari manifold $\mathcal{N}^\eps$ and the infimum value $c^\eps$ are defined as in Proposition~\ref{prop:2.1}. The existence of a minimizer $u^\eps \in \mathcal{N}^\eps$ as in Proposition~\ref{prop:2.1} such that $\mathcal{E}^\eps(u^\eps)=c^\eps$ is no longer direct nor true because of compactness issues. 

In a first step, we derive upper bounds on $c^\eps$. Next, we perform the a priori asymptotic analysis of solutions of $\Peps$ satisfying similar upper bounds. Finally, we prove existence results in appropriate cases of $\Omega$ and $q$. 

\subsection{Upper bound on the energy}

\begin{proposition}
\label{propUnboundedUpper}
We have
\[
c^\eps \leq \frac{\kappa^2}{4\pi} \log \frac{1}{\eps} - \sup_{x\in \Omega} \mathcal{W}(x) +\mathcal{C} + o(1). 
\]
\end{proposition}
\begin{proof}
The proof goes as the proof of Corollary~\ref{cor:upper}. 
The main difference is that $q$ and $H(\Hat{x}, \cdot)$ are not bounded as in the proof of Lemma~\ref{lemmaHatuNehari}. However, since $\lim_{x \to \infty} \frac{1}{2\pi} \log \frac{1}{\abs{x-\Hat{x}}} +H(x, \Hat{x})=0$, one still has, for every $x \in \Omega$,
\[
  H(\Hat{x}, x) \le \frac{q(x)}{\kappa}+C,
\]
whence, starting from \eqref{ineqVorticitySetUpperFrac}, one obtains
\[
  \frac{\dfrac{1}{2\pi} \log \dfrac{1}{\eps}+\dfrac{1}{2\pi} \log \dfrac{\eps}{\abs{x-\Hat{x}}}+ q(x)+C}{\dfrac{\kappa}{2\pi} \log \dfrac{1}{\eps}+\frac{q(x)}{\kappa}} \ge  
\dfrac{\log \dfrac{1}{\eps}+H(\Hat{x}, \Hat{x})}{\dfrac{\kappa}{2\pi} \log \dfrac{1}{\eps}+q(\Hat{x})+\sigma}. 
\]
Since $q \ge 0$, it follows that 
\[
  \frac{1}{\kappa}+\frac{\dfrac{1}{2\pi} \log \dfrac{\eps}{\abs{x-\Hat{x}}}}{\kappa \log \dfrac{1}{\eps}}\ge \frac{1}{\kappa}+O\bigl(\logeps^{-1}\bigr),
\]
and it suffices to continue as in Lemmas~\ref{lemmaHatuNehari}~and~\ref{lemmaEnergyHatu}. 
\end{proof}

\subsection{Functional inequalities on the half-plane}

In order to perform the asymptotic analysis of the solutions and to study their existence, we first provide some useful functional type inequalities and convergence results on the half-plane $\R^2_+$ that will be used in the next section. 

\begin{proposition}
\label{ineqUnboundedIneqW}
We have for $u\in \mathrm{D}^{1, 2}_0(\R^2_+)$, 
\[
  \muleb{2}\bigl(\{ x \in \R^2_+ \st u(x) \ge Wx_1\}\bigr)\le C \int_{\R^2_+} \abs{\nabla u}^2,
\]
and, for every $p>0$,
\[
  \int_{\R^2_+} \bigl(u(x)-Wx_1\bigr)^p_+\, dx \le \frac{C}{W^2} \Bigl( \int_{\R^2_+} \abs{\nabla u}^2 \Bigr)^{1+\frac{p}{2}}. 
\]
\end{proposition}

A similar statement is proved by Yang Jianfu \cite[Lemma 4]{Yang1991} with a different proof relying on an isometry between $\mathcal{D}^{1,2}_0(\R^2_+)$ and the space of cylindrically symmetric elements of $\mathcal{D}^{1,2}_0(\R^4)$ \cite[Lemma 1]{Yang1991}.

\begin{proof}
Define $A_u=\{ x \in \R^2_+ \st u(x) \ge Wx_1\}$. First we have, by the Chebyshev and Hardy inequalities
\[
  \muleb{2}(A_u ) \le \frac{1}{W^2} \int_{\R^2_+} \frac{\abs{u(x)}^2}{\abs{x_1}^2}\, dx \le \frac{4}{W^2} \int_{\R^2_+} \abs{\nabla u}^2. 
\]
By Sobolev's inequality, it follows
\[
\begin{split}
  \int_{\R^2_+} (u(x)-Wx_1)_+^p\, dx &=\int_{A_u} (u-Wx_1)^p\, dx  \\
  &\le C \Norm{\nabla(u-Wx_1)}^p_{2} \muleb{2}(A_u) \\
  &\le \frac{C'}{W^2} \Norm{\nabla u}_{2}^2(\Norm{\nabla u}_2+W \muleb{2}(A_u)^\frac{1}{2})^p \\
  &\le \frac{C''}{W^2} \Bigl( \int_{\R^2_+} \abs{\nabla u}^2\Bigr)^{1+\frac{p}{2}}. \qedhere
\end{split}
\]
\end{proof}

As a consequence

\begin{lemma}
\label{lemmaUnboundedInequalityq}
We have for $u\in \mathrm{D}^{1, 2}_0(\R^2_+)$, 
\[
  \muleb{2}\bigl(\{ x \in \R^2_+ \st u(x) \ge q(x)\}\bigr)\le C \int_{\R^2_+} \abs{\nabla u}^2,
\]
and for every $p>0$
\[
  \int_{\R^2_+} (u-q)^p_+ \le C \Bigl( \int_{\R^2_+} \abs{\nabla u}^2 \Bigr)^{1+\frac{p}{2}}. 
\]
\end{lemma}

We also have a compactness theorem

\begin{lemma}
\label{lemmaUnboundedCompactness}
For every $p < \infty$ and $L>0$, the map $\Phi : \mathrm{D}^{1, 2}_0(\R^2_+) \to \mathrm{L}^p(\R_+\times ]-L, L[) : u \mapsto (u-Wx_1)_+$ is completely continuous. 
\end{lemma}
\begin{proof}
By Rellich's Theorem, $u \mapsto \Phi(u)\chi_{]0, \lambda[\times ]-L, L[}$ is completely continuous for every $\lambda > 0$. On the other hand,
\[
\int_{]\lambda, +\infty[\times ]-L, L[ } \hspace{-2em}(u(x)-Wx_1)_+^p\, dx \le \frac{C}{\lambda}\int_{]\lambda, +\infty[\times ]-L, L[} \hspace{-2em}(u(x)-\tfrac{W}{2}x_1)_+^{p+1}\, dx \le \frac{C}{\lambda} \Norm{\nabla u}_2^{p+3},
\]
therefore, on every bounded subset of $\mathrm{D}^{1, 2}_0(\R^2_+)$, $\Phi$ is a uniform limit of completely continuous maps. The conclusion follows. 
\end{proof}

\begin{lemma}\label{campeones}
Let $(u_n) \subset \mathrm{D}^{1, 2}_0(\R^2_+)$. If $(u_n)$ is bounded in $\mathrm{D}^{1, 2}_0(\R^2_+)$ and 
\[
\sup_{y\in \R} \int_{\R_+ \times ]y-1, y+1[} (u_n-Wx_1)_+^p \to 0,
\]
then 
\[
  \int_{\R^+_2} (u_n-Wx_1)_+^s \to 0,
\]
for every $s>0$. 
\end{lemma}

This kind of result was first obtained by P.-L.\thinspace Lions \cite[Lemma I.1]{Lions1984}. The idea of our proof comes from V.\thinspace Coti Zelati and P.\thinspace Rabinowitz \cite{CotiZelatiRabinowitz1992}. 

\begin{proof}
By the Gagliardo--Nirenberg inequality \cite[p.\thinspace 125]{Nirenberg1959},
\[
\begin{split}
  \int_{\R_+\times ]y-1, y+1[}(u_n-Wx_1)_+^{p+2} 
  \le C\int_{\R_+\times ]y-1, y+1[}\hspace{-4em}& (u_n-Wx_1)_+^p \\
&\times \int_{\R_+ \times ]y-1, y+1[}\hspace{-4em} (\abs{\nabla (u_n-Wx_1)_+}^2+\abs{(u_n-Wx_1)_+}^2). 
\end{split}
\]
Integrating with respect to $y\in \R$, one obtains
\[
\begin{split}
  \int_{\R^2_+}(u_n-Wx_1)_+^{p+2}
\le C\biggl(\sup_{y \in \R} \int_{\R_+\times ]y-1, y+1[}\hspace{-4em}& (u_n-Wx_1)_+^p\biggr)\\
&\times\int_{\R_+^2} \bigl(\abs{\nabla (u_n-Wx_1)_+}^2+\abs{(u_n-Wx_1)_+}^2\bigr). 
\end{split}
\]
Since by Lemma~\ref{ineqUnboundedIneqW}
\[
\int_{\R_+^2} \abs{\nabla (u_n-Wx_1)_+}^2+\abs{(u_n-Wx_1)_+}^2
\le C \bigl(\Norm{\nabla u_n}_{2}^2+\Norm{\nabla u_n}_{2}^4\bigr),
\]
$(u_n-Wx_1) \to 0$ in $\mathrm{L}^{p+2}(\R^2_+)$. By Lemma~\ref{ineqUnboundedIneqW}, the general case $s \ne p+2$ follows by interpolation. 
\end{proof}

\subsection{Asymptotic behavior of solutions}
\label{sectionUnboundedAsymptotics}
In this section, we assume that $(v^\eps)$ is a sequence of solutions to $\Peps$ satisfying \eqref{assumptEnergyUpperbound}. We shall prove 
\begin{proposition}
\label{prop:1maiUnbounded}
Proposition~\ref{prop:1mai} holds under the assumptions on $\Omega$ and $q$ of this section.
\end{proposition}

\subsubsection{Step 1: First quantitative properties of the solutions}

We first have the counterpart of Proposition~\ref{propositionEstimatesueps}
\begin{proposition}
\label{propositionUnboundedEstimatesueps} 
The estimates \eqref{ineqMuAeps}, \eqref{ineqVortexEnergy}, \eqref{ineqVortexPotential}, 
\eqref{eq:2etoiles} and \eqref{ineqTotalVorticity} hold for some constant $C$ independent of $\eps$. 
\end{proposition}
\begin{proof}
The proof of Proposition~\ref{propositionEstimatesueps} provides the estimates \eqref{ineqVortexEnergy}, \eqref{ineqVortexPotential}, \eqref{eq:2etoiles} and \eqref{ineqTotalVorticity} without any modification. The inequality \eqref{ineqMuAeps} needs a little more work, since its proof in Proposition~\ref{propositionEstimatesueps} relies on the Poincar\'e inequality. In the present setting, we replace it by the Chebyshev inequality and Lemma~\ref{lemmaUnboundedInequalityq}
\[
  \begin{split}
\muleb{2}(\{ x \in \Omega \st v^\eps(x) \ge q(x)+\tfrac{\kappa}{2\pi} \log \tfrac{1}{\eps}\})
  &\le \frac{1}{(\frac{\kappa}{2\pi} \log \frac{1}{\eps})^4} \int_{\Omega} (v^\eps-q)_+^4 \\
  &\le \frac{C}{\logeps^4} \logeps^3=C \logeps^{-1}. \qedhere
  \end{split}
\]
\end{proof}

\subsubsection{Step 2: Structure of the vorticity set}
As previously, we consider the connected components of $(A^\eps_i)_{i \in I_\eps}$ of $A_\eps$.

\begin{lemma}
\label{lemmaUnboundedAreaDiameter}
If $\eps > 0$ is sufficiently small, we have
for every $i \in I^\eps$,
\begin{equation}
\label{ineqUnboundedVorticityDiameter}
  \diam(A^\eps_i) \le C \eps \frac{\dist(A^\eps_i, \partial \Omega)}{e^{2W \dist(A^\eps, \partial\Omega)}}. 
\end{equation}
Moreover, if for every $x \in \Omega$,  one defines
\[
  A^\eps_x=\bigcup \Bigl\{ A^\eps_i \st  B(x, \tfrac{1}{2}\dist(x, \partial \Omega) +1) \cap A^\eps_i \ne \emptyset \Bigr\},
\]
then
\[
  \muleb{2}(A^\eps_x) \le C \eps^2e^{-\mu \dist(x, \partial \Omega)}. 
\]
\end{lemma}
\begin{proof}
Let  
\[
  w=\frac{v^\eps}{\min_{\partial A^\eps_i}q^\eps}. 
\]
Proceeding as in \eqref{ineqCapacity}, we obtain, using once more Proposition~\ref{propositionCapacityArea}
\[
   \frac{2\pi (\tfrac{\kappa}{2\pi}\log \tfrac{1}{\eps}+W\dist(A^\eps_i, \Omega)+d')^2}{\frac{\kappa^2}{2\pi} \log \frac{1}{\eps}}\le \log \biggl(C\bigl(1+\dist(A^\eps_i, \partial \Omega)\bigr)\Bigl(1+\frac{\dist(A^\eps, \partial \Omega)}{\diam A^\eps_i}\Bigr)\biggr). 
\]
Therefore,
\[
  \frac{1}{\eps}\le C\frac{1+\dist(A^\eps, \partial \Omega)}{e^{2W(\dist(A_\eps, \partial \Omega)-1)}}  \Bigl(1+\frac{\dist(A^\eps_i, \partial \Omega)}{\diam A^\eps_i}\Bigr),
\]
from which \eqref{ineqUnboundedVorticityDiameter} follows. 

Consider now $A^\eps_x$. By \eqref{ineqUnboundedVorticityDiameter}, $A^\eps_x \subset B(x, \frac{2}{3} \dist(x, \partial \Omega)+1)$ when $\eps$ is small enough, so that 
\[
  \frac{2\pi}{\capa_\Omega (A^\eps_x)} \ge \frac{\bigl(\tfrac{\kappa}{2\pi}\log \tfrac{1}{\eps} + \frac{W}{3}\dist(x, \partial \Omega)+d'\bigr)^2}{\frac{\kappa^2}{4\pi} \log \frac{1}{\eps}}. 
\]
By Proposition~\ref{propositionCapacityLocalArea}, we obtain 
\[
  \muleb{2}(A^\eps_x) \le C\bigl(\dist(x, \partial \Omega)+1\bigr)^2 \eps^2 e^{-\frac{4W}{3}\dist(x, \partial \Omega)} \le C \eps^2 e^{-\mu\dist(x, \partial \Omega)}. \qedhere
\]
\end{proof}
\begin{remark}
A slightly more careful proof shows that one can take any $\mu < W/2$, provided $C$ is large enough. 
\end{remark}

The next Lemma, counterpart of Lemma~\ref{lemmaVortexSplit}, insures that essential vortices are not too far from the boundary. 

\begin{lemma}
\label{lemmaUnboundedVortexSplit}
There exists constants $\gamma, C, c>0$, such that, when $\eps$ is small enough:
If \eqref{eqSplitVortices} holds, we have \eqref{ineqLowerBoundArea}, \eqref{ineqLowerBoundDiam}, \eqref{ineqLowerBoundDistance}, \eqref{ineqLowerBoundVortex} and
\[
\dist(A^\eps_i, \partial \Omega)\le C,
\]
while if \eqref{eqSplitVortices} does not hold, then \eqref{ineqfsVanishing} holds. 
\end{lemma}
\begin{proof}
The proof follows essentially the one of Lemma~\ref{lemmaVortexSplit}. The inequality  \eqref{ineqLowerBoundDiam} follows immediately from \eqref{ineqLowerBoundArea} and \eqref{ineqUnboundedVorticityDiameter}. 
\end{proof}

As in the case of a bounded domain, the vorticity set can be split into a vanishing vorticity set $V^\eps$ and an essential one $E^\eps$, defined by \eqref{eqDefVeps} and \eqref{eqDefEeps}. 
Since the gradient of $q$ is only locally integrable, Lemma~\ref{lemmaVanishingVorticity} only gives local information. 

\begin{lemma}
\label{lemmaUnboundedVanishing}
For every $s \ge 1$, we have
\[
   \sup_{x \in \Omega}	\Norm{\omega^\eps_v}_{\mathrm{L}^s(B(x, 1))} = o(\eps^{p(1-\frac{2}{r})-2(1-\frac{1}{s})}). 
\]
In particular, if $\frac{1}{s} \ge 1-p(\frac{1}{2}-\frac{1}{r})$, then $\omega^\eps_v \to 0$ in $\mathrm{L}^s_{\mathrm{loc}}(\Omega)$. 
\end{lemma}

\subsubsection{Step 3: Small scale asymptotics}
For the small scale asymptotics, one first note that Lemma~\ref{lemmaSmallScaleLocalEstimates} still holds. Indeed, the only step that relied on the boundedness of $\Omega$ was
\eqref{ineqGomegav}. For every $\rho>0$, regularity estimates still yields for $x \in B(x^\eps_i, \frac{1}{2})$
\[
  \Bigl\lvert\int_{B(x^\eps_i, \rho)} G(x, y)\omega^\eps_v(y)\, dy \Bigr\rvert\le C \Norm{\omega^\eps_v}_{\mathrm{L}^s(B(x^\eps_i, 2\rho))},
\]
and the conclusion follows from Lemma~\ref{lemmaUnboundedVanishing}. 
On the other hand, since $\Omega$ is contained in a half-plane, by comparing its Green function by the Green function of a half-plane, we have
\[
  G(x, y) \le \frac{1}{2\pi} \log \Bigl(1+\frac{C\bigl(1+ \dist(x, \partial \Omega)\bigr)}{\abs{x-y}}\Bigr). 
\]
Since $\dist(x^\eps_i, \partial \Omega)$ is bounded, 
we have, for every $x \in B(x^\eps_i, 1)$,   
\[
  \int_{\Omega \setminus B(x^\eps_i, \rho)} G(x, y)\omega^\eps_v(y)\, dy \le \frac{\kappa^\eps}{2\pi} \log \Bigl(1+\frac{C}{\rho}\Bigr) \to 0,
\]
as $\rho \to \infty$, uniformly in $\eps > 0$. 

Lemma~\ref{lemmaSmallScaleLocalEstimates} being established, the proof of Lemma~\ref{lemmaLocalAsymptotics} also adapts straightforwardly.

\subsubsection{Step 4: Global asymptotics}
For Proposition~\ref{propositionAsymptoticsW21}, one obtains a little more than the $\mathrm{W}^{2, 1}_{\mathrm{loc}}(\Omega)$ convergence. Setting $\Omega_\delta=\{ x \in \Omega \st \dist(x, \partial \Omega)> \delta\}$, one has

\begin{proposition}
We have
\[
  v^\eps=\Tilde{v}^\eps+o(1)
\]
in $\mathrm{W}^{2, 1}_{\mathrm{loc}}(\Omega_\delta)$ for every $\delta > 0$, in $\mathrm{W}^{1, 2}_0(\Omega)$, and in $\mathrm{L}^\infty(\Omega)$. 
\end{proposition}
\begin{proof}
One defines $\Tilde{\omega}^\eps_1$ and $w^\eps_v$, and $w^\eps_s$ as in the proof of Proposition~\ref{propositionAsymptoticsW21}.
One defines
\begin{align*}
 w^\eps_r(x)&=\int_{\Omega} \Bigl(H(x,y)-\frac{1}{4\pi}\log (\abs{x-y}^2+4x_1y_1)\Bigr)  \bigl(\omega^\eps_1(y)-\Tilde{\omega}^\eps_1(y)\bigr)\, dy,\\
 w^\eps_h(x)&=\int_{\Omega} \frac{1}{4\pi}\log (\abs{x-y}^2+4x_1y_1)  \bigl(\omega^\eps_1(y)-\Tilde{\omega}^\eps_1(y)\bigr)\, dy.
\end{align*}
Recalling that $0 < c \le \dist(x^\eps_1, \partial \Omega) \le C$, one treats the terms $w^\eps_v$, $w^\eps_s$ and $w^\eps_v$ as in the proof of Proposition~\ref{propositionAsymptoticsW21}; the term $w^\eps_s$ is treated similarly to the term $w^\eps_s$. The proof of the convergences up to the boundary follows then as in the proof of  Proposition~\ref{propositionAsymptoticsW21}.
\end{proof}

For Corollary~\ref{corollaryAsymptotic}, we have, instead of \eqref{ineqAepsBall}, 
\[
  q(y)+ \frac{\kappa}{2\pi} \log \frac{1}{\eps} < v^\eps(y^\eps) \le \frac{\kappa_1^\eps}{2\pi} \log \Bigl(1+\frac{C\dist(x^\eps_1, \partial \Omega)}{\abs{y-x^\eps_1}}\Bigr)+O(1). 
\]
The remaining part of the proof carries over identically since $\dist(x^\eps_1, \partial \Omega)$ remains bounded as $\eps \to 0$. Corollary~\ref{corEnergy} also follows without any modification. 

\subsection{Existence of solutions}

In this section we present sufficient conditions for the existence of a minimizer for $c^\eps$. 

\medskip

Assume that $\Omega \subset ]a_0, +\infty[ \times \R$ is a Lipschitz domain, and that 
\begin{equation}
\label{condPerturbation}
  \lim_{t \to +\infty} \inf \{ x_1 \in \R \st \exists x_2 \in \R, (x_1, x_2) \in \Omega \text{ and } \abs{x_2} \ge t\}=0. 
\end{equation}
Assume also that there exist $\Hat{W}, \Hat{d}>0$ such that  
and 
\[
  \lim_{t \to +\infty} \inf_{\abs{x_2} > t} \frac{q(x)-\Hat{W}x_1-\Hat{d}}{1+\abs{x_1}} \ge 0. 
\]
We define
\[
\Hat{\mathcal{E}}^\eps (u)= \frac{1}{2} \int_{\R^2_+} \abs{\nabla u}^2-\frac{1}{\eps^2}\int_{\R^2_+} F(u-\Hat{W}x_1-\Hat{d}) 
\]
and the minimax level
\[
\Hat{c}^\eps=\inf_{u \in \mathrm{D}^{1, 2}_0(\R^2_+)}\max_{t>0} \Hat{\mathcal{E}}^\eps(tu). 
\]

We first recall and investigate about the case where $q$ is affine and $\Omega$ is the half-plane. In this case, by definition, $c^\eps=\Hat{c}^\eps$. 

\begin{theorem}[Yang \cite{Yang1991}]
\label{thmYang}
If $\Omega=\R^2_+$ and $q(x)=Wx_1+d$, then problem \eqref{problemPeps} admits a solution $u \in \mathrm{D}^{1, 2}_0(\Omega)$. 
\end{theorem}

The proof in \cite{Yang1991} allows to state that 

\begin{proposition}\label{prop:al}
The critical level $c^\eps=\Hat{c}^\eps$ depends continuously on $W$ and $d$. 
\end{proposition}
\begin{proof}[Sketch of the proof]
We can assume without loss of generality that $\eps=1$ and skip any reference to it. Given converging sequences $W_n \to W$ and $d_n \to d$, we set  
\[
 \mathcal{E}_n(u)=\frac{1}{2} \int_{\R^2_+} \abs{\nabla u}^2-\int_{\R^2_+} F(u-W_nx_1-d_n)
\]
By Theorem~\ref{thmYang}, $\mathcal{E}$ and $\mathcal{E}_n$ possess (some) ground-states $u$ and $u_n$, for which we set $c_n=\mathcal{E}(u_n)$. 
There exist $\tau_n \to 1$ such that  $\dualprod{d\mathcal{E}_n(\tau_n u)}{\tau_n u}=0$. Therefore, 
\[
  c_n \le \mathcal{E}_n(\tau_n u) \to \mathcal{E} (u)=c. 
\]
This implies that $c$ is upper semi-continuous. In particular, since 
\[
  \Bigl(\frac{1}{2}-\frac{1}{p+1}\Bigr)\Norm{\nabla u_n}^2\le \mathcal{E}_n(u_n)
\]
the sequence $(u_n)$ is bounded in $\mathrm{D}^{1, 2}_0(\R^2_+)$. Choosing $\check{W}=\inf W_n>0$, we obtain 
by Proposition~\ref{ineqUnboundedIneqW}
\begin{multline*}
  \Bigl(\int_{\R^2_+} (u_n-\tfrac{1}{2}\check{W}x_1)^{p+1}_+\Bigr)^\frac{2}{p+3} 
  \le \int_{\R^2_+} \abs{\nabla u_n}^2 
  \le \int_{\R^2_+} u_n f(u_n-W_nx_1-d_n) \\
  \le \int_{\R^2_+} u_n f(u_n-\check{W}x_1) 
  \le C \int_{\R^2_+} (u_n-\tfrac{1}{2}\check{W}x_1)^{p+1}_+,
\end{multline*}
so that $(u_n-\frac{1}{2} \check{W} x_1)_+ \not \to 0$ in $L^{p+1}(\R^2_+)$. By Lemma~\ref{campeones}, up to translation in the $x_2$ direction, we have $(u_n-\frac{1}{2} \check{W} x_1)_+ \not \to 0$ in $L^{p+1}(\R_+\times ]-1, 1[)$. Hence, there exists $0\neq v\in \mathrm{D}^{1, 2}_0(\R^2_+)$ such that $u_n \weakto v$ in $\mathrm{D}^{1, 2}_0(\R^2_+)$ and $u_n \to v$ almost everywhere and in $L^r_{\mathrm{loc}}(\R^2_+)$ for $r \ge 1$. In particular,  $d\mathcal{E}(v)=0$ and by Fatou's Lemma, we have
\[
\begin{split}
  c &\le \mathcal{E}(v)=\int_{\R^2_+} (W x_1+d)^p (v-Wx_1-d)+(\tfrac{1}{2}-\tfrac{1}{p+1})(v-Wx_1-d)_+^{p+1}\\
  &\le \liminf_{n \to \infty}\int_{\R^2_+} (W_n x_1+d)^p (u_n-Wx_1-d)+(\tfrac{1}{2}-\tfrac{1}{p+1})(u_n-Wx_1-d)_+^{p+1}\\
  &= \liminf_{n \to \infty} \mathcal{E}_n(u_n)=\liminf_{n \to \infty} c_n. \qedhere
\end{split}
\]
\end{proof}

\begin{proposition}
\label{propositionPS}
If 
\[
  c^\eps < \Hat{c}^\eps
\]
then there exists $u_\eps \in \mathrm{D}^{1, 2}_0(\Omega)$ such that $d\mathcal{E}^\eps(u_\eps)=0$ and $\mathcal{E}^\eps(u_\eps)=c^\eps$. 
\end{proposition}
\begin{proof}
We use the same strategy as P.\thinspace Rabinowitz \cite{Rabinowitz1992} for the nonlinear Schr\"odinger equation on $\R^N$. 

The minimization problem can be reformulated as a mountain-pass problem (see, e.g.\ \cite[Chapter 4]{Willem1996}). By Ekeland's variational principle, there exists a sequence $(u_n) \subset \mathrm{D}^{1, 2}_0(\Omega)$ such that $d\mathcal{E}^\eps(u_n) \to 0$ and $\mathcal{E}^\eps(u_n) \to c^\eps$, see \cite[Theorem 4.3]{MawhinWillem} or \cite[Theorem 1.15]{Willem1996}. We have 
\[
  \Bigl(\frac{1}{2}-\frac{1}{p+1}\Bigr)\Norm{\nabla u}_{\mathrm{L}^2}^2\le \mathcal{E}^\eps(u_n)-\dualprod{\mathcal{E}^\eps(u_n)}{u_n} \to c^\eps,
\]
so that $(u_n)$ is bounded in $\mathrm{D}^{1, 2}_0(\Omega)$. There exists $u \in \mathrm{D}^{1, 2}_0(\Omega)$ such that, up to a subsequence, $u_n \weakto u$ $\mathrm{D}^{1, 2}_0(\Omega)$. By Rellich's Theorem, for every $\varphi \in \mathrm{D}^{1, 2}_0(\Omega)$, 
$\dualprod{d\mathcal{E}^\eps(u_n)}{\varphi} \to \dualprod{d\mathcal{E}^\eps(u)}{\varphi}$,
so that $d\mathcal{E}^\eps(u)=0$. If $u \ne 0$, then $u \in \mathcal{N}^\eps$ and by Fatou's Lemma
\[
 \begin{split}
   \mathcal{E}^\eps(u)&=\frac{1}{\eps^2}\int_{\Omega} \frac{f(u-q^\eps)u}{2}-F(u-q^\eps)\\
   &=\frac{1}{2\eps^2}\int_{\Omega} q^\eps(u-q^\eps)_+^{p}+(1-\tfrac{2}{p+1}) (u-q^\eps)_+^{p+1} \\
   &\le \liminf_{n \to \infty} \frac{1}{2\eps^2}\int_{\Omega} q^\eps(u_n-q^\eps)_+^{p}+(1-\tfrac{2}{p+1})(u_n-q^\eps)_+^{p+1}\\
   &= \liminf_{n \to \infty} \mathcal{E}^\eps(u_n)-\tfrac{1}{2}\dualprod{d\mathcal{E}^\eps(u_n)}{u_n}=c^\eps,
 \end{split}
\]
so that $u$ fits the claim. 

Otherwise, for any $\delta < \min( \Hat{W}, \Hat{d})$, let $R > 0$ be such that 
\[
  -\delta \le  \inf \bigl\{ s \in \R \st \exists r \in \R, (s, r) \in \Omega \text{ and } \abs{s}\ge R\bigr\},
\]
and,
\begin{equation}
\label{ineqqdelta}
  q(x)\ge \Hat{q}_\delta(x):=(\Hat{W}-\delta)x_1+\Hat{d}-\delta \qquad \text{if $\abs{x_2} \ge R$}. 
\end{equation}
We have, for $\Omega_R = \{ x \in \Omega \st |x_2|\geq R\}$, and in view of Lemma~\ref{lemmaUnboundedCompactness}, 
\begin{equation}
\label{ineqOmegaR}
\begin{split}
  c^\eps&=\lim_{n \to \infty}\mathcal{E}^\eps(u_n)-\dualprod{d\mathcal{E}^\eps(u_n)}{u_n}\\
  &\le \liminf_{n \to \infty} \frac{1}{\eps^2}\int_{\Omega} u_n (u_n-q)_+^p 
  = \liminf_{n \to \infty} \frac{1}{\eps^2}\int_{\Omega \setminus \Omega_R} u_n (u_n-q)_+^p \\
  &\le C \liminf_{n \to \infty} \frac{1}{\eps^2}\int_{\Omega_R} \Bigl(u_n-\frac{q}{1+\delta}\Bigr)_+^{p+1}
  \le C \liminf_{n \to \infty} \frac{1}{\eps^2}\int_{\Omega_R} (u_n-\Hat{q}_\delta)_+^{p+1}. 
\end{split}
\end{equation}

Let $\psi \in C^\infty(\R)$ such that $\supp \psi \subset [-2\delta, -\delta]$, $\psi (t)=0$ for $t \le -2 \delta$ and $\psi(t)=1$ for $t \ge -\delta$. We set $\varphi(x_1, x_2)=\psi(x_1)$. Note that $\supp \nabla \varphi \cap \Bar{\Omega}$ is compact, so that by Rellich's Theorem,
\[
  \int_{\Omega} \abs{\nabla \varphi}^2\abs{u_n}^2 \to 0,
\]
and therefore, defining $v_n=\varphi u_n$, 
\[
  \int_{\Omega} \abs{\nabla v_n}^2=\int_{\Omega} \abs{\nabla u_n}^2+o(1). 
\]
For every $\tau > 0$
\begin{multline*}
  \max_{\theta > 0} \mathcal{E}(\theta u_n)
  \ge \mathcal{E}(\tau u_n)
  =\Hat{\mathcal{E}}_{\delta}(\tau v_n)+\frac{\tau^2}{2}\int_{\Omega} \abs{\nabla u_n}^2-\abs{\nabla v_n}^2 \\
+\frac{1}{\eps^2}\int_{\Omega} F(\tau v_n-\Hat{q}_\delta)-F(\tau u_n-q). 
\end{multline*}
Choose now $\tau_n$ such that $\Hat{\mathcal{E}}_\delta(\tau_n v_n)=\sup_{\tau > 0}\Hat{\mathcal{E}}_\delta(\tau v_n)$. If $\tau_n \ge 1$, we have,
\begin{multline*}
   \tau_n^2 \int_{\Omega} \abs{\nabla v_n}^2 = \frac{1}{\eps^2}\int_{\Omega} \tau_n v_n f(\tau_n v_n-\Hat{q}_\delta) \ge \tau_n^{p+1}\frac{1}{\eps^2}\int_{\Omega} (v_n-\Hat{q}_\delta)_+^{p+1} \\
  \ge \tau_n^{p+1}\frac{1}{\eps^2}\int_{\Omega_R} (v_n-\Hat{q}_\delta)^{p+1}_+
  = \tau_n^{p+1} \frac{1}{\eps^2}\int_{\Omega_R} (u_n-\Hat{q}_\delta)^{p+1}_+,
\end{multline*}
so that by \eqref{ineqOmegaR} we obtain 
\[
\tau_n \le \max\Biggl(1, \biggl( \frac{\int_{\Omega} \abs{\nabla v_n}^2}{\int_{\Omega_R} (u_n-\Hat{q}_\delta)^{p+1}_+} \biggr)^\frac{1}{p-1}\Biggr),
\]
and the quantity on the right-hand side is bounded in view of \eqref{ineqOmegaR}. 
This implies that $\tau_n v_n \weakto 0$ and $\tau_n u_n \weakto 0$ in $D^{1, 2}(\Omega)$, and by Lemma~\ref{lemmaUnboundedCompactness}, that 
\[
\int_{\Omega \setminus \Omega_R} F(\tau_n v_n-\Hat{q}_\delta)-F(\tau_n u_n-q) \to 0, \qquad\text{as }n\to +\infty. 
\]
On the other hand, by \eqref{ineqqdelta}, $\Hat{q}_\delta \le q$ in $\Omega \setminus \Omega_R$, and 
\[
 \int_{\Omega_R} F(\tau_n v_n-\Hat{q}_\delta)-F(\tau_n u_n-q) 
= \int_{\Omega_R} F(\tau_n u_n-\Hat{q}_\delta)-F(\tau_n u_n-q)\ge 0. 
\]

Hence,
\[
  \liminf_{n \to \infty} \mathcal{E}(u_n) \ge \liminf_{n \to \infty} \Hat{\mathcal{E}}_\delta(\tau_n v_n) \ge \Hat{c}_\delta := \inf_{v \in \mathrm{D}^{1, 2}_0(]-2\delta, +\infty[\times \R)} \Hat{\mathcal{E}}_\delta(v), 
\]
and the conclusion follows from Proposition~\ref{prop:al}, sending $\delta$ to zero. 
\end{proof} 

From Proposition~\ref{propositionPS}, we derive

\begin{theorem}
\label{theoremExistenceLevels}
If 
\[
  \sup_{x \in \Omega}\frac{\kappa^2}{2} H(x, x)-\kappa q(x)  > \frac{\kappa^2}{4\pi} \Bigl(\log \frac{\kappa}{2\pi \Hat{W}}-1\Bigr)-\kappa \Hat{d},
\]
then, if $\eps$ is sufficiently small, there exists $u^\eps \in \mathrm{D}^{1, 2}_0(\Omega)$ such that $d\mathcal{E}^\eps(u^\eps)=0$ and $\mathcal{E}^\eps(u^\eps)=c^\eps$. 
\end{theorem}
\begin{proof}
By Proposition~\ref{propUnboundedUpper}, we have
\[
  c^\eps \le \frac{\kappa^2}{4\pi} \log \frac{1}{\eps}-\sup_{x\in \Omega}\Bigl(\frac{\kappa^2}{2} H(x, x)-\kappa q(x)
\Bigr) +\mathcal{C} + o(1). 
\]
On the other hand, in view of Theorem~\ref{thmYang}, $\Hat{\mathcal{E}}^\eps$ possesses a ground-state whose energy is bounded by $\frac{\kappa^2}{4\pi}\log \frac{1}{\eps}+O(1)$. It follows from Proposition~\ref{prop:1maiUnbounded}  applied to these ground-states  that
\[
\begin{split}
	\Hat{c}^\eps&=\frac{\kappa^2}{4\pi} \log \frac{1}{\eps}-\sup_{x\in \R^2_+} \Bigl(\frac{\kappa^2}{4\pi} \log 2x_1-\kappa (\Hat{W}x_1+\Hat{d})\Bigr)	 +\mathcal{C} + o(1) \\
 &=\frac{\kappa^2}{4\pi} \log \frac{1}{\eps}- \Bigl( \frac{\kappa^2}{4\pi} \Bigl(\log \frac{\kappa}{2\pi \Hat{W}}-1\Bigr)-\kappa \Hat{d} \Bigr) +\mathcal{C} + o(1). 
\end{split}
\]
Therefore, when $\eps$ is small enough, $c^\eps < \Hat{c}^\eps$, and the conclusion follows from Proposition~\ref{propositionPS}. 
\end{proof}

\section{Pair of vortices in bounded domains}
\label{sectionVortexPair}

In this section, $\Omega\subset \R^2$, $f : \R \to \R$ and $q: \Omega \to \R$ are as in Section~\ref{sectionSingleVortex}. For $\Eps=(\eps_+, \eps_-) >0$, $\kappa_+>0$ and $\kappa_-<0$
given, 
and consider solutions of the boundary value problems   
\[
\left\{
\begin{aligned}
-\Delta u^\Eps &= \frac{1}{\eps_+{}^2} f(u^\Eps - q^\eps_+) - \frac{1}{\eps_-{}^2} f(q^\eps_-
-u^\Eps) &
 & \text{in $\Omega$}, \\
u^\Eps &=  0 & &\text{on $\partial \Omega$},
\end{aligned}
\right. \tag{\protect{$\mathcal{Q}^\Eps$}}
\label{Qeps}
\]
where $q^\eps_\pm = q+ \frac{\kappa_\pm}{2\pi} \log \frac{1}{\eps_\pm}$.

We consider are the least energy nodal solutions of \eqref{Qeps} obtained by minimizing the energy functional
\[
\mathcal{E}^\Eps(u)=  \int_{\Omega} \Bigl(\frac{|\nabla u|^2}{2} -
\frac{1}{\eps_+{}^2}F(u-q^\eps_+) -\frac{1}{\eps_-{}^2}F(q^\eps_- -u) \Bigr) 
\]
over the natural constraint given by the nodal Nehari set
\[
\mathcal{M}^\Eps = \left\{ u\in H^1_0(\Omega) \ : \ u_+\neq 0,  u_- \ne 0, \ \langle
d\mathcal{E}^\Eps(u), u_+\rangle = \langle
d\mathcal{E}^\Eps(u), u_-\rangle =0\right\}. 
\] 
It is a standard \cites{CastroCossioNeuberger,BartschWethWillem,BartschWeth2003,BartschWeth2005} to
prove the
\begin{proposition}
Assume that $q^\eps_+$ is positive on $\Omega$ and $q^\eps_-$ is
negative on $\Omega$, so that $\Meps \neq \emptyset$, and define
\[
d^\Eps = \inf_{u\in \mathcal{M}^\Eps} \mathcal{E}^\Eps(u). 
\]
There exists $u^\Eps \in \mathcal{M}^\Eps$ such that $\mathcal{E}^\Eps(u^\Eps)=d^\Eps$,
and $u^\Eps$ is a nonnegative solution of $\Qeps$. 
\end{proposition}

Our focus is the asymptotics of $u^\Eps$ for a sequence $\Eps \to (0, 0)$. We assume that $0 < c < \frac{\log \eps_+}{\log \eps_-} < C < \infty$, and we will write $\logEps$ instead of $\log \eps_+$ or $\log \eps_-$ in asymptotic expansions. 

We extend the definition of $U_\kappa$ given by \eqref{Ukappa} for $\kappa < 0$ by 
$U_\kappa=-U_{-\kappa}$ and $\rho_{\kappa}=\rho_{-\kappa}$. One still has, when $\abs{x}$ is large enough, $U_\kappa(x)=\frac{\kappa}{2\pi}\log \frac{\rho_\kappa}{\abs{x}}$.
We also set
\[
 \mathcal{C}_\pm=
\frac{\kappa_\pm^2}{4\pi} \log \rho_{\kappa_\pm} +
\int_{B(0, \rho_{\kappa_\pm})}\Bigl(\frac{|\nabla U_{\rho_{\kappa_\pm}}|^2}{2} - \frac{U_{\rho_{\kappa_\pm}}^{p+1}}{p+1}\Bigr). 
\]
The Kirchhoff--Routh function $\mathcal{W}$ is defined for $(x_+, x_-)\in \Omega^2_*=\{ (y_+, y_-) \in \Omega \st y_+ \ne y_-\}$ by
\[
\begin{split}
 \mathcal{W}(x_+,x_-)= \,
&\frac{\kappa_+^2}{2}H(x_+, x_+) + \frac{\kappa_-^2}{2}H(x_-, x_-) +
\kappa_+\kappa_- G(x_+, x_-)\\
&
 -\frac{\kappa_+}{2\pi} q(x_+) -\frac{\kappa_-}{2\pi} q(x_-).
\end{split}
\]
We set 
\begin{equation}
\begin{aligned}\label{defiqNodal}
  A^\Eps_\pm &=\Big\{ x \in \Omega \st \pm u^\Eps(x)> \pm q^\Eps(x) +  \frac{\kappa_\pm}{2\pi} \log \frac{1}{\eps_\pm}\Big\}, \\
  \omega_\pm^\Eps&=\pm \frac{1}{\eps_\pm{}^2} f(\pm(u^\Eps-q^\Eps)), \\ 
  \kappa^\Eps_\pm&=\int_{\Omega} \omega_\pm^\Eps, \\
  x^\Eps&=\frac{1}{\kappa^\Eps}\int_{\Omega} x \, \omega^\Eps_\pm(x)\, dx, \\
  \rho^\Eps_\pm&=\rho_{\kappa^\Eps_\pm}.
\end{aligned}
\end{equation}

We will prove

\begin{theorem}\label{thm:K3}
As $\Eps \to 0$, we have
\[
\begin{split}
  u^\Eps= \, & U_{\kappa_+^\Eps} \Big(\frac{\cdot-x^\Eps}{\eps_+}\Big)+\kappa_+^\Eps\Bigl(\frac{1}{2\pi} \log \frac{1}{\eps_+ \rho_{+}^\Eps}+ H(x^\Eps_+, \cdot)\Bigr)\\
&+U_{\kappa_-^\Eps} \Big(\frac{\cdot-x^\Eps}{\eps_-}\Big)+\kappa_-^\Eps\Bigl(\frac{1}{2\pi} \log \frac{1}{\eps_- \rho_-^\Eps}+ H(x^\Eps_-, \cdot)\Bigr)+o(1),
\end{split}
\]
in $\mathrm{W}^{2, 1}_\mathrm{loc}(\Omega)$, in $\mathrm{W}^{1, 2}_0(\Omega)$, and in $\mathrm{L}^\infty(\Omega)$, where
\[
  \kappa^\Eps=\kappa_\pm+\frac{2\pi}{\log \frac{1}{\eps_\pm}}\Bigl(q(x^\Eps)-\kappa_\pm H(x^\Eps, x^\Eps)-\kappa_\mp G(x_\pm, x_\mp) -\frac{\kappa}{2\pi} \log \frac{1}{\rho_{\kappa_\pm}} \Bigr)+o(\logEps^{-1}),
\]
and
\[
   \mathcal{W}(x^\Eps_+,x^\Eps_-) \to \sup_{(x_+,x_-) \in \Omega^2_*} \mathcal{W}(x_+,x_-). 
\]
One also has
\[
  B(x^\Eps_\pm, \Bar{r}_\pm^\Eps) \subset A_\pm^\Eps \subset B(x_\pm^\Eps, \mathring{r}_\pm^\Eps),
\] 
with $\Bar{r}^\Eps_\pm=\eps_\pm \rho_{\kappa_\pm}+o(\eps_\pm)$ and $\mathring{r}_\pm^\Eps=\eps_\pm
\rho_{\kappa_\pm} +o(\eps_\pm)$. Finally, 
\[
  \mathcal{E}^\Eps (u^\Eps)= \frac{\kappa^2_+}{4\pi} \log \frac{1}{\eps_+}+\frac{\kappa^2_-}{4\pi} \log \frac{1}{\eps_-}-\mathcal{W}(x^\Eps_+,x^\Eps_-)+\mathcal{C}_++\mathcal{C}_-+o(1). 
\]
\end{theorem}



\subsection{Upper bounds on the energy}

We compute upper bounds on $d^\eps$ by constructing suitable elements in
$\Meps$. 

\begin{lemma}
\label{lemNodalUpperBound}
For every $\Hat{x}_+, \Hat{x}_- \in \Omega$ such that $\Hat{x}_+\ne \Hat{x}_-$, there exists 
\[
  \Hat{\kappa}^{\pm}_{\Eps}=\kappa_\pm+\frac{2\pi}{\log \dfrac{1}{\eps_\pm}}\Bigl( q(\Hat{x}_\pm)-\kappa_\pm H(\Hat{x}_\pm, \Hat{x}_\pm)-\kappa_{\mp} G(\Hat{x}_\pm, \Hat{x}_\mp)+\dfrac{\kappa_\pm}{2\pi} \log \rho_{\kappa_\pm} \Bigr)+O\bigl(\logEps^{-2}\bigr),
\]
such that, if 	
\[
\begin{split}
  \Hat{u}^\Eps(x)
=&U_{\Hat{\kappa}_+^\Eps}\Bigl(\frac{x-\Hat{x}_+}{\eps_+}\Bigr)+\Hat{\kappa}^\Eps_+ 
    \Bigl( \frac{1}{2\pi} \log \frac{1}{\eps_+\Hat{\rho}_+^\Eps}+H(\Hat{x}_+, x) \Bigr)\\
&+U_{\Hat{\kappa}_-^\Eps}\Bigl(\frac{x-\Hat{x}_-}{\eps_-}\Bigr)+\Hat{\kappa}^\Eps_- 
    \Bigl( \frac{1}{2\pi} \log \frac{1}{\eps_-\Hat{\rho}_-^\Eps}+H(\Hat{x}_-, x) \Bigr),
\end{split}
\]
then 
\[
  \Hat{u}^\Eps\in \mathcal{M}^\Eps. 
\]
Moreover, 
\[
  \Hat{A}_\pm^\Eps:=\bigl\{ x \st \pm \Hat{u}^\Eps(x) > \pm q^\Eps_\pm(x) \bigr\} \subset B(\Hat{x}_\pm, \Hat{r}_\pm^\Eps), \\
\]
with $\Hat{r}_\pm^\Eps=\eps_\pm \rho_{\kappa_\pm}+o(\Eps)$. 
\end{lemma}

\begin{proof}
For every $\boldsymbol{\sigma}=(\sigma_+, \sigma_-) \in \R^2$, we define
\begin{gather*}
  \Hat{\kappa}^\pm_{\Eps, \boldsymbol{\sigma}}=\frac{q^\Eps_\pm(x_\pm)-\kappa_\mp G(x_\pm, x_\mp)+\sigma_\pm}{\frac{1}{2\pi} \log \frac{1}{\eps_\pm \rho_{\kappa_\pm}}+H(x_\pm, x_\pm)},\\
\begin{split}
  \Hat{u}_{\Eps, \boldsymbol{\sigma}} =&U_{\Hat{\kappa}_+^\Eps}\Bigl(\frac{x-\Hat{x}_+}{\eps_+}\Bigr)+\Hat{\kappa}^\eps_+ 
    \Bigl( \frac{1}{2\pi} \log \frac{1}{\eps_+\Hat{\rho}_+^\Eps}+H(\Hat{x}_+, x) \Bigr)\\
&+U_{\Hat{\kappa}_-^\Eps}\Bigl(\frac{x-\Hat{x}_-}{\eps_-}\Bigr)+\Hat{\kappa}^\eps_- 
    \Bigl( \frac{1}{2\pi} \log \frac{1}{\eps_-\Hat{\rho}_-^\Eps}+H(\Hat{x}_-, x) \Bigr),
\end{split}
\end{gather*}
and we set
\[
  g^{\Eps}_\pm(\boldsymbol{\sigma})=\langle d \mathcal{E}_{\Eps} (\Hat{u}^{\Eps, \boldsymbol{\sigma}}), \Hat{u}^{\Eps, \sigma_\pm} \rangle. 
\]
We compute as in the proof of Lemma~\ref{lemmaHatuNehari}, 
\begin{equation}\label{eq:badaboum1}
  \int_{\Omega} \abs{\nabla u^{\Eps, \sigma_\pm}}^2
  =\int_{B(0, \rho_{\Hat{\kappa}_{\Eps, \boldsymbol{\sigma}}})} \abs{\nabla U_{\Hat{\kappa}_{\Eps, \boldsymbol{\sigma}}}}^2+\Hat{\kappa}^{\Eps, \boldsymbol{\sigma}}_{\pm}  \Bigl(\frac{\kappa_\pm}{2\pi} \log \frac{1}{\eps_\pm}+q(\Hat{x}_\pm)+\sigma_\pm\Bigr)+O(\abs{\Eps}). 
\end{equation}
We also set 
\[
\Hat{\omega}^{\Eps, \boldsymbol{\sigma}}=\frac{1}{\eps_+^2} f(\Hat{u}^{\Eps, \boldsymbol{\sigma}} - q^\Eps_+) - \frac{1}{\eps_-^2} f(q^\Eps_- -\Hat{u}^{\Eps, \boldsymbol{\sigma}}), 
\]
and we compute as in the proof of Lemma~\ref{lemmaHatuNehari} 
\begin{multline}\label{eq:badaboum2}
  \frac{1}{\eps_\pm^2}\int_{\Omega} \Hat{\omega}^{\Eps, \boldsymbol{\sigma}} \Hat{u}^{\Eps, \boldsymbol{\sigma}}_\pm
 =\int_{\R^2} F(U_{\kappa_\pm}+\sigma_\pm) \\+ (\tfrac{\kappa}{2\pi} \log \tfrac{1}{\eps_\pm}+q(\Hat{x}_\pm)+\sigma_\pm)\int_{\R^2}f(U_{\kappa_\pm}+\sigma_\pm)+o(1). 
\end{multline}
Combining \eqref{eq:badaboum1} and \eqref{eq:badaboum2} we obtain
\[
  g^\Eps_\pm(\boldsymbol{\sigma})=\frac{\kappa_\pm}{2\pi} \log \frac{1}{\eps_\pm} \Bigl( \int_{\R^2} f(U_{\kappa_\pm})-f(U_{\kappa_\pm}+\sigma_\pm)\Bigr)+O(1). 
\]
By the Poincar\'e--Miranda Theorem (see e.g.\ \cite{Kulpa1997}), when $\abs{\Eps}$ is small,  there exists $\boldsymbol{\sigma}_\Eps$ such that $g^\Eps(\boldsymbol{\sigma}_\Eps)=0$ and $\boldsymbol{\sigma}_\Eps=o(1)$ as $\Eps \to 0$. 
\end{proof}

Evaluating $\mathcal{E}_\Eps(\Hat{u}_\Eps)$ yields

\begin{corollary}\label{cor:Nodalupper}
As $\abs{\Eps} \to 0$, we have
\begin{equation*}\begin{split}
d^\Eps \leq\ &\frac{\kappa^2_+}{4\pi} \log \frac{1}{\eps_+}+\frac{\kappa^2_-}{4\pi} \log \frac{1}{\eps_-}-\mathcal{W}(x_+,x_-)+\mathcal{C}_++\mathcal{C}_-+o(1). 
\end{split}\end{equation*}
\end{corollary}

\subsection{Asymptotic behavior of solutions}

We shall prove the counterpart of Proposition~\ref{prop:1mai}

\begin{proposition}\label{prop:1maiNodal}
Let $(v^\Eps)$ be a family of solutions to \eqref{problemPeps} such that $v^\Eps_\pm \ne 0$
\begin{equation}
\label{assumptEnergyUpperboundNodal}
  \mathcal{E}^\Eps(v^\Eps) \le \frac{\kappa^2_+}{4\pi} \log \frac{1}{\eps_+}+\frac{\kappa^2_-}{4\pi} \log \frac{1}{\eps_-}+O(1),
\end{equation}
as $\Eps \to 0$. Define the
quantities $A_\pm^\Eps$, $\omega_\pm^\Eps$, $\kappa_\pm^\Eps$, $x_\pm^\Eps$ and $\rho_\pm^\Eps$ for $v^\Eps$ as in \eqref{defiqNodal} for $u^\Eps$.  Then 
\[
\begin{split}
  v^\Eps= \, & U_{\kappa_+^\Eps} \Big(\frac{\cdot-x^\Eps}{\eps_+}\Big)+\kappa_+^\Eps\Bigl(\frac{1}{2\pi} \log \frac{1}{\eps_+ \rho_{+}^\Eps}+ H(x^\Eps_+, \cdot)\Bigr)\\
&+U_{\kappa_-^\Eps} \Big(\frac{\cdot-x^\Eps}{\eps_-}\Big)+\kappa_-^\Eps\Bigl(\frac{1}{2\pi} \log \frac{1}{\eps_- \rho_-^\Eps}+ H(x^\Eps_-, \cdot)\Bigr)+o(1),
\end{split}
\]
\text{in $\mathrm{W}^{2, 1}_\mathrm{loc}(\Omega)$, in $\mathrm{W}^{1, 2}_0(\Omega)$, and in $\mathrm{L}^\infty(\Omega)$}, where
\[
  \kappa^\Eps=\kappa_\pm+\frac{2\pi}{\log \frac{1}{\Eps_\pm}}\Bigl(q(x^\Eps)-\kappa_\pm H(x^\Eps, x^\Eps)-\kappa_\mp G(x_\pm, x_\mp) -\frac{\kappa}{2\pi} \log \frac{1}{\rho_{\kappa_\pm}} \Bigr)+o(\logEps^{-1}).
\]
In particular, we have
\[
  \mathcal{E}^\Eps (v^\Eps)= \frac{\kappa^2_+}{4\pi} \log \frac{1}{\eps_+}+\frac{\kappa^2_-}{4\pi} \log \frac{1}{\eps_-}-\mathcal{W}(x_+,x_-)+\mathcal{C}_++\mathcal{C}_-+o(1).
\]
and
\[
  B(x^\Eps_\pm, \Bar{r}_\pm^\Eps) \subset A_\pm^\Eps \subset B(x_\pm^\Eps, \mathring{r}_\pm^\Eps),
\] 
with $\Bar{r}^\Eps_\pm=\eps_\pm \rho_{\kappa_\pm}+o(\eps_\pm)$ and $\mathring{r}_\pm^\Eps=\eps_\pm
\rho_{\kappa_\pm} +o(\eps_\pm)$. 
\end{proposition}

In other words, 
$v^\Eps$ satisfies the same asymptotics as the one stated in
Theorem~\ref{thm:K1} for $v^\Eps$ except for the convergence of $x^\Eps$.

\subsubsection{Step 1: First quantitative properties of the solutions}


\begin{proposition}
\label{propositionNodalEstimatesueps} 
We have, as $\abs{\Eps} \to 0$,
\begin{gather*}
\muleb{2}(A^\Eps_\pm) =O\bigl(\logEps^{-1}\bigr), \\
\int_{A^\Eps_+} \abs{\nabla (v^\Eps-q^\Eps_\pm)}^2 =O(1), \\
\frac{1}{\eps_{\pm}^2}\int_{A^\Eps_\pm} F(\pm(v^\Eps-q^\Eps_\pm)) =O(1),\\
\int_{\Omega\setminus A^\Eps_\pm} \abs{\nabla v^\Eps_\pm}^2 \leq \frac{\kappa^2_\pm}{2\pi} \log\frac{1}{\Eps_\pm} + O(1), \\
\pm \int_{\Omega} \omega^\Eps_\pm \leq \pm \kappa_\pm + O\bigl(\logEps^{-1}\bigr). 
\end{gather*}
\end{proposition}
\begin{proof}
First note that by Theorem~\ref{thm:K1},
\[
  \mathcal{E}_\Eps(v^\Eps_\pm) \ge \frac{\kappa_+^2}{4\pi} \log \frac{1}{\eps_\pm}+O(1). 
\]
By \eqref{assumptEnergyUpperboundNodal}, this implies that 
\[
  \mathcal{E}_\Eps(v^\Eps_\pm) =\frac{\kappa_+^2}{4\pi} \log \frac{1}{\eps_\pm}+O(1). 
\]
We are now in position to proceed as in the proof of Proposition~\ref{propositionEstimatesueps}, testing $(\mathcal{Q}^\Eps)$ against $v^\Eps_+$ and $v^\Eps_-$ instead of $v^\Eps$, then against $\min(v^\Eps, q^\Eps_+)$ and $\max(v^\Eps, q^\Eps_-)$ instead of $\min(v^\Eps, q^\Eps)$, and finally against $(v^\Eps-q^\Eps_+)_+$ and $(q^\Eps_--v^\Eps)_+$ instead of $(v^\Eps-q^\Eps_+)_+$. We skip the details. 
\end{proof}

\subsubsection{Step 2: Structure of the vorticity set}
In this subsection we further describe the vorticity set $A^\Eps=A^\Eps_+\cup A^\Eps_-$. Since it is an open set, it contains at most countably many connected components that we label $A^\Eps_{\pm, i}$, $i \in I^\Eps_\pm$. 
First we have a control on the total area and on the diameter of each connected component. 

\begin{lemma}
\label{lemmaNodalAreaDiameter}
If $\abs{\Eps}$ is sufficiently small, we have
\[
  \muleb{2}(A^\Eps_\pm) \le C \eps_\pm^2
\]
and, for every $i \in I^\Eps_\pm$,
\begin{equation}
\label{ineqNodalVorticityDiameter}
  \diam(A^\Eps_{\pm, i}) \le C \eps_\pm. 
\end{equation}
\end{lemma}
\begin{proof}
It suffices to repeat the arguments in the proof of Lemma~\ref{lemmaNodalAreaDiameter}. 
\end{proof}

\begin{lemma}
\label{lemmaNodalVortexSplit}
There exists constants $\gamma, C, c>0$ such that, when $\abs{\Eps}$ is small enough,
if 
\begin{equation}
\label{eqNodalSplitVortices}
  \int_{A^\Eps_{\pm, i}} \abs{\nabla (v^\Eps - q^\Eps_\pm)}^2 > \gamma^2,
\end{equation}
then for every $j \in I^\Eps_\mp$,
\begin{gather}
\label{ineqNodalLowerBoundMeas}  \muleb{2}(A_{\pm, i}^\Eps)\ge c\eps^2_\pm, \\ 
\label{ineqNodalLowerBoundDiam}  \diam(A^\Eps_{\pm, i})\ge c\eps_\pm, \\
\label{ineqNodalLowerBoundBoundary}   \dist(A^\Eps_{\pm, i}, \partial \Omega)\ge c, \\
\label{ineqNodalSignDistance}  \dist(A^\Eps_{\pm, i}, A^\Eps_{\mp, j})\ge c,
\end{gather}
while if \eqref{eqNodalSplitVortices} does not hold, then 
\[
  \int_{A^\Eps_{\pm, i}} \abs{\omega^\Eps}^s \le C \Norm{\nabla q}_{\mathrm{L}^r(A^\Eps_{\pm, i})}^{sp} \muleb{2}(A^\Eps_{\pm,i})^{1+sp(\frac{1}{2}-\frac{1}{r})},
\]
where $C$ only depends on $s \ge 1$. 
\end{lemma}
\begin{proof}
The proof is very similar to the one of Lemma~\ref{lemmaVortexSplit} except for \eqref{ineqNodalSignDistance} which remains to be proved. To that purpose, we consider the function
\[
  \eta_\Eps=\frac{\frac{v^\Eps_+}{\kappa_+}+\frac{v^\Eps_-}{\kappa_+}}{\log \frac{1}{\eps_+\eps_-}}. 
\]
We have 
\[
\eta_\Eps \restrictedto{A^\Eps_{\pm, i}}=\frac{\log \frac{1}{\eps_+}}{\log \frac{1}{\eps_+\eps_-}}+O\bigl(\logEps^{-1}\bigr),
\]
and
\[
\eta_\Eps \restrictedto{A^\Eps_{\mp, j}}=\frac{- \log \frac{1}{\eps_-}}{\log \frac{1}{\eps_+\eps_-}}+O\bigl(\logEps^{-1}\bigr). 
\]
Therefore,
\[
  \frac{2\pi}{\capa(A^\Eps_+, \R^2\setminus A^\Eps_-)}\ge \log \frac{1}{\eps_+ \eps_-}+O(1). 
\]
Using Proposition~\ref{propositionCapacityBoundDistance} with $\Omega=\R^2 \setminus \overline{A^\Eps_{\pm, i}}$ and $K=\overline{A^\Eps_{\mp, j}}$, and applying \eqref{ineqNodalVorticityDiameter} to $A^\Eps_{\mp, j}$ and \eqref{ineqNodalLowerBoundMeas} to $A^\Eps_{\pm, i}$, we are led to
\[
  \log \frac{1}{\eps_+\eps_-} \le \log C \Bigl( 1+\frac{\dist(A^\Eps_{\pm, i}, A^\Eps_{\mp, j})}{\eps_\mp}\Bigr)\Bigl( 1+\frac{\dist(A^\Eps_{\pm, i}, A^\Eps_{\mp, j})}{\eps_\pm}\Bigr)+O(1),
\]
which can not hold if $\dist(A^\Eps_{\pm, i}, A^\Eps_{\mp, j}) \to 0$. 
\end{proof}

The vorticity set is split into four subsets:
\begin{align*}
  V^\Eps_\pm&=\bigcup \Bigl\{A_{\pm, i}^\Eps \st \int_{A_{\pm, i}^\Eps} \abs{\nabla (v^\Eps - q^\Eps_\pm}^2 \le \gamma^2\Bigr\}, \\
  E^\Eps_\pm&=\bigcup \Bigl\{A_{\pm, i}^\Eps \st \int_{A_{\pm, i}^\Eps} \abs{\nabla (v^\Eps - q^\Eps_\pm)}^2 > \gamma^2\Bigr\}. 
\end{align*}
By Proposition~\ref{propositionNodalEstimatesueps}, the sets $E^\Eps_+$ and $E^\Eps_-$ contain finitely many connected components, and by \eqref{ineqNodalLowerBoundMeas}, \eqref{ineqNodalLowerBoundDiam}, \eqref{ineqNodalLowerBoundBoundary} and \eqref{ineqNodalSignDistance}, they can thus be split as $E^\Eps_\pm=\bigcup_{j=1}^{k^\Eps_\pm} E^\Eps_{\pm, j}$, where $E^\Eps_{\pm, j}$ are nonempty open sets such that 
\begin{gather*}
  \frac{\dist(E^\Eps_{\pm, i}, E^\Eps_{\pm, j})}{\eps_\pm} \to \infty,\\
  \liminf_{\Eps \to 0} \dist(E^\Eps_{\pm, i}, E^\Eps_{\mp, j}) >0,\\
  \liminf_{\Eps \to 0} \dist(E^\Eps_{\pm, i}, \partial \Omega) >0, \\
\limsup_{\Eps \to 0} \frac{\diam (E^\Eps_{\pm, i})}{\eps_\pm} < \infty,
\end{gather*}
as $\Eps \to 0$. 
By definition of $E^\Eps$ and by \eqref{ineqVortexEnergy}, $k^\Eps_+$ and $k^\Eps_-$ remain bounded as $\Eps \to 0$. 

\subsubsection{Step 3: Small scale asymptotics}
We set 
\begin{align*}
  \omega^\Eps_{\pm, v}&=\omega^\Eps \charfun{V_{\pm}^\Eps}, &
  \omega^\Eps_{\pm, i}&=\omega^\Eps \charfun{E^\Eps_{\pm, i}}, \\
  \kappa^\Eps_{\pm, i}&=\int_{\Omega} \omega^\Eps_{\pm, i}, &
  x^\Eps_{\pm, i}&=\frac{1}{\kappa^\Eps_{\pm, i}}\displaystyle\int_{\Omega} x\omega^\Eps_{\pm, i}(x)\, dx. 
\end{align*}

Using the analogues of Lemma~\ref{lemmaVanishingVorticity} and Lemma~\ref{lemmaSmallScaleLocalEstimates}, one obtains the analogue of Lemma~\ref{lemmaLocalAsymptotics}. 

\begin{lemma}
\label{lemmaNodalLocalAsymptotics}
When $\Eps$ is small, we have $k^\Eps_+=k^\Eps_-=1$, and
\begin{multline*}
  \kappa^\Eps_{\pm, 1}=\kappa_\pm +\frac{2\pi}{\log \frac{1}{\eps_\pm}}\Bigl(q(x^\Eps_\pm)-\kappa_\pm H(x^\Eps_\pm, x^\Eps_\pm)-\kappa_{\mp} G(x^\Eps_{\pm}, x^\Eps_{\mp})-\frac{\kappa_\pm}{2\pi} \log \frac{1}{\rho_{\kappa_\pm}} \Bigr) \\
+o({\logEps}^{-1})
\end{multline*}
and $v^\Eps_{\pm} \to U_{\kappa_\pm}$ in $\mathrm{W}^{1, r}_{\mathrm{loc}}(\R^2)$. 
\end{lemma}

\subsubsection{Step 4: Global asymptotics}

The counterpart of Proposition~\ref{propositionAsymptoticsW21} is now

\begin{proposition}
\label{propositionAsymptoticsW21Nodal}
We have
\[
\begin{split}
  v^\Eps= & \ U_{\kappa_{+,1}^\Eps}\Bigl(\frac{\cdot-x^\Eps_{+, 1}}{\eps_+}\Bigr)+\kappa^\Eps_{+, 1}\Bigl(\frac{1}{2\pi} \log \frac{1}{\eps_+ \rho_{\kappa_{+, 1}^\Eps}}+H(x^\Eps_{+, 1}, \cdot)\Bigr) \\
&+U_{\kappa_{-,1}^\Eps}\Bigl(\frac{\cdot-x^\Eps_{-, 1}}{\eps_-}\Bigr)+\kappa^\Eps_{-, 1}\Bigl(\frac{1}{2\pi} \log \frac{1}{\eps_- \rho_{\kappa_{-, 1}^\Eps}}+H(x^\Eps_{-, 1}, \cdot)\Bigr)+o(1)
\end{split}
\]
in $\mathrm{W}^{2, 1}_{\mathrm{loc}}(\Omega)$, in $\mathrm{W}^{1, 2}_0(\Omega)$, and in $\mathrm{L}^\infty(\Omega)$. 
\end{proposition}

We have now all the ingredients to complete the
\begin{proof}[Proof of Proposition~\ref{prop:1maiNodal}]
It follows from the combination of Lemma~\ref{lemmaNodalLocalAsymptotics}, Proposition~\ref{propositionAsymptoticsW21Nodal} and the counterparts of Corollaries~\ref{corollaryAsymptotic} and \ref{corEnergy}.
\end{proof}

\begin{proof}[Proof of Theorem~\ref{thm:K3}]
Since the solutions have the upper bound Corollary~\ref{cor:Nodalupper}, one can conclude from Proposition~\ref{prop:1maiNodal}. 
\end{proof}

\section{Desingularized solutions of the Euler equation}
\label{sect:resu}

\subsection{Bounded domains}
In bounded domains we shall successively consider stationary vortices, rotating vortices and stationary pairs of vortices.

\subsubsection{Stationary vortices in simply-connected bounded domains}
Let us first deduce Theorem~\ref{thm:resu} from Theorem~\ref{thm:K1}.

\begin{proof}[Proof of Theorem~\ref{thm:resu}]
Take $q=-\psi_0$, where $\psi_0$ satisfies \eqref{eqpsi0}. One checks that $\psi_0 \in W^{1+\frac{1}{s}, s}(\Omega)$ so that $u \in W^{1, r}(\Omega)$ for every $r < \infty$. Define $\mathbf{v}_\eps=(\nabla u_\eps)^\perp$ where $u_\eps$ is given by Proposition~\ref{prop:2.1}. The conclusion then follows from Theorem~\ref{thm:K1}.
\end{proof}

We have constructed in Theorem~\ref{thm:resu} a family of solutions that concentrates around a global maximum of the Kirchhoff--Routh function $\mathcal{W}$; it is also possible to construct family of solutions that concentrate around a \emph{local} maximum of $\mathcal{W}$:

\begin{theorem}\label{thmLocalMinimum}
	Let $\Omega \subset \R^2$ be a bounded simply-connected smooth domain and $v_n:\partial \Omega \to \R\in L^s(\partial \Omega)$ for some $s>1$ be such that $\int_{\partial \Omega} v_n = 0.$ Let $\kappa >0$ be given and let $\Hat{x} \in \Omega$ be a strict local minimizer of $\mathcal{W}$. For $\eps>0$ there exist smooth stationary solutions $\mathbf{v}_\eps$ of the Euler equation in $\Omega$ with outward boundary flux given by $v_n$, corresponding to vorticities $\omega_\eps$, such that ${\rm supp}(\omega_\eps) \subset B(x_\eps, C\eps)$ for some $x_\eps \in \Omega$ and $C>0$ not depending on $\eps$. Moreover, as $\eps \to 0$, 
\[
  \int_\Omega \omega_\eps \to \kappa
\]
and $x_\eps \to \Hat{x}$. 
\end{theorem}
\begin{proof}
Assume that $\Hat{x}$ is the unique minimizer of $\mathcal{W}$ in $B(\Hat{x}, \rho)$. Define $q \in C^\infty(\Bar{\Omega})$ so that $q=-\psi_0$ in $B(\Hat{x}, \rho/2)$, where $\psi_0$ satisfies \eqref{eqpsi0} and for every $x \in \Omega$,
\[
 \kappa q(x)-\frac{\kappa^2}{2} H(x, x) > \kappa q(x_*)-\frac{\kappa^2}{2} H(x_*, x_*). 
\]
We now apply Theorem~\ref{thm:K1} with $q$. By construction of $q$, we have $x_\eps \to \Hat{x}$. 

But then, one has, still by Theorem~\ref{thm:K1}
\[
  u_\eps(x) \ge \frac{\kappa}{2\pi} \log \frac{1}{\abs{x_\eps-x}}+O(1). 
\]
Therefore, when $\eps$ is small enough, $u_\eps \le -\psi_0+\frac{\kappa}{2\pi}\log \frac{1}{\eps}$ and $u_\eps \le q_\eps$ in $\Omega \setminus B(x_\eps, \rho/2)$. Therefore, for such $\eps$, $u_\eps$ solves $-\eps^2\Delta u_\eps=f(u_\eps+\psi_0- \frac{\kappa}{2\pi} \log \frac{1}{\eps})$ in $\Omega$. One can now take $\mathbf{v}_\eps= (\nabla (u_\eps+\psi_0))^\perp$ and show that this is a stationary solution to the Euler equation. 
\end{proof}

\subsubsection{Stationary vortices in multiply-connected bounded domains}

If $\Omega$ is not simply connected then $\Omega = \Omega_0 \setminus \bigcup_{h=1}^m \Omega_h$, where $\Omega_0, \dotsc, \Omega_m$ are bounded simply connected domains, one can prescribe for $h \in \{1, \dotsc, m\}$, the circulations $\int_{\partial \Omega_h} \mathbf{v}\cdot \tau=\gamma_h$.
In that case $\mathbf{v}_0$ is the unique harmonic field whose normal component on the boundary is $v_n$; i.e., $\mathbf{v}_0$ satisfies
\[
\left\{
\begin{aligned}
  \nabla \cdot \mathbf{v}_0&=0, & & \text{in $\Omega$}, \\
  \nabla \times \mathbf{v}_0&=0, & & \text{in $\Omega$}, \\
  n \cdot \mathbf{v}_0&=v_n& & \text{on $\partial \Omega$},\\
\int_{\partial \Omega_h} \mathbf{v}_0 \cdot \tau &=\gamma_h & &\text{for $h \in \{1, \dotsc, m\}$}.
\end{aligned}
\right. 
\]
If $\int_{\partial \Omega_h} v_n=0$ for every $h \in \{1, \dotsc, m\}$,  $\mathbf{v}_0=(\nabla \psi_0)^\perp$ where 
\begin{equation}
 \label{eqpsi0NotConnected}
\left\{
\begin{aligned}
-\Delta \psi_0&=0& &\text{in $\Omega$}, \\
-\frac{\partial \psi_0}{\partial \tau}&=v_n & &  \text{on $\partial \Omega$},\\
\int_{\partial \Omega_h}  \frac{\partial \psi_0}{\partial n} & = \gamma_h & &\text{for $h \in \{1, \dotsc, m\}$}.\\
\end{aligned}
\right. 
\end{equation} 
The Kirchhoff--Routh function associated to the vortex dynamics is then given by 
\[
  \mathcal{W}_*(x)=\frac{\kappa^2}{2}H_*(x, x)+\kappa \psi_0(x),
\]
where one should recall that $\psi_0$ depends on $v_n$ and $\gamma_h$ for $h \in \{1, \dotsc, m\}$.

We have 
\begin{theorem}\label{thm:MultiplyConnected}
Let $\Omega \subset \R^2$ be a bounded smooth domain and $v_n:\partial \Omega \to \R\in L^s(\partial \Omega)$ for some $s>1$ be such that $\int_{\partial \Omega_h} v_n = 0$ for every $h \in \{0, \dotsc, m\}$. Let $\gamma_h \in \R$ for $h \in \{1, \dotsc, m\}$ and let $\kappa >0$ be given. For $\eps>0$ there exist smooth stationary solutions $\mathbf{v}_\eps$ of the Euler equation in $\Omega$ with outward boundary flux given by $v_n$ and circulations given by $\gamma_h$, corresponding to vorticities $\omega_\eps$, such that ${\rm supp}(\omega_\eps) \subset B(x_\eps, C\eps)$ for some $x_\eps \in \Omega$ and $C>0$ not depending on $\eps$. Moreover, as $\eps \to 0$, 
\[
  \int_\Omega \omega_\eps \to \kappa,
\]
and
\[
   \mathcal{W}_*(x^\eps) \to \sup_{x \in \Omega} \mathcal{W}_*(x). 
\]
\end{theorem}
\begin{proof}
	The proof is almost identical to the one of Theorem~\ref{thm:resu}, it relies on Theorem~\ref{thm:K1m} instead of Theorem~\ref{thm:K1}.
\end{proof}

\begin{remark}
One could similarly prove a counterpart of Theorem~\ref{thmLocalMinimum} for multiply connected domains.
\end{remark}

\subsubsection{Rotating vortices in a discs}
If $\Omega$ is invariant under rotation, one can consider the Euler equation in a reference frame rotating with angular velocity $\alpha$:
\[
\left\{
  \begin{aligned}
    \nabla \cdot \mathbf{v} &= 0, \\
    \mathbf{v}_t + \mathbf{v}\cdot \nabla \mathbf{v}&=-\nabla p+2\alpha \mathbf{v}^\perp-\alpha^2x. 
  \end{aligned}
\right. 
\]
The vorticity of $\mathbf{v}$ with respect to an inertial frame is $\nabla \times \mathbf{v}+2\alpha$. 
The movement of singular vortices is governed by Kirchhoff's law \eqref{equationKirchhoff}, where $\mathcal{W}$ is replaced by  $\mathcal{W}_\alpha(x)=\mathcal{W}(x)+\sum_i \alpha \frac{\abs{x}^2}{2}$. 

The stream-function method to construct stationary solutions in a rotating reference frame can be adapted to this situation. If $-\Delta \psi=f(\psi)-2\alpha$, setting $\mathbf{v}=(\nabla \psi)^\perp$ and $p= F(\psi)-\frac{\alpha^2}{2}\abs{x}^2-\frac{1}{2}\abs{\nabla \psi}^2$ yields a solution\footnote{With the same velocity field, choosing as pressure $p=F(\psi)-2\alpha \psi-\frac{1}{2}\abs{\nabla \psi}^2$ would of course give a solution to the Euler equation in a Galilean frame. }. In particular, the solution is irrotational outside on the set where $\psi=0$. 

\begin{theorem}
\label{thmRotating}
Let $\rho > 0$, $\kappa >0$ and $\alpha > 0$. If $\kappa < 2\pi \alpha \rho^2$
For $\eps>0$ there exist smooth rotating solutions $\mathbf{v}_\eps$ of the Euler equation in $B(0, \rho)$ with angular velocity $\alpha$, corresponding to vorticities $\omega_\eps$, such that ${\rm supp}(\omega_\eps)$ is contained in a disc of radius $O(\eps)$ around a point rotating on the circle of radius $\sqrt{\rho^2-\frac{\kappa}{2\pi\alpha}}$. Moreover, as $\eps \to 0$, 
\[
\int_\Omega \omega_\eps \to \kappa. 
\]
\end{theorem}
\begin{proof}
Take 
\[
 q(x)=-\alpha \frac{\abs{x}^2}{2}. 
\]
and apply Theorem~\ref{thm:K1}. One checks that 
\[
  \mathbf{v}_\eps(x, t)=(\nabla u_\eps)^\perp (R(\alpha t) x),
\]
where $R(\alpha t)$ denote the rotation of $\alpha t$, satisfies Euler equation.
Since
\[
 \mathcal{W}_\alpha(x)=\frac{\kappa^2}{4\pi}\log \frac{\rho^2-\abs{x}^2}{\rho}+\frac{\kappa \alpha}{2} \abs{x}^2,
\]
attains its maximum on the circle of radius $\sqrt{\rho^2-\frac{\kappa}{2\pi\alpha}}$, one has the desired concentration result. 
\end{proof}

\begin{remark}
When $\kappa > 2\pi \alpha \rho^2$, the minimizer concentrates around $0$; one recovers thus stationary solutions as in Theorem~\ref{thm:resu}. 
\end{remark}

\subsubsection{Stationary pairs of vortices in bounded domains}

\begin{theorem}
Let $\Omega \subset \R^2$ be a bounded simply-connected smooth domain and $v_n:\partial \Omega \to \R\in L^s(\partial \Omega)$ for some $s>2$ be such that $\int_C v_n = 0$ over each connected component $C$ of $\partial \Omega$. Let $\kappa_+ >0$ and $\kappa_- < 0$ be given. For $\eps>0$ there exist smooth stationary solutions $\mathbf{v}_\eps$ of the Euler equation in $\Omega$ with outward boundary flux given by $v_n$, corresponding to vorticities $\omega_\eps$, such that ${\rm supp}(\omega_\eps^\pm) \subset B(x_\eps^\pm, C\eps)$ for some $x^{\pm}_\eps \in \Omega$ and $C>0$ not depending on $\eps$. Moreover, as $\eps \to 0$, 
\[
\int_\Omega \omega_\eps^\pm \to \kappa^\pm
\]
	 and
\[
   \mathcal{W}(x_\eps^+, x_\eps^-) \to \sup_{x^+, x^- \in \Omega} \mathcal{W}(x^+, x^-). 
\]
 \end{theorem}
\begin{proof}
This follows from Theorem~\ref{thm:K3} in the same lines as Theorem~\ref{thm:resu}.
\end{proof}

\begin{remark}
There is also a counterpart of Theorem~\ref{thmLocalMinimum} for vortex pairs, concerning the existence of solutions near local maxima of the Kirchoff--Routh function and a counterpart of Theorem~\ref{thm:MultiplyConnected} for domains which are not simply connected.
\end{remark}

\begin{remark}
One can also address the question of rotating vortex pairs. Combining the ingredients of the proof of Theorem~\ref{thmRotating}, one can prove the existence of rotating vortex pairs of strength $\kappa_+ > 0$ and $\kappa_- > 0$ that concentrates around two antipodal rotating points at distance $\rho _+$ and $\rho _-$ which maximize the function
\[
  \frac{\alpha \kappa_+}{2} \rho _+^2 +\frac{\alpha \kappa_-}{2}\rho _-^2+\frac{\kappa_+^2}{4\pi}\log (1-\rho _+^{2})+\frac{\kappa_-^2}{4\pi}\log (1-\rho _-^{2})+\frac{\kappa_+\kappa_-}{2\pi} \log \frac{1+\rho _+\rho _-}{\rho _++\rho _-}.
\]
In contrast with Theorem~\ref{thmRotating}, the pair of vortices obtained is always a nontrivial pair of rotating vortices for any $\alpha \ne 0$, $\kappa_+ > 0$ and $\kappa_- < 0$.
\end{remark}

\subsection{Unbounded domains}
We now consider the application  of the results of Section~\ref{sectUnbounded} to the desingularization of vortices in unbounded domains.

\subsubsection{Translating vortex pair in the plane}
We first consider the construction of a pair of vortices in $\R^2$. 
First recall that pair of vortices translating at velocity $\mathbf{W}$ in a flow with vanishing velocity at infinity is, up to a Galilean change of variables a pair of stationary vortices in a flow with velocity at infinity $-\mathbf{W}$.
The stream-function of the corresponding irrotational flow is $\psi_0(x)=\mathbf{W}^\perp \cdot x$.
Therefore, the positions of two vortices of opposite intensities $\kappa$ and $-\kappa$ in the moving reference frame is a critical point of the Kirchhoff--Routh $\mathcal{W}$ defined by
\[
 \frac{-\kappa^2}{2\pi} \log \frac{1}{\abs{x-y}} + \mathbf{W}^\perp \cdot x.
\]

%
%

\begin{theorem}
Let $W \ge 0$ and $\kappa \ge 0$, for every $\eps > 0$ there exist smooth stationary solutions $\mathbf{v}_\eps$ of the Euler equation in $\R^2$ symmetric with respect to the $x_2$ axis and such that $\lim_{x_1 \to \infty} \mathbf{v}_\eps(x)=(0, W)$, corresponding to vorticities $\omega_\eps$, such that ${\rm supp}(\omega_\eps) \cap \R^2_+ \subset B(\Bar{x}, C\eps)$, where $\Bar{x}= (\frac{\kappa}{4\pi W}, 0)$.
\end{theorem}
\begin{proof}
The problem can be reduced to finding a solution in $\R^2_+$ with vanishing flux on the boundary. The corresponding Kirchhoff--Routh function is 
\[
 \mathcal{W}(x)=\frac{\kappa^2}{4\pi} (\log 2x_1)-\kappa Wx_1
\]
This follows from the existence result of Theorem~\ref{thmYang}, the asymptotics of Proposition~\ref{propUnboundedUpper} and Proposition~\ref{prop:1maiUnbounded}.
\end{proof}

\subsubsection{Stationary vortex in the half-plane with non-vanishing flux}
The method just used extends to non-vanishing flux boundary conditions:

\begin{theorem}
Let $v_n \in L^1 (\R)\cap L^s_{\mathrm{loc}}(\R)$ for $s > 1$. If $\int_{-\infty}^0 v_n=-\int_0^\infty v_n>0$.
For every $W > 0$ and $\kappa > 0$, if $\kappa/W$ is small enough and if $\epsilon > 0$ sufficiently small there exist smooth stationary solutions $\mathbf{v}_\eps$ of the Euler equation in $\Omega$ with outward boundary flux given by $v_n$ and  $\lim_{x_1 \to \infty} \mathbf{v}_\eps(x)=(0, W)$, corresponding to vorticities $\omega_\eps$, such that ${\rm supp}(\omega_\eps) \subset B(x_\eps, C\eps)$ for some $x_\eps \in \Omega$ and $C>0$ not depending on $\eps$, and $\int_{\R^2_+} \omega_\eps \to \kappa$. 
\end{theorem}
\begin{proof}
Define $\psi_0$ by 
\[
\left\{
\begin{aligned}
  -\Delta \psi_0 & = 0 & & \text{in $\R^2_+$}, \\
  \partial_2 \psi_0 &=v_n& & \text{on $\partial \R^2_+$}, \\
  \psi_0(0, x_2) &\to 0 & &\text{as $\abs{x_2} \to \infty$},\\
  \frac{\psi_0(x)}{x_1} &\to -W & & \text{as $x_1 \to \infty$}. 
\end{aligned}
\right. 
\]
One checks that by our assumptions, 
\[
 \psi_0(0) > 0. 
\]
In order to apply  Theorem~\ref{theoremExistenceLevels}, we need to find $\Hat{x} \in \Omega$ such that 
\begin{equation}
\label{eqTrenchStrict}
  \kappa \psi_0 (\Hat{x})+\frac{\kappa^2}{4\pi} \log 2 \Hat{x}_1 > \frac{\kappa^2}{4\pi} \Bigl(\log \frac{\kappa}{2\pi W}-1\Bigr). 
\end{equation}
One takes $\Hat{x}=(\frac{\kappa}{4\pi W}, 0)$. If $\kappa/W$ is small enough, one has
\[
 \kappa \psi_0 (\Hat{x}) > 0,
\]
and one checks that 
\[
 \kappa \psi_0 (\Hat{x})+\frac{\kappa^2}{4 \pi} \log 2 \Hat{x}_1 > \frac{\kappa^2}{4\pi} \log \frac{\kappa}{2\pi W} > \frac{\kappa^2}{4\pi} \Bigl(\log \frac{\kappa}{2\pi W}-1\Bigr). 
\]
The conclusion follows then from Theorem~\ref{theoremExistenceLevels}. 
\end{proof}

\subsubsection{Stationary vortex in a perturbed half-plane}
Instead of perturbing the boundary condition on the half-plane, one can instead perturb the geometry. The first situation is the situation in which one has for example enlarged a little bit the half-plane around $0$:

\begin{theorem}
Assume that $\Omega$ is a simply-connected perturbation of $\R^2_+$ in the sense of \eqref{condPerturbation}. Let $\Bar{x} \in \partial \Omega$ be such that $x_1 > \Bar{x}_1$ for every $x \in \Omega$, $\partial \Omega$ is of class $C^2$ in a neighborhood of $\Bar{x}$,
then for every $W > 0$, if $\kappa > 0$ is sufficiently small and if $\epsilon > 0$ sufficiently small there exist smooth stationary solutions $\mathbf{v}_\eps$ of the Euler equation in $\Omega$ with vanishing boundary flux and  $\lim_{x_1 \to \infty} \mathbf{v}_\eps(x)=(0, W)$, corresponding to vorticities $\omega_\eps$, such that ${\rm supp}(\omega_\eps) \subset B(x_\eps, C\eps)$ for some $x_\eps \in \Omega$ and $C>0$ not depending on $\eps$ and $\int_{\Omega} \omega_\eps \to \kappa$. 
\end{theorem}

\begin{proof}
We are going to obtain the solutions by applying  Theorem~\ref{theoremExistenceLevels} with $q=-\psi_0$. 
Let $\mathbf{v}_0$ be the irrotationnal stationary solution to the Euler equation with vanishing flux on $\partial \Omega$ and $\lim_{x_1 \to \infty} \mathbf{v}_0(x)=(0, W)$, i.e.\ $\mathbf{v}_0=\nabla \psi_0^\perp$ with 
\begin{equation}
\label{eqPsi0}
\left\{
 \begin{aligned}
   -\Delta \psi_0 &= 0, & &\text{in $\Omega$}, \\
   \psi_0 &= 0  & &\text{on $\partial \Omega$}, \\
   \tfrac{\psi_0(x)}{x_1} &\to -W &&\text{as $x \to \infty$}. 
 \end{aligned}
\right. 
\end{equation}
In order to apply  Theorem~\ref{theoremExistenceLevels}, we need to find $\Hat{x} \in \Omega$ such that the condition
\eqref{eqTrenchStrict}
holds. First, by the strong maximum principle, one has $\partial_1 \psi(\Bar{x})>-W$, so that there exists $\gamma \in (0, W)$ such that in a neighborhood of $\Bar{x}$, 
\[
 \psi_0(x) >  -\gamma (x_1-\Bar{x_1}). 
\]
If we consider the point $\Hat{x}=(\Bar{x}_1+\frac{\kappa}{4\pi W}, \Bar{x}_2)$, one has
\[
 \kappa \psi_0(\Hat{x}) > -\gamma \frac{\kappa^2}{4\pi W}. 
\]
On the other hand, if $K$ denotes the curvature of $\partial \Omega$ at $\Bar{x}$, one has by Proposition~\ref{propositionAsymptotH},
\[
 \frac{\kappa^2}{2} H (\Hat{x}, \Hat{x})=-\frac{\kappa^2}{4\pi} \log \frac{\kappa}{2\pi W}+O(\kappa^3). 
\]
Therefore, if $\kappa$ is small enough, one has \eqref{eqTrenchStrict}, and one can then apply Theorem~\ref{theoremExistenceLevels} to obtain the conclusion. 
\end{proof}

\subsubsection{Translating vortex pair near a translating axisymmetric obstacle}
We can also treat a situation in some sense opposite to the situation of the previous section. We obtain the desingularization of vortices on a set which is obtained by removing some part of the half-plane. By a Galilean change of variables and by extension by symmetry of the flow, this corresponds also physically to a rigid body in translation together with a pair of vortices. A similar problem was studied through the vorticity method by B.\thinspace Turkington \cite{Turkington1983}

\begin{theorem}
Let $D \subset \R^2$ be a compact simply-connected set with non-empty interior and symmetric with respect to the $x_1$ variable. Then for every $\kappa> 0$ and $W > 0$, if $\eps > 0$ is sufficiently small there exist smooth stationary solutions $\mathbf{v}_\eps$ of the Euler equation in $\R^2 \setminus D$ symmetric with respect to the $x_2$ axis, with vanishing boundary flux and such that $\lim_{x_1 \to \infty} \mathbf{v}_\eps(x)=(0, W)$, corresponding to vorticities $\omega_\eps$, such that ${\rm supp}(\omega_\eps) \cap \R^2_+ \subset B(x_\eps, C\eps)$ for some $x_\eps \in \Omega \cap \R^2_+$ and $C>0$ not depending on $\eps$ and $\int_{\R^2_+ \setminus D} \omega_\eps \to \kappa$. 
\end{theorem}

\begin{proof}
Set $\Omega=\R^2_+ \setminus D$. We shall consider the case $W > 0$ and $\kappa > 0$, and we shall assume that $B(0, \rho) \subset D \subset B(0, R)$. We use again Theorem~\ref{theoremExistenceLevels} and therefore we shall prove that \eqref{eqTrenchStrict} holds for some $\Hat{x} \in \R^2$ where $\psi_0$ solves \eqref{eqPsi0}. 
We shall take $\Hat{x}^\lambda=(\frac{\kappa}{4\pi W}, \lambda \frac{\kappa}{4\pi W})$ where $\lambda \in \R$.
By the maximum principle on $\Omega$, one has for $x \in \Omega$,
\[
  \psi_0(x)>-Wx_1+W \frac{x_1 \rho^2}{\abs{x}^2}. 
\]
Hence, we have 
\[
 \kappa \psi_0(\Hat{x}^\lambda) \ge - \frac{\kappa^2 }{4\pi} + \frac{4\pi W^2}{1+\lambda^2},
\]
with $c' > 0$. 
We also use the formula of the Green function $\Tilde{G}$ of $\R^2_+ \setminus B(0, R)$ used by B.\thinspace Turkington \cite[p.\thinspace 1047]{Turkington1983}
\[
 \Tilde{G}(x, y)=\frac{1}{4\pi} \log \frac{1+\dfrac{4x_1y_1}{\abs{x-y}^2}}{1+\dfrac{4R^2 x_1y_1 } {(x_1y_1+x_2 y_2-R^2)^2+(x_2y_1-x_1y_2)^2}}
\]
Since $\Tilde{G}(x, y) \le G(x, y)$, one has therefore
\[
 H(x,x)\ge \frac{1}{2\pi} \log 2 x_1 - \frac{1}{2\pi} \log \Bigl(1+ \frac{4R^2x_1^2}{(\abs{x}^2-R^2)^2}\Bigr),
\]
whence
\[
  \frac{\kappa^2}{2} H(\Hat{x}^\lambda, \Hat{x}^\lambda) \ge \frac{\kappa^2}{4\pi} \log \frac{\kappa}{2\pi W}+O(\lambda^{-4}).
\]
One checks thus that for $\lambda$ sufficiently large, 
\[
  \kappa \psi_0 (\Hat{x}^\lambda)+\frac{\kappa^2}{4\pi} \log 2 \Hat{x}^\lambda_1 > \frac{\kappa^2}{4\pi} \Bigl(\log \frac{\kappa}{2\pi W}-1\Bigr).  
\]
and the conclusion thus follows from  Theorem~\ref{theoremExistenceLevels}. 
\end{proof}





\appendix

\section{Capacity estimates}
Let $\Omega \subset \R^2$ be open. The electrostatic capacity of a compact set $K \subset \Omega$ is
\[
  \capa(K, \Omega)=\inf \Bigl\{ \int_{\Omega} \abs{\nabla \varphi}^2 \st \varphi \in C^\infty_c(\Omega)\text{ and } \varphi = 1 \text{ on $K$} \Bigr\}. 
\]
Let us first recall the following standard capacity estimate which was discovered by H.\thinspace Poincar\'e \cite[p.\thinspace 17--22]{Poincare} and whose first complete proof was given by G.\thinspace Szeg\H o~\cite{Szego1930}.

\begin{proposition}
\label{propositionCapacityArea}
Let $\Omega \subset \R^2$ have finite measure. For every $K \subset \Omega$,
\[
  \frac{4\pi}{\capa(K, \Omega)} \le \log \frac{\muleb{2}(\Omega)}{\muleb{2}(K)}. 
\]
\end{proposition}
\begin{proof}
One shows by the P\'olya--Szeg\H o inequality (for a modern treatment, see e.g. \cite{Kawohl1985}, \cite{LiebLoss2001} or \cite{BrockSolynin2000}) that 
\[
 \capa(K, \Omega) \ge \capa (\overline{B(0, \rho)}, B(0,R)
\]
if $\rho$ and $R$ are chosen so that $\muleb{2}(B(0,\rho))=\muleb{2}(K)$ and $\muleb{2}(B(0,R))=\muleb{2}(\Omega)$. One can then compute explicitly the right-hand-side to reach the conclusion.
\end{proof}

When $\mathcal{L}^2(\Omega)=+\infty$, Proposition~\ref{propositionCapacityArea} loses its interest. However, one still has:

\begin{proposition}
\label{propositionCapacityLocalArea}
Let $K \subset \R^2_+$, we have
\[
  \frac{4\pi}{\capa(K, \Omega)} \le \log \frac{8\pi \sup_{x \in K} \abs{x}^2}{\muleb{2}(K)}. 
\]
\end{proposition}
\begin{proof}
Set $a=\sup_{x \in K} \abs{x}^2=1$ and define the conformal transformation
\[  
  \psi(z)= \frac{z-a}{z+a}. 
\]
We have $\psi(\R^2_+)=B(0, 1)$. By the previous Lemma, we have
\[
  \frac{4\pi}{\capa(\psi(K), B(0, 1))} \le \log \frac{2\pi}{\muleb{2}(\psi(K))}. 
\]
The conclusion comes from
\[
  \muleb{2}(\psi(K))=\int_K \abs{\psi'}^2\ge \frac{\muleb{2}(K)}{4a^2}. \qedhere
\]
\end{proof}

Another question about estimates of the capacity is whether one can estimate the diameter of $K$, instead of its area, by its capacity. This is possible if one assumes moreover that $K$ is connected. L.\thinspace E.\thinspace Fraenkel \cite{Fraenkel1981} has obtained in this direction the inequality
\[
 \frac{2\pi}{\capa(K, \Omega)}\le \log C \frac{\diam K}{\sqrt{\mathcal{L}^2(\Omega)}}.
\]
We improve this estimate so that it holds on unbounded sets and it takes into account the distance from the boundary. 

\begin{proposition}
\label{propositionBoundDiameter}
Let $\Omega$ be such that $\R^2\setminus \Omega$ is connected and contains a ball of radius $\rho$ and $K \subset \Omega$ be compact. Then,
\[
  \frac{2\pi}{\capa(K, \Omega)}\le \log 16\Bigl(1+ \frac{\dist(K, \partial \Omega)}{2\rho}\Bigr)\Bigl(1+ \frac{2\dist(K, \partial \Omega)}{\diam(K)}\Bigr).
\]
\end{proposition}
\begin{proof}
Since $K$ is compact, up to translations and rotations we can assume that $0 \in K$ and $\dist(K, \partial \Omega)=\dist(0, \partial \Omega)$. Let $A^*$ and $\Omega^*$ be the sets obtained by circular symmetrization around $0$ introduced by V.\thinspace Wolontis \cite[III.1]{Wolontis1952} (see also J.\thinspace Sarvas \cite{Sarvas1972}). We have
\begin{gather*}
  \capa(A^*, \Omega^*) \le \capa(A, \Omega),\\
  [-\diam(A)/2, 0] \subset A^*,
\end{gather*}
and, since $\R^2 \setminus \Omega^*$ contains a ball of radius $\rho$, 
\[
  [\dist(A, \partial \Omega), \dist(A, \partial \Omega)+2\rho] \subset \R^2\setminus \Omega^*. 
\]
We have thus
\[
  \capa(A, \Omega) \ge \capa([-\diam(A)/2, 0], \R^2 \setminus [\dist(A, \partial \Omega), \dist(A, \partial \Omega)+2\rho]. 
\]
Now, identifying $\R^2$ with $\C$, there exists a M\"obius transformations that brings the points 
$-\diam(A)/2$, $0$, $\dist(A, \partial \Omega)$ and $\dist(A, \partial \Omega)+2\rho$ to $-1$, $0$, $s$ and $\infty$
with
\[
  s=\frac{(2\rho+\dist(K, \partial \Omega)+\frac{1}{2}\diam(K))\dist(K, \partial \Omega)}{\rho \diam(K)},
\]
from which we deduce that 
\[
 \capa(A, \Omega) \ge \capa ([-1, 0], \C \setminus [s, +\infty[).
\]
The conclusion comes from the next lemma.
\end{proof}

As in L.\thinspace E.\thinspace Fraenkel's proof \cite{Fraenkel1981}, we use

\begin{lemma}
Let $s>0$. We have
\[
  \frac{2\pi}{\capa([-1, 0], \R^2 \setminus [s, \infty))}\le \log 16(1+s).
\]
\end{lemma}
\begin{proof}
We have the formula \cite[5.60 (1)]{Vuorinen1988}
\[
  \capa([-1, 0], \R^2 \setminus [s, \infty))=2 \frac {\mathcal{K}(\sqrt{1/(1+s)})}{\mathcal{K}(\sqrt{s/(1+s)})},
\]
where $\mathcal{K}$ is the complete elliptic integral of the first kind
\[
  \mathcal{K}(\gamma)=\int_0^{\frac{\pi}{2}} \frac{1}{\sqrt{1-\gamma^2 (\sin \theta)^2}}\,d\theta. 
\]
Since (see \cite{AndersonVamanamurthyVuorinen1997})
\[
  \frac{\mathcal{K}(\gamma)}{\mathcal{K}(\sqrt{1-\gamma^2})} > \frac{\pi}{2\log\Bigl(2
\dfrac{1+\sqrt{1-\gamma^2}}{\gamma}\Bigr) }
\]
We have then
\[
  \capa([-1, 0], \R^2 \setminus [s, \infty))> \frac{\pi}{\log 2(\sqrt{s}+\sqrt{1+s})}> \frac{\pi}{\log 4\sqrt{1+s}}=\frac{2\pi}{\log 16(1+s)}.\qedhere
\]
\end{proof}


We also have an estimate in the case where the inner radius $\rho$ of $\R^2 \setminus \Omega$ is replaced by the connectedness and the measure of $\R^2 \setminus \Omega$.

\begin{proposition}
\label{propositionCapacityBoundDistance}
Let $\Omega$ be such that $\R^2\setminus \Omega$ is connected and has finite measure and $K \subset \Omega$ be compact. 
We have
\[
  \frac{2\pi}{\capa(K, \Omega)}\le  \log 16\Bigl(1+ \frac{\pi \dist(K, \partial \Omega)\diam(\R^2 \setminus \Omega)}{2\muleb{2}(\R^2 \setminus \Omega)}\Bigr)\Bigl(1+ \frac{2\dist(K, \partial \Omega)}{\diam(K)}\Bigr)
\]
\end{proposition}
\begin{proof}
One begins as in the proof of the previous proposition. 
We have then that 
\[
  [\dist(K, \partial \Omega), \dist(K, \partial \Omega)+\frac{2\muleb{2}(\R^2 \setminus \Omega)}{\pi \diam (\R^2 \setminus \Omega)}] \subset \R^2\setminus \Omega^*. 
\]
And one continues as previously. 
\end{proof}

\section{Green function asymptotics}

This appendix is devoted to the study of the asymptotic expansion of Green's function near a point of the boundary:

\begin{proposition}
\label{propositionAsymptotH}
Let $\Omega \subset \R^2$ and assume that $\Omega$ is of class $C^2$ around $0$ and that the tangent to $\partial \Omega$ is perpendicular to $x_1$. One has then as $\eps \to 0$, 
\[
  G (\eps x, \eps y)=\frac{1}{4\pi} \log \frac{\abs{x-y}^2+4x_1y_1}{\abs{x-y}^2}-
\eps \frac{K}{2\pi} \frac{x_1  \abs{y}^2+y_1 \abs{x}^2}{\abs{x-y}^2+4x_1y_1}+o(\eps). 
\]
uniformly on compact subsets of $\R^2_+ \times \R^2_+$,
where $K$ is the curvature of $\partial \Omega$ at $0$.
In particular,
\[
 H (\eps x, \eps x)=\frac{1}{2\pi} \log 2\eps x_1-\eps \frac{K\abs{x}^2}{4\pi x_1}  +o(\eps). 
\]
\end{proposition}

\begin{proof}
Define
\[
 w_{\eps, y}(x) = \frac{1}{\eps} \Bigl(\frac{1}{4\pi} \log \frac{\abs{x-y}^2+4x_1y_1}{\abs{x-y}^2}-G(\eps x, \eps y)\Bigr). 
\]
This function is defined for every $x, y \in \Omega^\eps=\{ z \in \R^2 \st \eps z \in \Omega \}$. 
Moreover, $w_{\eps, y}$ satisfies
\[
\left\{\begin{aligned}
 -\Delta w_{\eps, y} &= 0 &&\text{in $\Omega$},\\
  w_{\eps, y} &= \frac{1}{4\pi\eps} \log \frac{\abs{x-y}^2+4x_1y_1}{\abs{x-y}^2} && \text{on $\partial \Omega$}. 
\end{aligned}\right. 
\]

By construction, $w_{\eps, y}$ is a bounded function. 
We first claim that $w_{\eps, y}$ is bounded uniformly in $L^\infty(\Omega_\eps)$ as $\eps \to 0$ and $y$ stays in a compact subset of $\R^2$. Indeed, since $\Omega$ is $C^2$ around $0$, there exists $r > 0$ such that if $z \in \partial \Omega \cap B(0, r)$, $\abs{z_1} \le C \abs{z_2}^2$. One has thus, for $x \in \partial \Omega_\eps \cap B(0, \frac{r}{\eps})$, $\abs{x_1} \le C \eps \abs{x_2}^2$, and therefore, when $\eps$ is small enough
\[
 \abs{w_{\eps, y}(x)} \le \frac{C'}{\eps} \frac{\eps y_1 \abs{x_2}^2}{\abs{x-y}^2}
\]
On the other hand, if $x \in \partial \Omega_\eps \setminus B(0, \frac{r}{\eps})$, then if $\eps$ is small enough, $x \in \partial \Omega_\eps \cap B(0, \frac{r}{2\eps})$ so that $x_1 \le 2 \abs{x-y}$ and $\abs{x-y} \ge \frac{r}{2\eps}$, and
\[
 \abs{w_{\eps, y}}(x) \le C \eps. 
\]

Since, $\Omega$ is of class $C^2$, there exists a function $f : I \subset \R  \to \R$ such that $\partial \Omega \cap B(0, r')= \{(f(t), t) \in \Omega \st t \in I \}$. One has thus, using the Taylor expansion of $f$ and recalling that $f(0)=0$ and $f'(0)=0$,
\[
 w_{\eps, y}(x)=\frac{1}{4\pi \eps} \log \Bigl( 1+ \frac{4 y_1 \eps^{-1}f(\eps x_2)}{( \eps^{-1}f(\eps x_2)-y_1)^2+(x_2-y_2)^2}\Bigr). 
\]

Therefore, by classical regularity estimates, $w_{\eps, y}$ converges uniformly with respect to compact subsets of $\R^2_+ \times \R^2_+$ to the unique bounded solution of 
\[
 \left\{
\begin{aligned}
 -\Delta w_y &= 0 &&\text{in $\R^2_+$}, \\
 w_y&=\frac{f''(0)}{2\pi}\frac{y_1 x_2^2}{y_1^2+(x_2-y_2)^2}&&\text{on $\partial \R^2_+$}. 
\end{aligned}
\right. 
\]
One can check that 
\[
 w_y(x)=\frac{f''(0)}{2\pi}\frac{y_1 (x_1^2+x_2^2)+x_1(y_1^2+y_2^2)}{(x_1+y_1)^2+(x_2-y_2)^2}.
\]
The announced expressions for $G(\eps x, \eps y)$ and $H(\eps x, \eps x)$ follow. 
\end{proof}



\end{document}